\documentclass{article}
\usepackage[a4paper,top=2cm,bottom=2cm,left=2cm,right=2cm]{geometry}

\usepackage{amsmath}
\usepackage{amssymb}
\usepackage{amsthm}
\usepackage{mathrsfs}
\usepackage{tikz-cd}
\usepackage{stmaryrd}
\usepackage{hyperref}
\usepackage{authblk}
\usepackage{t-angles}
\usepackage[all]{xy}
\usepackage{comment}
\usepackage{mathtools}
\usepackage[toc,page]{appendix}
\usepackage{enumitem}
\usepackage[normalem]{ulem}
\usepackage[nottoc,numbib]{tocbibind}
\makeatletter
\def\namedlabel#1#2{\begingroup
    #2%
    \def\@currentlabel{#2}%
    \phantomsection\label{#1}\endgroup}
\makeatother

\usepackage[T1]{fontenc}
\usepackage[utf8]{inputenc}
\usepackage{lmodern}

\newtheorem{proposition}{Proposition}[section]
\newtheorem{definition}[proposition]{Definition}
\newtheorem{lemma}[proposition]{Lemma}
\newtheorem{theorem}[proposition]{Theorem}
\newtheorem{remark}[proposition]{Remark}
\newtheorem{example}[proposition]{Example}
\newtheorem{corollary}[proposition]{Corollary}

\allowdisplaybreaks
\newcommand{\mm}{\mathfrak{M}}

\newcommand{\ot}{\otimes}

\title{{\bf Noncommutative differential geometry on crossed product algebras}}
\author{{\large Andrea Sciandra}\footnote{andrea.sciandra@unito.it}~$^{,1}$}
\author{{\large Thomas Weber}\footnote{thomasmartin.weber@unibo.it}~$^{,2,3}$}
\affil{
\centerline{\sl
{$^1$Dipartimento di Matematica "Giuseppe Peano", Universit\`a degli Studi di Torino}}

\centerline{\sl  Via Carlo Alberto 10, 10123 Torino, Italy}

\vskip 0.4 cm

\centerline{\sl
{$^2$Dipartimento di Matematica, Universit\`a di Bologna}}

\centerline{\sl  Piazza di Porta S. Donato 5, I-40126 Bologna, Italy}

\vskip 0.4 cm

\centerline{$^3${\sl Istituto Nazionale di Fisica Nucleare, Sezione di Bologna}}

\centerline{\sl  Via Berti Pichat 6/2, I-40127 Bologna, Italy}
}

\date{}

\begin{document}

\maketitle

\begin{abstract}
We provide a differential structure on arbitrary cleft extensions $B:=A^{\mathrm{co}H}\subseteq A$ for an $H$-comodule algebra $A$. This is achieved by constructing a covariant calculus on the corresponding crossed product algebra $B\#_\sigma H$ from the data of a bicovariant calculus on the structure Hopf algebra $H$ and a calculus on the base algebra $B$, which is compatible with the $2$-cocycle and measure of the crossed product. The result is a quantum principal bundle with canonical strong connection and we describe the induced bimodule covariant derivatives on associated bundles of the crossed product. All results specialize to trivial extensions and smash product algebras $B\#H$ and we give a characterization of the smash product calculus in terms of the differentials of the cleaving map $j\colon H\to A$ and the inclusion $B\hookrightarrow A$. The construction is exemplified for pointed Hopf algebras. In particular, the case of Radford Hopf algebras $H_{(r,n,q)}$ is spelled out in detail.
\end{abstract}

\tableofcontents

\section{Introduction}

Without doubt, principal bundles play a decisive role in differential geometry, topology and gauge theory. An ingenious generalization to quantum principal bundles on noncommutative algebras was proposed by Schneider \cite{Schneider} and Brzezi\'nski--Majid \cite{BrMa}. The central idea is to employ Hopf--Galois extensions \cite{KreTak}, where an $H$-comodule algebra $A$ encodes the total space, the Hopf algebra $H$ replaces the structure group and the base is given by the subalgebra $B=A^{\mathrm{co}H}\subseteq A$ of $H$-coinvariant elements. Principality of such a quantum bundle is then characterized by the invertibility of the Hopf--Galois map $A\otimes_BA\to A\otimes H$. In order to ensure an appropriate geometric behaviour one commonly assumes faithful flatness of the Hopf--Galois extension, thus resulting in a categorical equivalence of $B$-modules and $A$-$H$-Hopf modules \cite{Schneider}. Throughout the last decades the notion of quantum principal bundle has been featured in numerous articles \cite{BrZi,Hajac96,Krutov}, just to name a few, establishing it as common knowledge in noncommutative differential geometry.

A class of particular interest is given by cleft extensions, i.e., $H$-comodule algebras $A$ which admit a cleaving map $j\colon H\to A$. The latter is a convolution invertible comodule map and it provides an explicit inverse of the Hopf--Galois map. It was shown by Doi--Takeuchi \cite{DoiTak} that cleft extensions are precisely the Hopf--Galois extensions with the normal basis property, meaning that $A$ is isomorphic (as comodule algebra) to the crossed product algebra $B\#_\sigma H$ with $B=A^{\mathrm{co}H}$ and, conversely, every crossed product algebra gives rise to a cleft extension. A crossed product algebra is determined by a $2$-cocycle $\sigma\colon H\otimes H\to B$ and a measure (this is a weak version of an action) of $H$ on a $\sigma$-twisted $H$-module algebra $B$. In case of the trivial $2$-cocycle we recover smash product algebras $B\#H$, which correspond to trivial extensions $B=A^{\mathrm{co}H}\subseteq A$. The isomorphism $A\cong B\#_\sigma H$ provides a certain triviality of the bundle, in the sense that the total space algebra can be expressed through the base algebra and structure Hopf algebra, however, we would like to stress that the algebra structure of $B\#_\sigma H$ is nontrivial, as it involves the measure and $2$-cocycle. It follows that cleft extensions are automatically faithfully flat if the antipode of the Hopf algebra is invertible and one considers vector spaces \cite{Hajac}.

While in classical differential geometry there is a canonical differential structure on every manifold, namely the de Rham calculus, in noncommutative geometry there are typically many differential calculi on a given noncommutative algebra. For this reason, a noncommutative differential structure requires additional information, for example a spectral triple in the framework of noncommutative geometry \cite{Connes}, or an ideal in the kernel of the counit in case of a covariant calculus on a Hopf algebra \cite{Wo}. The main contribution of this paper is to provide a canonical covariant calculus on arbitrary cleft extensions $B=A^{\mathrm{co}H}\subseteq A$ (or, equivalently, on arbitrary crossed product algebras $B\#_\sigma H$) from a bicovariant calculus $(\Omega^1(H),\mathrm{d}_H)$ on the structure Hopf algebra $H$ and a calculus $(\Omega^1(B),\mathrm{d}_B)$ on the base algebra $B$, which is compatible with the $2$-cocycle and measure of the crossed product. We call it the crossed product calculus. The construction (Section \ref{Sec3.1}) is in the spirit of the normal basis property, where the $1$-forms $\Omega^1(A)$ on $A$ are build as the direct sum $(\Omega^1(B)\otimes H)\oplus(B\otimes\Omega^1(H))$ and the $A$-bimodule actions mimic the crossed product multiplication; the differential is the sum $\mathrm{d}_A=\mathrm{d}_B+\mathrm{d}_H$. In a similar way, higher order forms are constructed (Section \ref{Sec3.3}) and we prove that the crossed product calculus forms a quantum principal bundle (Section \ref{Sec3.4}) in the sense of \cite{BrMa}, meaning that there is a short exact sequence
$$
0\longrightarrow\Omega^1(B)\otimes H\hookrightarrow\Omega^1(A)\overset{\mathrm{ver}}\longrightarrow A\otimes{}^{\mathrm{co}H}\Omega^1(H)\longrightarrow0,
$$
where $\mathrm{ver}$ denotes the vertical map and ${}^{\mathrm{co}H}\Omega^1(H)$ are the left $H$-coinvariant $1$-forms on $H$. This allows us to identify $\Omega^1(B)\otimes H$ as the horizontal forms and $B\otimes\Omega^1(H)$ as the vertical forms of the quantum principal bundle. It follows that the quantum principal bundle is strong (i.e., $A\Omega^1(B)=\Omega^1(B)A$) and regular in the sense of \cite{BM} (it descends from the universal quantum principal bundle with strong connection). We describe the induced bimodule covariant derivatives on all associated bundles. As an application, we construct higher order vertical maps and exact sequences for the crossed product calculus and show that the noncommutative de Rham cohomology $\mathrm{H}_\mathrm{dR}^\bullet(A)$ can be computed from the tensor product of cohomologies $\mathrm{H}_\mathrm{dR}^\bullet(B)\otimes\mathrm{H}_\mathrm{dR}^\bullet(H)$ via the Leray--Serre spectral sequence. We further give a bijection between connections and connection $1$-forms on the crossed product calculus (Section \ref{Sec3.5}), using the notion of fundamental vector field.

For the convenience of the reader we include preliminaries (Section \ref{Sec2}), recalling the notions of Hopf algebra theory and noncommutative differential geometry which are relevant for the crossed product calculus, rendering the presentation self-contained. To further justify the importance of the crossed product calculus we conclude the paper with explicit examples (Section \ref{Sec4}) arising from pointed Hopf algebras, particularly from the Radford Hopf algebras $H_{(r,n,q)}$. We would like to stress that the resulting calculus on a pointed Hopf algebra $H$ is covariant with respect to the group algebra of $G/N$, where $G$ and $N$ are the groups of group-like elements of $H$ and of the indecomposable component of $H$ containing the unit, respectively. In particular, it is not covered by the Woronowicz classification \cite{Wo}.

It is important to note that part of this paper is inspired by the work \cite{PS} of Pflaum--Schauenburg. They construct a canonical covariant calculus on trivial extensions $B=A^{\mathrm{co}H}\subseteq A$ or, equivalently, on smash product algebras $B\#H$. The results of this paper generalize the findings in \cite{PS}, in the sense that we consider crossed product algebras $B\#_\sigma H$, with arbitrary $2$-cocycles $\sigma\colon H\otimes H\to B$. The generalization is not straightforward, since, for example, the necessary compatibility of $\sigma$ with the differential of $B$ which we provide is automatic in case of the trivial $2$-cocycle. We further give a classification of the smash product calculus in Section \ref{Sec:ClasSmash}, in terms of the differentials of the cleaving and inclusion maps. There are interesting classes of cleft extensions which are not trivial (see Example \ref{ex:torus} and Section \ref{Sec4}), in particular in the context of pointed Hopf algebras, which further justify the relevance of this generalization. Let us also point out the reference \cite{Ryan}, where differential calculi on double smash products are considered. A generalization to double crossed products might be feasible and we leave this as a future project. To avoid confusion let us clarify that the authors of \cite{Ryan} use the name cross product instead of smash product,
while we follow the convention of \cite{Mon}, where crossed products allow for non-trivial $2$-cocycles and measures.

\section{Preliminaries}\label{Sec2}
In this section we recall some basic notions and results concerning (co)action compatible modules, crossed product algebras and cleft extensions, noncommutative differential calculi, as well as quantum principal bundles and their connections.
The reader who is familiar with these concepts can proceed to Section \ref{Sec3}, which contains the first original contribution of this paper.

We fix a field $\Bbbk$ and all vector spaces are understood to be $\Bbbk$-vector spaces unless otherwise specified. By a linear map we mean a $\Bbbk$-linear map and the unadorned tensor product $\ot$ is the one of vector spaces. All linear maps whose domain is a tensor product will usually be defined on generators and understood to be extended by linearity. Algebras over $\Bbbk$ will be associative and unital. Our general references for Hopf algebra theory are \cite{Sw} and \cite{Mon}, where one can find the notions of coalgebra, bialgebra, Hopf algebra and (co)modules. We denote the multiplication and the unit of an algebra by $m$ and $u$, while the comultiplication and the counit of a coalgebra are denoted by $\Delta$ and $\varepsilon$; $S$ will denote the antipode of a Hopf algebra, which we always assume invertible in the following. 
Given a coalgebra $C$ we will use Sweedler's notation for calculations involving the comultiplication, meaning we shall write $\Delta(c)=c_{1}\otimes c_{2}$ for any $c\in C$ (with the usual summation convention, where $c_{1}\otimes c_{2}$ stands for the finite sum $\sum_i{c^i_{1}\otimes c^i_{2}}$). 
Given a right (resp. left) $C$-comodule $M$ we use Sweedler’s like notation for calculations involving the coaction $\rho\colon M\to M\otimes C$ (resp. $\lambda\colon M\to C\otimes M$), namely we omit finite summation and write $\rho(m)=m_{0}\otimes m_{1}$ (resp. $\lambda(m)=m_{-1}\otimes m_{0}$), for any $m\in M$.

\subsubsection*{Equivariant and covariant modules}

For $H$ a bialgebra it is well-known that the category of left $H$-modules $_{H}\mm$ is a monoidal category:
for every $M$ and $N$ in $_{H}\mm$ the tensor product $M\otimes N$ has a left $H$-action given by $h\cdot(m\otimes n)=h_{1}\cdot m\otimes h_{2}\cdot n$ for every $h\in H$, $m\in M$ and $n\in N$, where we denote the actions on $M$ and $N$ by the same symbol. The unit object is the base field $\Bbbk$ with left $H$-action given by $h\cdot k=\varepsilon(h)k$ for every $h\in H$ and $k\in\Bbbk$. Similarly, one defines the monoidal categories $\mm_{H}$ of right $H$-modules and $_{H}\mm_{H}$ of $H$-bimodules.
Dually, the category of right $H$-comodules $\mm^{H}$ is monoidal: for every $M$ and $N$ in $\mm^{H}$ the tensor product $M\otimes N$ has a right $H$-coaction given by $\rho_{M\otimes N}(m\otimes n)=m_{0}\otimes n_{0}\otimes m_{1}n_{1}$ for every $m\in M$ and $n\in N$. The unit object is $\Bbbk$ with right $H$-coaction $\rho_{\Bbbk}(k)=k\otimes1_{H}$ for every $k\in\Bbbk$. Similarly, one defines the monoidal categories $^{H}\mm$ of left $H$-comodules and $^{H}\mm^{H}$ of $H$-bicomodules. \medskip

A \textit{left $H$-module algebra} is an algebra object in the category $_{H}\mm$. Explicitly, an algebra $(A,m,u)$ is a left $H$-module algebra if $A$ is a left $H$-module such that $m:A\otimes A\to A$ and $u:\Bbbk\to A$ are left $H$-linear maps, i.e., for all $h\in H$ and $a,b\in A$ we have
\[
h\cdot(ab)=(h_{1}\cdot a)(h_{2}\cdot b)\qquad \text{and}\qquad h\cdot1_{A}=\varepsilon(h)1_{A}.
\]
Morphisms of left $H$-module algebras are left $H$-linear algebra morphisms. Similarly, right $H$-module algebras and $H$-bimodule algebras are defined.

A \textit{right $H$-comodule algebra} is an algebra object in the category $\mm^{H}$. Explicitly, an algebra $(A,m,u)$ is a right $H$-comodule algebra if $A$ is a right $H$-comodule such that $m:A\otimes A\to A$ and $u:\Bbbk\to A$ are right $H$-colinear maps, i.e., for all $a,b\in A$ we have
\[
\rho(ab)=a_{0}b_{0}\otimes a_{1}b_{1}\qquad \text{and}\qquad \rho(1_{A})=1_{A}\otimes1_{H}.
\]
Morphisms of right $H$-comodule algebras are right $H$-colinear algebra morphisms.
Similarly, left $H$-comodule algebras and $H$-bicomodule algebras are defined.

Given a left $H$-module algebra $A$, we can consider the category of left (right, bi) $A$-modules in $_{H}\mm$, denoted by $_{A}(_{H}\mm)$ ($(_{H}\mm)_{A}$, $_{A}(_{H}\mm)_{A}$).
We refer to an object $M$ in the category $_{A}(_{H}\mm)$ as a \textit{left $H$-equivariant left $A$-module}. Explicitly, $M$ is a left $A$-module $M$ with action $A\otimes M\ni a\otimes m\mapsto am\in M$ which is also a left $H$-module with action $H\otimes M\ni h\otimes m\mapsto h\cdot m\in M$ such that
\begin{equation*}%\label{comp.actions}
    h\cdot(am)=(h_{1}\cdot a)(h_{2}\cdot m)\qquad \text{for all}\ h\in H,\ a\in A\ \text{and}\ m\in M.
\end{equation*}
A morphism of left $H$-equivariant left $A$-modules is a left $H$-linear and left $A$-linear map. Similarly, left $H$-equivariant right $A$-modules and $A$-bimodules are defined.

Analogously, if $A$ is a right $H$-comodule algebra, we can consider the category of left (right, bi) $A$-modules in the category $\mm^{H}$, denoted by $_{A}(\mm^{H})$ ($(\mm^{H})_{A}$, $_{A}(\mm^{H})_{A}$), even removing the parentheses.
An object $M$ in $_{A}(\mm^{H})$ is called a \textit{right $H$-covariant left $A$-module}. Explicitly, $M$ is a left $A$-module with action $A\otimes M\ni a\otimes m\mapsto am\in M$ which is also a right $H$-comodule with coaction $\rho:M\to M\otimes H$ such that
\begin{equation*}
\rho(am)=a_{0}m_{0}\otimes a_{1}m_{1}\qquad \text{for all}\ a\in A\ \text{and}\ m\in M.
\end{equation*}
Morphisms of right $H$-covariant left $A$-modules are right $H$-colinear and left $A$-linear maps.  Similarly, right $H$-covariant right $A$-modules and $A$-bimodules are defined.

For more information about equivariant and covariant modules we refer the reader to \cite{SchYD}.

\subsubsection*{Crossed product algebras}

In this subsection we present crossed product algebras and their correspondence with cleft extensions, see \cite{DoiTak} or \cite[Chapter 7]{Mon}. Let $H$ be a Hopf algebra and $B$ an algebra. We say that $H$ \textit{measures} $B$ if there is a linear map $H\otimes B\to B$, $h\otimes b\mapsto h\cdot b$, such that $h\cdot1_{B}=\varepsilon(h)1_{B}$ and $h\cdot(bb')= (h_{1}\cdot b)(h_{2}\cdot b')$ for all $h\in H$ and $b,b'\in B$. Let further $\sigma:H\otimes H\to B$ be a linear map which is convolution invertible, i.e., such that there exists a linear map $\sigma^{-1}:H\otimes H\to B$ such that $\sigma(h_{1}\otimes h'_{1})\sigma^{-1}(h_{2}\otimes h'_{2})=\varepsilon(hh')1_{B}=\sigma^{-1}(h_{1}\otimes h'_{1})\sigma(h_{2}\otimes h'_{2})$ for every $h,h'\in H$. Then the \textit{crossed product} $B\#_{\sigma}H$ of $B$ with $H$ is the vector space $B\otimes H$ endowed with the linear map
\begin{equation}\label{crossed}
    (B\otimes H)\otimes(B\otimes H)\to B\otimes H,\qquad (b\otimes h)\otimes(b'\otimes h')\mapsto b(h_{1}\cdot b')\sigma(h_{2}\otimes h'_{1})\otimes h_{3}h'_{2}
\end{equation}
for all $b,b'\in B$ and $h,h'\in H$.
It is well-known (see e.g. \cite[Lemma 7.1.2]{Mon}) that the crossed product $B\#_{\sigma}H$ is an associative algebra with multiplication given by the assignment \eqref{crossed} and unit $1_{B}\otimes1_{H}$ if and only if the following two conditions are satisfied: 
\begin{enumerate}
\item[i.)] $B$ is a \textit{$\sigma$-twisted left $H$-module}, i.e., $1_{H}\cdot b=b$ for all $b\in B$ and 
\begin{equation*}%\label{Twistedmodule}
    h\cdot(h'\cdot b)=\sigma(h_{1}\otimes h'_{1})(h_{2}h'_{2}\cdot b)\sigma^{-1}(h_{3}\otimes h'_{3})
\end{equation*}
for all $h,h'\in H$ and $b\in B$.

\item[ii.)] $\sigma$ is a \textit{2-cocycle with values in $B$}, that is, $\sigma(h\otimes1_{H})=\varepsilon(h)1_{B}=\sigma(1_{H}\otimes h)$ for all $h\in H$ and 
 \begin{equation*}%\label{cocycle}
(h_{1}\cdot\sigma(h'_{1}\otimes h''_{1}))\sigma(h_{2}\otimes h'_{2}h''_{2})=\sigma(h_{1}\otimes h'_{1})\sigma(h_{2}h'_{2}\otimes h'')
\end{equation*}
for all $h,h',h''\in H$.
\end{enumerate}
We refer to an algebra $B$ satisfying i.) and ii.) as a \textit{$\sigma$-twisted left $H$-module algebra}. 
It is straightforward to prove that, in case the above assumptions are satisfied, $B\#_\sigma H$ is a right $H$-comodule algebra with right $H$-coaction $\rho:=\mathrm{Id}_{B}\otimes\Delta\colon B\#_\sigma H\to(B\#_\sigma H)\otimes H$. Note that $(B\#_{\sigma}H)^{\mathrm{co}H}=B\otimes\Bbbk1_{H}\cong B$. In the following we refer to $B\#_\sigma H$ with the associative product \eqref{crossed} as the \textit{crossed product algebra}. The product will be denoted by
$$
(b\otimes h)\cdot(b'\otimes h')=b(h_{1}\cdot b')\sigma(h_{2}\otimes h'_{1})\otimes h_{3}h'_{2}
$$
for $b\otimes h,b'\otimes h'\in B\#_\sigma H$.

A particular class of crossed product algebras is obtained from left $H$-module algebras $B$ with trivial 2-cocycle $h\otimes h'\mapsto\varepsilon(hh')1_{B}$. The corresponding product is $(b\otimes h)\cdot(b'\otimes h')=b(h_{1}\cdot b')\otimes h_{2}h'$ and one obtains $B\#H$, the \textit{smash product algebra} of $B$ and $H$.

\medskip

We recall that crossed product algebras classify cleft extensions, with smash product algebras corresponding to trivial extensions.
Let $A$ be a right $H$-comodule algebra and consider $B:=A^{\mathrm{co}H}$ the subalgebra of right $H$-coinvariant elements.  
We call $B\subseteq A$ a \textit{cleft extension} if there is a convolution invertible morphism $j:H\to A$ of right $H$-comodules, the so-called \textit{cleaving map}. If $j:H\to A$ is a morphism of right $H$-comodule algebras we call $B\subseteq A$ a \textit{trivial extension}.
By definition, every trivial extension is a cleft extension. The convolution inverse of a cleaving map $j\colon H\to A$ is commonly denoted by $j^{-1}\colon H\to A$. The cleaving map is injective and we can assume it to be unital without loss of generality.

Any crossed product algebra $B\#_{\sigma}H$ is a cleft extension $B\subseteq B\#_{\sigma}H$ with cleaving map $j:H\to B\#_{\sigma}H,\ h\mapsto1_{B}\otimes h$. Note that $j^{-1}:h\mapsto\sigma^{-1}(S(h_{2})\otimes h_{3})\otimes S(h_{1})$. Conversely, given a cleft extension $B\subseteq A$ with cleaving map $j:H\to A$, define a measure on $B$
\[
\cdot:H\otimes B\to B,\ h\otimes b\mapsto h\cdot b:=j(h_{1})bj^{-1}(h_{2}) 
\]
and a 2-cocycle with values in $B$
\[
\sigma:H\otimes H\to B,\ h\otimes h'\mapsto\sigma(h\otimes h'):=j(h_{1})j(h'_{1})j^{-1}(h_{2}h'_{2}). 
\]
Then $A\cong B\#_{\sigma}H$ are isomorphic as right $H$-comodule algebras via $\theta:a\mapsto a_{0}j^{-1}(a_{1})\otimes a_{2}$, with inverse $\theta^{-1}:b\otimes h\mapsto bj(h)$. A cleft extension is a trivial extension if and only if the corresponding crossed product algebra is a smash product algebra.
We refer to \cite[Theorem 7.2.2]{Mon} for a proof of this correspondence (see also \cite[Chapter 6]{Ma95}).

\subsubsection*{Noncommutative differential calculi}

In this subsection we recall differential structures on noncommutative algebras, together with some well-known classification results. The original source of this material is \cite{Wo} and we refer the reader to \cite[Chapter 1 \& 2]{BM}, \cite{SchDC} for more details. 

%In this work $H$ will always denote a bialgebra, while $A$ and $B$ are associative unital algebras.
A map $f\colon V^\bullet\to W^\bullet$ between $\mathbb{N}$-graded vector spaces $V^\bullet=\bigoplus_{n\in\mathbb{N}}V^n$, $W^\bullet=\bigoplus_{n\in\mathbb{N}}W^n$ is of degree $k\in\mathbb{Z}$ if $f(V^n)\subseteq W^{n+k}$ for all $n\in\mathbb{N}$, where we set $W^{-n}=\{0\}$ by default. The tensor product $V^\bullet\otimes W^\bullet$ of $\mathbb{N}$-graded vector spaces $V^\bullet$, $W^\bullet$ becomes an $\mathbb{N}$-graded vector space with $(V^\bullet\otimes W^\bullet)^n:=\bigoplus_{k=0}^nV^k\otimes W^{n-k}$ for all $n\in\mathbb{N}$.
Recall that a \textit{differential graded algebra} $(A^{\bullet},\wedge,\mathrm{d})$ is an $\mathbb{N}$-graded vector space $A^{\bullet}=\bigoplus_{n\in\mathbb{N}}{A^{n}}$, endowed with a degree $0$ map $\wedge\colon A^\bullet\otimes A^\bullet\to A^\bullet$, the \textit{wedge product}, a degree $1$ map $\mathrm{d}\colon A^\bullet\to A^{\bullet+1}$, the \textit{differential}, and an element $1\in A^0$, such that $\wedge$ is associative and unital with respect to $1$, and such that $\mathrm{d}^{2}:=\mathrm{d}\circ\mathrm{d}=0$, as well as the \textit{Leibniz rule}
\begin{equation*}%\label{Leibniz}
    \mathrm{d}(\omega\wedge\eta)=\mathrm{d}\omega\wedge\eta+(-1)^{n}\omega\wedge\mathrm{d}\eta
\end{equation*}
hold for all $\omega\in A^n$ and $\eta\in A^\bullet$.
A morphism $\Phi:(A^{\bullet},\wedge,\mathrm{d})\to(A'^{\bullet},\wedge',\mathrm{d}')$ between differential graded algebras is a linear map $\Phi:A^{\bullet}\to A'^{\bullet}$ of degree 0 such that $\Phi(\omega\wedge\eta)=\Phi(\omega)\wedge'\Phi(\eta)$ and $\Phi\circ\mathrm{d}=\mathrm{d'}\circ\Phi$ for all $\omega,\eta\in A^{\bullet}$. Furthermore, $\Phi$ is an isomorphism if it is invertible. \medskip

A \textit{differential calculus} (DC) on an algebra $A$ is a differential graded algebra $(\Omega^{\bullet}(A),\wedge,\mathrm{d})$ with $\Omega^{0}(A)=A$ such that, for all $n\geq1$, the vector space $\Omega^{n}(A)$ is spanned by elements of the form $a^{0}\mathrm{d}a^{1}\wedge\cdot\cdot\cdot\wedge\mathrm{d}a^{n}$ with $a^{i}\in A$. A morphism $\Phi$ from a DC $(\Omega^{\bullet}(A),\wedge,\mathrm{d})$ on an algebra $A$ to a DC $(\Omega^{\bullet}(A'),\wedge',\mathrm{d}')$ on an algebra $A'$ is a morphism of the underlying differential graded algebras. 

Given $(\Omega^{\bullet}(A),\wedge,\mathrm{d})$ and $(\Omega^{\bullet}(A'),\wedge',\mathrm{d}')$ DCi on algebras $A$ and $A'$, respectively, we call an algebra map $f:A\to A'$ \textit{differentiable} if there exists a morphism $\Omega^{\bullet}(f):\Omega^{\bullet}(A)\to\Omega^{\bullet}(A')$ of differential graded algebras with $\Omega^{0}(f)=f$. Thus, $\Omega^{\bullet}(f):\bigoplus_{n\geq0}{\Omega^{n}(A)}\to\bigoplus_{n\geq0}{\Omega^{n}(A')}$ is an algebra map defined by a collection of linear maps $\Omega^{n}(f):\Omega^{n}(A)\to\Omega^{n}(A')$ such that $\mathrm{d}'\circ\Omega^{n}(f)=\Omega^{n+1}(f)\circ\mathrm{d}$. Observe that $\Omega^{\bullet}(f)$ is completely determined by $f$ since $\Omega^n(f)(a^{0}\mathrm{d}a^{1}\wedge\cdot\cdot\cdot\wedge\mathrm{d}a^{n})=f(a^{0})\mathrm{d}'(f(a^{1}))\wedge\cdot\cdot\cdot\wedge\mathrm{d}'(f(a^{n}))$ for all $a^i\in A$.
\medskip

Let $H$ be a Hopf algebra and
recall that an $\mathbb{N}$-graded vector space $V=\bigoplus_{n\in\mathbb{N}}{V^{n}}$ is a 
\textit{graded right $H$-comodule} if there is a right $H$-coaction $\rho:V\to V\otimes H$ such that $\rho(V^{n})\subseteq V^{n}\otimes H$ for all $n\geq0$. 
Thus, for any $n\geq0$, the vector space $V^{n}$ is a right $H$-comodule via the restriction and corestriction $\rho^{n}:=\rho|_{V^{n}}:V^{n}\to V^{n}\otimes H$ and we can write $\rho=\bigoplus_{n\in\mathbb{N}}{\rho^{n}}$. A linear map $\phi:V\to W$ of degree $k$ between graded right $H$-comodules $(V,\rho_{V})$ and $(W,\rho_{W})$ is said to be right $H$-colinear if $\rho_{W}^{n+k}\circ\phi^{n}=(\phi^{n}\otimes\mathrm{Id}_{H})\circ\rho_{V}^{n}$ for all $n\geq0$, where $\phi^{n}:=\phi|_{V^{n}}:V^{n}\to W^{n+k}$. 
A \textit{graded right $H$-comodule algebra} is
a graded algebra which is also a graded right $H$-comodule and the algebra structure is compatible with the right $H$-coaction in the usual sense. 
A differential calculus $(\Omega^{\bullet}(A),\wedge,\mathrm{d})$ on a right $H$-comodule algebra $(A,\rho_{A})$ is called \textit{right $H$-covariant} if $(\Omega^{\bullet}(A),\wedge)$ is a graded right $H$-comodule algebra with graded coaction $\rho:\Omega^{\bullet}(A)\to\Omega^{\bullet}(A)\otimes H$ such that $\rho^{0}=\rho_{A}$ and the differential $\mathrm{d}:\Omega^{\bullet}(A)\to\Omega^{\bullet}(A)$ is right $H$-colinear.
Similarly, left $H$-covariant and $H$-bicovariant DC are defined on left $H$-comodule algebras and $H$-bicomodule algebras, respectively. In case of the comodule algebra $H$ with coaction given by the coproduct we call an $H$-covariant calculus covariant, i.e., we omit the reference to $H$. We also use the expression $(\Omega^\bullet,\mathrm{d})$ for a DC on an algebra $A$ when the latter is clear from the context.
\medskip

On every algebra $A$ one can construct the universal DC $(\Omega^\bullet_u,\mathrm{d}_u)$ as the common kernel of adjacent tensor product multiplication \cite[Theorem 1.33]{BM}. If $A$ is a right $H$-comodule algebra, then $(\Omega^\bullet_u,\mathrm{d}_u)$ is right $H$-covariant. Moreover, every DC on $A$ is a quotient of $(\Omega^\bullet_u,\mathrm{d}_u)$, see e.g. \cite[Appendix B]{Du}. The latter means that, if $(\Omega^\bullet,\mathrm{d})$ is an arbitrary DC on $A$, there is a surjective morphism $(\Omega^\bullet_u,\mathrm{d}_u)\to(\Omega^\bullet,\mathrm{d})$ of differential graded algebras. Thus, the universal DC on $A$ is an initial object for the category of differential calculi over $A$, while the trivial calculus on $A$ is a terminal object.

Given a DC $(\Omega^\bullet,\mathrm{d})$ on $A$ and a DC $(\Omega'^{\bullet},\mathrm{d}')$ on another algebra $A'$ there is a DC $(\Omega^\bullet_\otimes,\mathrm{d}_\otimes)$ on the tensor product algebra $A\otimes A'$, called the \textit{tensor product calculus}. The graded module $\Omega_\otimes^n:=\bigoplus_{k=0}^n\Omega^k\otimes\Omega'^{n-k}$ becomes a graded algebra with multiplication determined on homogeneous elements $\omega,\eta\in\Omega^\bullet$ and $\omega',\eta'\in\Omega'^{\bullet}$ by 
$$
(\omega\otimes\omega')\wedge_\otimes(\eta\otimes\eta')
:=(-1)^{|\omega'||\eta|}(\omega\wedge\eta)\otimes(\omega'\wedge'\eta').
$$
The differential on $\Omega^\bullet_\otimes$ is given on homogeneous elements by
$$
\mathrm{d}_\otimes(\omega\otimes\omega')
:=\mathrm{d}\omega\otimes\omega'+(-1)^{|\omega|}\omega\otimes\mathrm{d}'\omega'.
$$
This construction implies that DCi form a monoidal category with the monoidal unit given by the trivial calculus on $\Bbbk$.
\medskip

In many situations the knowledge of $(\Omega^1,\mathrm{d}|_A)$ is sufficient to describe the full differential calculus on $A$, e.g. via "maximal prolongation" (see \cite[Lemma 1.32]{BM}) or via a braided exterior algebra construction (see \cite[Theorem 4.1]{Wo}). This motivates the importance of such "first order" DC data, defined as follows: a \textit{first order differential calculus} (FODC) on an algebra $A$ is a tuple $(\Omega^1(A),\mathrm{d}_{A})$, where:
\begin{enumerate}
\item[i.)] $\Omega^1(A)$ is an $A$-bimodule;

\item[ii.)] $\mathrm{d}_{A}:A\to\Omega^1(A)$ is a linear map which satisfies the Leibniz rule $\mathrm{d}_{A}(ab)=(\mathrm{d}_{A}a)b+a\mathrm{d}_{A}b$ for all $a, b\in A$;

\item[iii.)] $\Omega^1(A)=A\mathrm{d}_{A}A:=\mathrm{span}_{\Bbbk}\{a\mathrm{d}_{A}b\ |\ a,b\in A\}$\ (\textit{surjectivity condition}). 
\end{enumerate}
A morphism of FODCi $(\Omega^1(A),\mathrm{d}_{A})\to(\Omega^1(B),\mathrm{d}_{B})$ is a pair $(f,F)$, such that $f:A\to B$ is a morphism of algebras, $F:\Omega^1(A)\to\Omega^1(B)$ is an $A$-bimodule map, where $f$ defines the $A$-bimodule structure on $\Omega^1(B)$, and $F\circ\mathrm{d}_{A}=\mathrm{d}_{B}\circ f$. 
An isomorphism of FODCi $(\Omega^1(A),\mathrm{d}_{A})$ and $(\Omega^1(B),\mathrm{d}_{B})$ is a morphism $(f,F)\colon(\Omega^1(A),\mathrm{d}_{A})\to(\Omega^1(B),\mathrm{d}_{B})$, such that $f$ and $F$ are invertible.

Given $(\Omega^1(A),\mathrm{d}_{A})$ and $(\Omega^1(B),\mathrm{d}_{B})$ FODCi, an algebra map $f:A\to B$ is called \textit{differentiable}, if there exists a map $F:\Omega^1(A)\to\Omega^1(B)$ such that the pair $(f,F)$ is a morphism of FODCi. Observe that in this case $F$ is uniquely determined by $f$ via $F(a\mathrm{d}_{A}b)=f(a)\mathrm{d}_{B}(f(b))$. Thus, we sometimes write $F=\widehat{f}$.
One easily verifies that if $(f,F)\colon(\Omega^1(A),\mathrm{d}_{A})\to(\Omega^1(B),\mathrm{d}_{B})$ is an isomorphism of FODCi the inverse algebra morphism
$f^{-1}:B\to A$ of $f$ is automatically differentiable with $\widehat{f^{-1}}=\widehat{f}^{-1}:\Omega^1(B)\to\Omega^1(A)$, the inverse $A$-module morphism of $\widehat{f}$.

Let us briefly discuss the first order of covariant DC on comodule algebras. % and, dually, FODCi which are invariant under a Hopf algebra action. 
Consider a right (left) $H$-comodule algebra $A$ and a FODC $(\Omega^1(A),\mathrm{d}_{A})$. The latter is said to be \textit{right} (left) $H$-\textit{covariant} if $\Omega^1(A)$ is in $_{A}\mm^{H}_{A}$ (resp. $_{A}^{H}\mm_{A}$) and $\mathrm{d}_{A}$ is in $\mm^{H}$ (resp. $^{H}\mm$). Given an $H$-bicomodule algebra $A$, we say that $(\Omega^1(A),\mathrm{d}_{A})$ is $H$-\textit{bicovariant} if $\Omega^1(A)$ is in $_{A}^{H}\mm^{H}_{A}$ and $\mathrm{d}_{A}$ is in $^{H}\mm^{H}$.
As in the more general case of DC, we usually omit the reference to $H$ in case of the $H$-comodule $H$ itself with coaction given by the coproduct, i.e., we speak of \textit{right/left or bicovariant FODC} in this setting.

We recall from \cite[Theorem 2.26]{BM} the classification theorem due to Woronowicz (for the original source of this result see \cite[Theorem 1.6]{Wo}): every right covariant FODC on $H$ is uniquely determined (up to isomorphism) by a left $H$-ideal $I\subseteq H^+:=\ker\varepsilon$. The corresponding right covariant calculus on $H$ is given by $\Omega^1(H):=(H^{+}/I)\otimes H$ and $\mathrm{d}_{H}h=(\pi\otimes\mathrm{Id})(\Delta(h)-1_{H}\otimes h)$, where the left and right $H$-actions are
\[
h\cdot([g]\otimes h')=[h_{1}g]\otimes h_{2}h',\ ([g]\otimes h')\cdot h=[g]\otimes h'h
\]
for $h,h'\in H$, $g\in H^{+}$ and $\pi\colon H^+\to H^+/I$, $g\mapsto\pi(g):=[g]$ denotes the quotient map. The right $H$-coaction on $\Omega^1(H)$ is $\rho:=\mathrm{Id}_{H^{+}/I}\otimes\Delta$. 
This FODC is bicovariant if and only if $\mathrm{Ad}_{L}(I)\subseteq H\otimes I$, where $\mathrm{Ad}_L$ denotes the \textit{adjoint left coaction} $\mathrm{Ad}_{L}:H\to H\otimes H,\ h\mapsto h_{1}S(h_{3})\otimes h_2$. In this case, the left $H$-coaction on $\Omega^1(H)$ is given by
\[
\lambda([g]\otimes h)=g_{1}S(g_{3})h_{1}\otimes([g_{2}]\otimes h_{2}),
\]
for all $g\in H^{+}$ and $h\in H$.

\subsubsection*{Quantum principal bundles and connections}%\label{Sec2.4}

We recall from \cite{BrMa} the definition of a quantum principal bundle and of a connection on it, see also \cite[Section 5.4]{BM}. In this subsection $H$ denotes a Hopf algebra, $A$ is a right $H$-comodule algebra and we write $B:=A^{\mathrm{co}H}\subseteq A$ for the subalgebra of coinvariants. Given a FODC $(\Omega^1(A),\mathrm{d}_A)$ we always consider the \textit{pullback calculus} on $B$: this is the FODC $(\Omega^1(B),\mathrm{d}_B)$ given by the restrictions $\Omega^1(B):=B\mathrm{d}_{A}B\subseteq\Omega^1(A)$, $\mathrm{d}_B:=\mathrm{d}_{A}|_B\colon B\to\Omega^1(B)$. We further define the \textit{horizontal forms} as the $A$-bimodule generated by $\Omega^1(B)$, i.e., $\Omega^1_\mathrm{hor}(A):=A\Omega^1(B)A\subseteq \Omega^1(A)$. Then, a \textit{quantum principal bundle} (QPB) on $A$ is the data of a right $H$-covariant FODC $(\Omega^1(A),\mathrm{d}_{A})$ on $A$ and a bicovariant FODC $(\Omega^1(H),\mathrm{d}_{H})$ on $H$
such that the \textit{vertical map} 
\[
\mathrm{ver}:\Omega^1(A)\to A\otimes{}^{\mathrm{co}H}\Omega^1(H),\qquad a\mathrm{d}_{A}a'\mapsto aa'_{0}\otimes S(a'_{1})\mathrm{d}_{H}a'_{2} 
\]
is well-defined and the \textit{Atiyah sequence}
\begin{equation}\label{Atiyahseq'}
    0\longrightarrow\Omega^1_{\mathrm{hor}}(A)\overset{i}\longrightarrow\Omega^1(A)\overset{\mathrm{ver}}\longrightarrow A\otimes{}^{\mathrm{co}H}\Omega^1(H)\longrightarrow0
\end{equation}
is exact in $_{A}\mm$, where $i$ is the inclusion.

If the vertical map is well-defined one clearly has $\Omega^1_{\mathrm{hor}}(A)\subseteq\mathrm{ker}(\mathrm{ver})$. Thus, we have a QPB if $\mathrm{ver}$ is well-defined and surjective such that $\mathrm{ker}(\mathrm{ver})\subseteq\Omega^1_{\mathrm{hor}}(A)$.
Let us discuss a criterion for $\mathrm{ver}$ to be well-defined.

Let $A$ be a right $H$-comodule algebra with a FODC $(\Omega^1(A),\mathrm{d}_{A})$ on $A$ and a bicovariant FODC $(\Omega^1(H),\mathrm{d}_{H})$ on $H$.
By \cite[Theorem 2.8]{PS} we have that $\rho:A\to A\otimes H$ is differentiable if and only if $(\Omega^1(A),\mathrm{d}_{A})$ is right $H$-covariant and there is a map $\pi:\Omega^1(A)\to A\square_{H}\Omega^1(H)$ in $_{A}\mm^{H}_{A}$ such that $(\mathrm{Id}_{A}\otimes\mathrm{d}_{H})\circ\rho=\pi\circ\mathrm{d}_{A}$.
Using the isomorphism $\phi:A\square_{H}\Omega^1(H)\to A\otimes{}^{\mathrm{co}H}\Omega^1(H)$, $a\otimes\gamma\mapsto a\otimes S(\gamma_{-1})\gamma_{0}$ we obtain
\[
\phi\circ\pi(a\mathrm{d}_{A}a')=a_{0}a'_{0}\otimes S((a_{1}\mathrm{d}_{H}a'_{1})_{-1})(a_{1}\mathrm{d}_{H}a'_{1})_{0}=a_{0}a'_{0}\otimes S(a_{1}a'_{1})a_{2}\mathrm{d}_{H}a'_{2}=a_{0}a'_{0}\otimes S(a'_{1})\varepsilon(a_{1})\mathrm{d}_{H}a'_{2}=\mathrm{ver}(a\mathrm{d}_{A}a').
\]
Thus, if we assume to have a bicovariant FODC $(\Omega^1(H),\mathrm{d}_{H})$ and a FODC $(\Omega^1(A),\mathrm{d}_{A})$, $\rho$ being differentiable is equivalent to $(\Omega^1(A),\mathrm{d}_{A})$ being right $H$-covariant and $\mathrm{ver}$ being well-defined (see also \cite[page 425]{BM}). If $\rho$ is differentiable then $\pi=\pi_{2}\circ\widehat{\rho}$, where $\pi_{2}:\Omega^1(A\otimes H)\to A\otimes\Omega^1(H)$ is the canonical projection, and so $\mathrm{ver}=\phi\circ\pi_{2}\circ\widehat{\rho}$. Moreover, $\rho':=\pi_{1}\circ\widehat{\rho}$ is the right $H$-coaction of $\Omega^1(A)$, where $\pi_{1}:\Omega^1(A\otimes H)\to\Omega^1(A)\otimes H$ is the other canonical projection.

Note that the inclusion $i:\Omega^1_{\mathrm{hor}}(A)\to\Omega^1(A)$ is a morphism in $_{A}\mm^{H}$ and that $A\otimes{}^{\mathrm{co}H}\Omega^1(H)$ is in $_{A}\mm^{H}$ with left $A$-action given by the multiplication on the first tensor factor and right $H$-coaction given by the diagonal one. 
If $\mathrm{ver}$ is well-defined then it is automatically right $H$-colinear, indeed
\[
\begin{split}
    \mathrm{ver}(a\mathrm{d}_{A}a')_{0}\otimes\mathrm{ver}(a\mathrm{d}_{A}a')_{1}&=(aa'_{0})_{0}\otimes(S(a'_{1})\mathrm{d}_{H}a'_{2})_{0}\otimes(aa'_{0})_{1}(S(a'_{1})\mathrm{d}_{H}a'_{2})_{1}\\&=a_{0}a'_{0}\otimes S(a'_{2})_{1}(\mathrm{d}_{H}a'_{3})_{0}\otimes a_{1}a'_{1}S(a'_{2})_{2}(\mathrm{d}_{H}a'_{3})_{1}\\&=a_{0}a'_{0}\otimes S(a'_{3})\mathrm{d}_{H}a'_{4}\otimes a_{1}a'_{1}S(a'_{2})a'_{5}\\
    %&=a_{0}a'_{0}\otimes S(a'_{1})\mathrm{d}_{H}a'_{2}\otimes a_{1}a'_{3}=\mathrm{ver}(a_{0}\mathrm{d}_{A}a'_{0})\otimes a_{1}a'_{1}\\
    &=\mathrm{ver}((a\mathrm{d}_{A}a')_{0})\otimes(a\mathrm{d}_{A}a')_{1}.
\end{split}
\]
Thus, exactness of the sequence \eqref{Atiyahseq'} in $_{A}\mm$ coincides with its exactness in $_{A}\mm^{H}$.
\medskip

A \textit{connection} is a splitting of \eqref{Atiyahseq'}, i.e., a morphism $c:A\otimes{}^{\mathrm{co}H}\Omega^1(H)\to\Omega^1(A)$ in $_{A}\mm^{H}$ such that $\mathrm{ver}\circ c=\mathrm{Id}$. 
Equivalently, a connection is a morphism $\pi:\Omega^1(A)\to\Omega^1(A)$ in $_{A}\mm^{H}$ such that $\pi\circ\pi=\pi$ and $\mathrm{ker}(\pi)=\Omega^1_{\mathrm{hor}}(A)$.
In fact, as shown in \cite[Proposition 5.41]{BM}, the two notions of connection coincide: to each $\pi$ one associates $c$ such that $c(a\otimes\gamma)=\pi\mathrm{ver}^{-1}(a\otimes\gamma)$ for all $a\in A$ and $\gamma\in{}^{\mathrm{co}H}\Omega^1(H)$ and to each $c$ is associated $\pi=c\circ\mathrm{ver}$.
A connection $\pi$ is called \textit{strong} if $(\mathrm{Id}-\pi)\mathrm{d}_{A}A\subseteq\Omega^1(B)A$ (cf. \cite[Definition 2.1]{Hajac96} and \cite[page 426]{BM}).
\medskip

Consider the universal FODCi $(\Omega^1_u(A),\mathrm{d}^{u}_{A})$ and $(\Omega^1_u(H),\mathrm{d}^{u}_{H})$. These form a QPB if and only if $B=A^{\mathrm{co}H}\subseteq A$ is a Hopf--Galois extension, i.e., if and only if the canonical map
$$
\mathrm{can}\colon A\otimes_BA\to A\otimes H,\qquad
a\otimes_Ba'\mapsto aa'_0\otimes a'_1
$$
is bijective. This result is due to \cite[Lemma 3.2]{Br}, see also \cite[Proposition 1.6]{Hajac96}. In the literature this is sometimes referred to as the \textit{universal} QPB. Connections on the universal QPB are given by unital right $H$-colinear\footnote{Here we endow $H$ with the right $H$-coaction $\mathrm{Ad}_R:h\mapsto h_{2}\otimes S(h_{1})h_{3}$ and $A\otimes A$ with the diagonal right coaction.} maps $\omega\colon H\to A\otimes A$, such that
$$
\mathrm{can}\circ\pi\circ\omega=1_A\otimes\mathrm{Id}_{H},
$$
where $\pi\colon A\otimes A\to A\otimes_BA$ is the quotient map. We also recall that, when the Hopf algebra $H$ has invertible antipode, the right $H$-coaction $\rho$ on $A$ converts to a left $H$-coaction $\lambda(a)=S^{-1}(a_{1})\otimes a_{0}$ for all $a\in A$, making $A$ a left $H^{\mathrm{op}}$-comodule algebra. Then, $\mathrm{Id}_{A}\otimes\rho$ and $\lambda\otimes\mathrm{Id}_{A}$ make $A\otimes A$ into an $H$-bicomodule. So, when $H$ has bijective antipode, strong connections on the universal QPB correspond to unital bicomodule maps $\omega:H\to A\otimes A$ such that $m_{A}\omega(h)=\varepsilon(h)1_{A}$ for all $h\in H$ 
(see \cite[Lemma 5.8]{BM}). Moreover, when $H$ has bijective antipode, the existence of a strong connection is equivalent to faithful flatness of $A$ as a left $B$-module and to the right $H$-equivariant projectivity of $A$, i.e., to the fact that there exists a left $B$-linear and right $H$-colinear map $s:A\to B\otimes A$ such that $m\circ s=\mathrm{Id}_{A}$, where $m$ is the restriction of $m_{A}$ to $B\otimes A$ (cf. \cite[Corollaries 2.2 and 2.4]{DGH}, see also \cite[Section 6.4]{Hajac}). If $B\subseteq A$ is a cleft extension with cleaving map $j$ there exists such an $s$, namely $s(a)=a_{0}j^{-1}(a_{1})\otimes j(a_{2})$. 
\medskip

In \cite[page 427]{BM} a QPB is called \textit{regular} if it descends from the universal QPB (i.e., from a Hopf--Galois extension) such that the latter has a strong connection. Hence, a QPB on a cleft extension is always regular in this sense. Moreover, a QPB is called \textit{strong} if $A\Omega^1(B)=\Omega^1(B)A$ as subspaces of $\Omega^1(A)$, see \cite[Section 5.4.2]{BM}.

\section{The crossed product calculus}\label{Sec3}

Given a left $H$-module algebra $B$, a bicovariant FODC on $H$ and an $H$-module FODC on $B$, there is a canonical right $H$-covariant FODC on the smash product algebra $B\#H$ (see \cite[Theorem 2.7]{PS}). In Section \ref{Sec3.1} we generalize this construction to the case of the crossed product algebra $B\#_{\sigma}H$, when $B$ is a $\sigma$-twisted left $H$-module algebra. A characterization of the smash product calculus follows in Section \ref{Sec:ClasSmash}. Higher order forms for the crossed product algebra are constructed in Section \ref{Sec3.3}, the QPB structure of crossed product 1-forms and corresponding bimodule covariant derivatives on associated bundles are discussed in Section \ref{Sec3.4}, while we prove the 1:1 correspondence of connections and connection $1$-forms of the crossed product calculus in Section \ref{Sec3.5}.

\subsection{Crossed product 1-forms}\label{Sec3.1}

In this subsection $H$ is a Hopf algebra and $B$ is a $\sigma$-twisted left $H$-module algebra. We introduce a right $H$-covariant FODC on the crossed product algebra $B\#_{\sigma}H$. First, we define a twisted analog of a $B$-bimodule. 

\begin{definition}\label{TwistedBimon}
Let $M$ be in $_{B}\mm_{B}$. We say that $M$ is a $\sigma$-\textit{twisted $B$-bimodule} if there exists a linear map $\cdot\colon H\otimes M\to M$ such that for all $m\in M$, $h,h'\in H$ and $b,b'\in B$

\begin{enumerate}
\item[i.)]\ $1_{H}\cdot m=m$,
\item[ii.)]\ $h\cdot(bmb')=(h_{1}\cdot b)(h_{2}\cdot m)(h_{3}\cdot b')$,
\item[iii.)]\ $h\cdot(h'\cdot m)=\sigma(h_{1}\otimes h'_{1})(h_{2}h'_{2}\cdot m)\sigma^{-1}(h_{3}\otimes h'_{3})$.
\end{enumerate}
\end{definition}

If $\sigma$ is the trivial 2-cocycle we recover the notion of left $H$-equivariant $B$-bimodule. Indeed, in this case $B$ becomes a left $H$-module algebra and the first and third equations imply that $M$ is in $_{H}\mm$, while the second equation confers that the left and right $B$-actions of $M$ are left $H$-linear, so that $M$ is in $_{B}(_{H}\mm)_{B}$.

\begin{lemma}\label{twistedHmodcal}
Let $(\Omega^1(B),\mathrm{d}_{B})$ be a FODC on $B$ such that there exists a linear map $\cdot:H\otimes\Omega^1(B)\to\Omega^1(B)$ with 
\begin{align}
h\cdot(b\mathrm{d}_{B}b')&=(h_{1}\cdot b)(h_{2}\cdot\mathrm{d}_{B}b')\label{comp.}\\ \mathrm{d}_{B}(h\cdot b)&=h\cdot \mathrm{d}_{B}b\label{H-lin}\\ \mathrm{d}_{B}\circ\sigma&=0\label{dsigma} 
\end{align}
for every $h\in H$ and $b,b'\in B$. Then $\Omega^1(B)$ is a $\sigma$-twisted $B$-bimodule.
\end{lemma}

\begin{proof}
Since every $\beta\in\Omega^1(B)$ is a finite sum $\beta=b^i\mathrm{d}_Bb'_i$ with $b^{i},b'_{i}\in B$, it is sufficient to show $i.)$, $ii.)$ and $iii.)$ of Definition \ref{TwistedBimon} for $b\mathrm{d}_{B}b'$, with $b,b'\in B$. By assumption $\Omega^1(B)$ is a $B$-bimodule and clearly $i.)$ holds, since
\[
1_{H}\cdot(b\mathrm{d}_{B}b')\overset{\eqref{comp.}}{=}(1_{H}\cdot b)(1_{H}\cdot\mathrm{d}_{B}b')\overset{\eqref{H-lin}}{=}b\mathrm{d}_{B}(1_{H}\cdot b')=b\mathrm{d}_{B}b'.
\]
Furthermore, for all $b,b',b''\in B$ and $h\in H$, we have
\[
\begin{split}
h\cdot((b\mathrm{d}_{B}b')b'')&=h\cdot(b((\mathrm{d}_{B}b')b''))=h\cdot(b\mathrm{d}_{B}(b'b'')-bb'\mathrm{d}_{B}b'')\\&\overset{\eqref{comp.}}{=}(h_{1}\cdot b)(h_{2}\cdot\mathrm{d}_{B}(b'b''))-(h_{1}\cdot bb')(h_{2}\cdot\mathrm{d}_{B}b'')\\&\overset{\eqref{H-lin}}{=}(h_{1}\cdot b)\mathrm{d}_{B}(h_{2}\cdot b'b'')-(h_{1}\cdot b )(h_{2}\cdot b')(h_{3}\cdot\mathrm{d}_{B}b'')\\&\overset{\eqref{H-lin}}{=}(h_{1}\cdot b)\mathrm{d}_{B}((h_{2}\cdot b')(h_{3}\cdot b''))-(h_{1}\cdot b)(h_{2}\cdot b')\mathrm{d}_{B}(h_{3}\cdot b'')\\&=(h_{1}\cdot b)\mathrm{d}_{B}(h_{2}\cdot b')(h_{3}\cdot b'')\overset{\eqref{H-lin}}{=}(h_{1}\cdot b)(h_{2}\cdot\mathrm{d}_{B}b')(h_{3}\cdot b'')\\&\overset{\eqref{comp.}}{=}(h_{1}\cdot b\mathrm{d}_{B}b')(h_{2}\cdot b''),
\end{split}
\]
which implies
\[
\begin{split}
h\cdot(b(b'\mathrm{d}_{B}b'')b''')&=h\cdot((bb'\mathrm{d}_{B}b'')b''')=(h_{1}\cdot bb'\mathrm{d}_{B}b'')(h_{2}\cdot b''')\\&=
(h_{1}\cdot bb')(h_{2}\cdot\mathrm{d}_{B}b'')(h_{3}\cdot b''')\\&=(h_{1}\cdot b)(h_{2}\cdot b')(h_{3}\cdot\mathrm{d}_{B}b'')(h_{4}\cdot b''')\\&=(h_{1}\cdot b)(h_{2}\cdot b'\mathrm{d}_{B}b'')(h_{3}\cdot b'''),
\end{split}
\]
i.e., $ii.)$ holds. For $iii.)$ we observe that, given \eqref{dsigma}, we have
\[
\begin{split}
\mathrm{d}_{B}(\sigma^{-1}(h\otimes h'))&=\varepsilon(h_{1}h'_{1})\mathrm{d}_{B}(\sigma^{-1}(h_{2}\otimes h'_{2}))=\sigma^{-1}(h_{1}\otimes h'_{1})\sigma(h_{2}\otimes h'_{2})\mathrm{d}_{B}(\sigma^{-1}(h_{3}\otimes h'_{3}))\\&=\sigma^{-1}(h_{1}\otimes h'_{1})\mathrm{d}_{B}(\sigma(h_{2}\otimes h'_{2}))\sigma^{-1}(h_{3}\otimes h'_{3})+\sigma^{-1}(h_{1}\otimes h'_{1})\sigma(h_{2}\otimes h'_{2})\mathrm{d}_{B}(\sigma^{-1}(h_{3}\otimes h'_{3}))\\&=\sigma^{-1}(h_{1}\otimes h'_{1})\mathrm{d}_{B}(\sigma(h_{2}\otimes h'_{2})\sigma^{-1}(h_{3}\otimes h'_{3}))=\sigma^{-1}(h_{1}\otimes h'_{1})\mathrm{d}_{B}(\varepsilon(h_{2}h'_{2})1_{B})=0.
\end{split}
\]
Thus, we obtain that
\begin{align*}
    h\cdot(h'\cdot b\mathrm{d}_{B}b')
    &=(h_1\cdot(h'_1\cdot b))(h_2\cdot(h'_2\cdot\mathrm{d}_{B}b'))=(h_1\cdot(h'_1\cdot b))\mathrm{d}_B(h_2\cdot(h'_2\cdot b'))\\
    &=\sigma(h_1\otimes h'_1)(h_2h'_2\cdot b)\sigma^{-1}(h_3\otimes h'_3)\mathrm{d}_B(\sigma(h_4\otimes h'_4)(h_5h'_5\cdot b')\sigma^{-1}(h_6\otimes h'_6))\\
    &=\sigma(h_1\otimes h'_1)(h_2h'_2\cdot b)\sigma^{-1}(h_3\otimes h'_3)\sigma(h_4\otimes h'_4)\mathrm{d}_B(h_5h'_5\cdot b')\sigma^{-1}(h_6\otimes h'_6)\\
    &=\sigma(h_1\otimes h'_1)(h_2h'_2\cdot b)\mathrm{d}_B(h_3h'_3\cdot b')\sigma^{-1}(h_4\otimes h'_4)\\
    &=\sigma(h_1\otimes h'_1)(h_2h'_2\cdot b)(h_3h'_3\cdot\mathrm{d}_{B}b')\sigma^{-1}(h_4\otimes h'_4)\\
    &=\sigma(h_1\otimes h'_1)(h_2h'_2\cdot(b\mathrm{d}_Bb'))\sigma^{-1}(h_3\otimes h'_3),
\end{align*}
where in the fourth equation we used the Leibniz rule and \eqref{dsigma}.
\end{proof}

Note that, if the map $\cdot:H\otimes\Omega^1(B)\to\Omega^1(B)$ in Lemma \ref{twistedHmodcal} exists, then it is unique, since by \eqref{comp.} and \eqref{H-lin} it is completely determined as $h\cdot(b\mathrm{d}_{B}b')=(h_{1}\cdot b)\mathrm{d}_{B}(h_{2}\cdot b')$ for all $h\in H$ and $b,b'\in B$.

\begin{definition}\label{HmodFODC}
Let $(\Omega^1(B),\mathrm{d}_{B})$ be a FODC on a $\sigma$-twisted left $H$-module algebra $B$. We say that it is a $\sigma$-\textit{twisted $H$-module} differential calculus if there exists a linear map $\cdot:H\otimes\Omega^1(B)\to\Omega^1(B)$ such that \eqref{comp.}, \eqref{H-lin} and \eqref{dsigma} of Lemma \ref{twistedHmodcal} are satisfied. In particular, $\Omega^1(B)$ is a $\sigma$-twisted $B$-bimodule.
\end{definition}

\begin{remark}
To explain the nomenclature of the previous definition, recall that in \cite[Definition 2.5]{PS}
an $H$-module differential calculus on a left $H$-module algebra $B$ is a FODC $(\Omega^1(B),\mathrm{d}_B)$, where $\Omega^1(B)$ is in $_{B}(_{H}\mm)$ and $\mathrm{d}_{B}$ is in $_{H}\mm$; it follows that $\Omega^1(B)$ is in $_{B}(_{H}\mm)_{B}$. In Definition \ref{HmodFODC} we recover this case if $\sigma$ is the trivial 2-cocycle. Indeed, as already mentioned, Definition \ref{TwistedBimon} implies that $\Omega^1(B)$ is in $_{B}(_{H}\mm)_{B}$, while \eqref{H-lin} of Lemma \ref{twistedHmodcal} implies that $\mathrm{d}_{B}$ is left $H$-linear.
Observe that \eqref{dsigma} is trivially satisfied when $\sigma$ is trivial, while we will show that when $\sigma$ is not trivial it is a necessary condition in order to construct the FODC on the crossed product algebra $B\#_{\sigma}H$.
\end{remark}

\begin{lemma}\label{bimoduleconstruction}
If $M$ is a $\sigma$-twisted $B$-bimodule and $N$ is in $^{H}_{H}\mm^{H}_{H}$, then $(M\otimes H)\oplus(B\otimes N)$ is in $_{B\#_{\sigma}H}\mm^{H}_{B\#_{\sigma}H}$ with respect to the $B\#_\sigma H$-actions determined for all $b,b'\in B$, $h,h'\in H$, $\beta\in M$ and $\gamma\in N$ by
\begin{align}
    (b'\otimes h')\otimes(\beta\otimes h+b\otimes\gamma)&\mapsto b'(h'_{1}\cdot\beta)\sigma(h'_{2}\otimes h_{1})\otimes h'_{3}h_{2}+b'(h'_{1}\cdot b)\sigma(h'_{2}\otimes\gamma_{-1})\otimes h'_{3}\gamma_{0}\label{leftact}\\
    (\beta\otimes h+b\otimes\gamma)\otimes(b'\otimes h')&\mapsto\beta(h_{1}\cdot b')\sigma(h_{2}\otimes h'_{1})\otimes h_{3}h'_{2}+b(\gamma_{-2}\cdot b')\sigma(\gamma_{-1}\otimes h'_{1})\otimes\gamma_{0}h'_{2}\label{rightact}
\end{align}
and right $H$-coaction
\begin{equation}\label{rightcoact}
    \rho':C\to C\otimes H,\qquad \beta\otimes h+b\otimes\gamma\mapsto\beta\otimes h_{1}\otimes h_{2}+b\otimes\gamma_{0}\otimes\gamma_{1}.
\end{equation}
\end{lemma}

\begin{proof}
Let $C:=(M\otimes H)\oplus(B\otimes N)$ and fix $b,b',b''\in B$, $h,h',h''\in H$, $\beta\in M$ and $\gamma\in N$. We first show that $C$ is a $B\#_{\sigma}H$-bimodule. Indeed, using that $M$ is a $\sigma$-twisted $B$-bimodule and $B$ is a $\sigma$-twisted left $H$-module algebra, we obtain
\[
\begin{split}
(b''\otimes h'')&\cdot((b'\otimes h')\cdot(\beta\otimes h))=(b''\otimes h'')\cdot(b'(h'_{1}\cdot\beta)\sigma(h'_{2}\otimes h_{1})\otimes h'_{3}h_{2})\\&=b''(h''_{1}\cdot b'(h'_{1}\cdot\beta)\sigma(h'_{2}\otimes h_{1}))\sigma(h''_{2}\otimes h'_{3}h_{2})\otimes h''_{3}h'_{4}h_{3}\\&=b''(h''_{1}\cdot b')(h''_{2}\cdot(h'_{1}\cdot\beta))(h''_{3}\cdot\sigma(h'_{2}\otimes h_{1}))\sigma(h''_{4}\otimes h'_{3}h_{2})\otimes h''_{5}h'_{4}h_{3}\\&=b''(h''_{1}\cdot b')\sigma(h''_{2}\otimes h'_{1})(h''_{3}h'_{2}\cdot\beta)\sigma^{-1}(h''_{4}\otimes h'_{3})(h''_{5}\cdot\sigma(h'_{4}\otimes h_{1}))\sigma(h''_{6}\otimes h'_{5}h_{2})\otimes h''_{7}h'_{6}h_{3}\\&=b''(h''_{1}\cdot b')\sigma(h''_{2}\otimes h'_{1})(h''_{3}h'_{2}\cdot\beta)\sigma^{-1}(h''_{4}\otimes h'_{3})\sigma(h''_{5}\otimes h'_{4})\sigma(h''_{6}h'_{5}\otimes h_{1})\otimes h''_{7}h'_{6}h_{2}\\&=b''(h''_{1}\cdot b')\sigma(h''_{2}\otimes h'_{1})(h''_{3}h'_{2}\cdot\beta)\sigma(h''_{4}h'_{3}\otimes h_{1})\otimes h''_{5}h'_{4}h_{2}\\&=(b''(h''_{1}\cdot b')\sigma(h''_{2}\otimes h'_{1})\otimes h''_{3}h'_{2})\cdot(\beta\otimes h)\\&=((b''\otimes h'')(b'\otimes h'))\cdot(\beta\otimes h)
\end{split}
\]
and
\[
(1_{B}\otimes1_{H})\cdot(\beta\otimes h)=1_{B}(1_{H}\cdot\beta)\sigma(1_{H}\otimes h_{1})\otimes1_{H}h_{2}=\beta\varepsilon(h_{1})\otimes h_{2}=\beta\otimes h.
\]
Similarly, it follows that
\[
\begin{split}
(b''\otimes h'')\cdot&((b'\otimes h')\cdot(b\otimes\gamma))=(b''\otimes h'')\cdot(b'(h'_{1}\cdot b)\sigma(h'_{2}\otimes\gamma_{-1})\otimes h'_{3}\gamma_{0})\\&=b''(h''_{1}\cdot b'(h'_{1}\cdot b)\sigma(h'_{2}\otimes\gamma_{-1}))\sigma(h''_{2}\otimes(h'_{3}\gamma_{0})_{-1})\otimes h''_{3}(h'_{3}\gamma_{0})_{0}\\&=b''(h''_{1}\cdot b')(h''_{2}\cdot(h'_{1}\cdot b))(h''_{3}\cdot\sigma(h'_{2}\otimes\gamma_{-2}))\sigma(h''_{4}\otimes h'_{3}\gamma_{-1})\otimes h''_{5}h'_{4}\gamma_{0}\\&=b''(h''_{1}\cdot b')\sigma(h''_{2}\otimes h'_{1})(h''_{3}h'_{2}\cdot b)\sigma^{-1}(h''_{4}\otimes h'_{3})(h''_{5}\cdot\sigma(h'_{4}\otimes\gamma_{-2}))\sigma(h''_{6}\otimes h'_{5}\gamma_{-1})\otimes h''_{7}h'_{6}\gamma_{0}\\&=b''(h''_{1}\cdot b')\sigma(h''_{2}\otimes h'_{1})(h''_{3}h'_{2}\cdot b)\sigma^{-1}(h''_{4}\otimes h'_{3})\sigma(h''_{5}\otimes h'_{4})\sigma(h''_{6}h'_{5}\otimes\gamma_{-1})\otimes h''_{7}h'_{6}\gamma_{0}\\&=b''(h''_{1}\cdot b')\sigma(h''_{2}\otimes h'_{1})(h''_{3}h'_{2}\cdot b)\sigma(h''_{4}h'_{3}\otimes\gamma_{-1})\otimes h''_{5}h'_{4}\gamma_{0}\\&=(b''(h''_{1}\cdot b')\sigma(h''_{2}\otimes h'_{1})\otimes h''_{3}h'_{2})\cdot(b\otimes\gamma)\\&=((b''\otimes h'')(b'\otimes h'))\cdot(b\otimes\gamma)
\end{split}
\]
and
\[
(1_{B}\otimes1_{H})\cdot(b\otimes\gamma)=1_{B}(1_{H}\cdot b)\sigma(1_{H}\otimes\gamma_{-1})\otimes1_{H}\gamma_{0}=b\varepsilon(\gamma_{-1})\otimes\gamma_{0}=b\otimes\gamma.
\]
Thus, $C$ is a left $B\#_{\sigma}H$-module. The fact that $C$ is a right $B\#_{\sigma}H$-module follows in complete analogy. 
Moreover, the two actions are compatible, indeed 
\[
\begin{split}
((b'\otimes h')&\cdot(\beta\otimes h))\cdot(b''\otimes h'')=(b'(h'_{1}\cdot\beta)\sigma(h'_{2}\otimes h_{1})\otimes h'_{3}h_{2})\cdot(b''\otimes h'')\\&=b'(h'_{1}\cdot\beta)\sigma(h'_{2}\otimes h_{1})(h'_{3}h_{2}\cdot b'')\sigma(h'_{4}h_{3}\otimes h''_{1})\otimes h'_{5}h_{4}h''_{2}\\&=b'(h'_{1}\cdot\beta)\sigma(h'_{2}\otimes h_{1})(h'_{3}h_{2}\cdot b'')\sigma^{-1}(h'_{4}\otimes h_{3})\sigma(h'_{5}\otimes h_{4})\sigma(h'_{6}h_{5}\otimes h''_{1})\otimes h'_{7}h_{6}h''_{2}\\&=b'(h'_{1}\cdot\beta)\sigma(h'_{2}\otimes h_{1})(h'_{3}h_{2}\cdot b'')\sigma^{-1}(h'_{4}\otimes h_{3})(h'_{5}\cdot\sigma(h_{4}\otimes h''_{1}))\sigma(h'_{6}\otimes h_{5}h''_{2})\otimes h'_{7}h_{6}h''_{3}\\&=b'(h'_{1}\cdot\beta)(h'_{2}\cdot(h_{1}\cdot b''))(h'_{3}\cdot\sigma(h_{2}\otimes h''_{1}))\sigma(h'_{4}\otimes h_{3}h''_{2})\otimes h'_{5}h_{4}h''_{3}\\&=b'(h'_{1}\cdot\beta(h_{1}\cdot b'')\sigma(h_{2}\otimes h''_{1}))\sigma(h'_{2}\otimes h_{3}h''_{2})\otimes h'_{3}h_{4}h''_{3}\\&=(b'\otimes h')\cdot(\beta(h_{1}\cdot b'')\sigma(h_{2}\otimes h''_{1})\otimes h_{3}h''_{2})\\&=(b'\otimes h')\cdot((\beta\otimes h)\cdot(b''\otimes h''))
\end{split}
\]
and
\[
\begin{split}
((b'\otimes h')&\cdot(b\otimes\gamma))\cdot(b''\otimes h'')=(b'(h'_{1}\cdot b)\sigma(h'_{2}\otimes\gamma_{-1})\otimes h'_{3}\gamma_{0})\cdot(b''\otimes h'')\\&=b'(h'_{1}\cdot b)\sigma(h'_{2}\otimes\gamma_{-1})((h'_{3}\gamma_{0})_{-2}\cdot b'')\sigma((h'_{3}\gamma_{0})_{-1}\otimes h''_{1})\otimes(h'_{3}\gamma_{0})_{0}h''_{2}\\&=b'(h'_{1}\cdot b)\sigma(h'_{2}\otimes\gamma_{-3})(h'_{3}\gamma_{-2}\cdot b'')\sigma(h'_{4}\gamma_{-1}\otimes h''_{1})\otimes h'_{5}\gamma_{0}h''_{2}\\&=b'(h'_{1}\cdot b)\sigma(h'_{2}\otimes\gamma_{-5})(h'_{3}\gamma_{-4}\cdot b'')\sigma^{-1}(h'_{4}\otimes\gamma_{-3})\sigma(h'_{5}\otimes\gamma_{-2})\sigma(h'_{6}\gamma_{-1}\otimes h''_{1})\otimes h'_{7}\gamma_{0}h''_{2}\\&=b'(h'_{1}\cdot b)\sigma(h'_{2}\otimes\gamma_{-5})(h'_{3}\gamma_{-4}\cdot b'')\sigma^{-1}(h'_{4}\otimes\gamma_{-3})(h'_{5}\cdot\sigma(\gamma_{-2}\otimes h''_{1}))\sigma(h'_{6}\otimes\gamma_{-1}h''_{2})\otimes h'_{7}\gamma_{0}h''_{3}\\&=b'(h'_{1}\cdot b)(h'_{2}\cdot(\gamma_{-3}\cdot b''))(h'_{3}\cdot\sigma(\gamma_{-2}\otimes h''_{1}))\sigma(h'_{4}\otimes\gamma_{-1}h''_{2})\otimes h'_{5}\gamma_{0}h''_{3}\\&=b'(h'_{1}\cdot b(\gamma_{-2}\cdot b'')\sigma(\gamma_{-1}\otimes h''_{1}))\sigma(h'_{2}\otimes(\gamma_{0}h''_{2})_{-1})\otimes h'_{3}(\gamma_{0}h''_{2})_{0}\\&=(b'\otimes h')\cdot(b(\gamma_{-2}\cdot b'')\sigma(\gamma_{-1}\otimes h''_{1})\otimes\gamma_{0}h''_{2})\\&=(b'\otimes h')\cdot((b\otimes\gamma)\cdot(b''\otimes h'')).
\end{split}
\]
Hence $C$ is a $B\#_{\sigma}H$-bimodule. Recall that $B\#_{\sigma}H$ is a right $H$-comodule algebra with $\rho:B\#_{\sigma}H\to(B\#_{\sigma}H)\otimes H$, $b\otimes h\mapsto b\otimes h_{1}\otimes h_{2}$. Clearly, $C$ is a right $H$-comodule via 
\begin{equation}%\label{rightcoact}
\rho':C\to C\otimes H,\qquad \beta\otimes h+b\otimes\gamma\mapsto\beta\otimes h_{1}\otimes h_{2}+b\otimes\gamma_{0}\otimes\gamma_{1}
\end{equation}
and we show that $\rho'$ is compatible with the $B\#_{\sigma}H$-actions, so that $C$ is in $_{B\#_{\sigma}H}\mm^{H}_{B\#_{\sigma}H}$. We have
\[
\begin{split}
\rho'((\beta\otimes h+b\otimes\gamma)\cdot(b'\otimes h'))&=\rho'(\beta(h_{1}\cdot b')\sigma(h_{2}\otimes h'_{1})\otimes h_{3}h'_{2}+b(\gamma_{-2}\cdot b')\sigma(\gamma_{-1}\otimes h'_{1})\otimes\gamma_{0}h'_{2})\\&=\beta(h_{1}\cdot b')\sigma(h_{2}\otimes h'_{1})\otimes h_{3}h'_{2}\otimes h_{4}h'_{3}+b(\gamma_{-2}\cdot b')\sigma(\gamma_{-1}\otimes h'_{1})\otimes\gamma_{0}h'_{2}\otimes\gamma_{1}h'_{3}\\&=(\beta\otimes h_{1})\cdot(b'\otimes h'_{1})\otimes h_{2}h'_{2}+(b\otimes\gamma_{0})\cdot(b'\otimes h'_{1})\otimes\gamma_{1}h'_{2}
\end{split}
\]
and 
\[
\begin{split}
\rho'((b'\otimes h')\cdot(\beta\otimes h+b\otimes\gamma))&=\rho'(b'(h'_{1}\cdot\beta)\sigma(h'_{2}\otimes h_{1})\otimes h'_{3}h_{2}+b'(h'_{1}\cdot b)\sigma(h'_{2}\otimes\gamma_{-1})\otimes h'_{3}\gamma_{0})\\&=b'(h'_{1}\cdot\beta)\sigma(h'_{2}\otimes h_{1})\otimes h'_{3}h_{2}\otimes h'_{4}h_{3}+b'(h'_{1}\cdot b)\sigma(h'_{2}\otimes\gamma_{-1})\otimes h'_{3}\gamma_{0}\otimes h'_{4}\gamma_{1}\\&=(b'\otimes h'_{1})\cdot(\beta\otimes h_{1})\otimes h'_{2}h_{2}+(b'\otimes h'_{1})\cdot(b\otimes\gamma_{0})\otimes h'_{2}\gamma_{1},
\end{split}
\]
so the thesis follows.
\end{proof}

\begin{corollary}\label{cor:sigmatwisted}
Given $(\Omega^1(H),\mathrm{d}_{H})$ a bicovariant FODC on $H$ and $(\Omega^1(B),\mathrm{d}_{B})$ a $\sigma$-twisted $H$-module FODC on $B$, $(\Omega^1(B)\otimes H)\oplus(B\otimes\Omega^1(H))$ is in $_{B\#_{\sigma}H}\mm^{H}_{B\#_{\sigma}H}$ with left and right $B\#_{\sigma}H$-actions given by \eqref{leftact} and \eqref{rightact}, respectively, and right $H$-coaction given by \eqref{rightcoact}.
\end{corollary}

Finally, we construct the announced right $H$-covariant FODC on $B\#_{\sigma}H$, starting from a bicovariant FODC on $H$ and a $\sigma$-twisted $H$-module FODC on $B$. We call it the \textit{crossed product calculus} on $B\#_{\sigma}H$.

\begin{theorem}\label{canFODC}
Given $(\Omega^1(H),\mathrm{d}_{H})$ a bicovariant FODC on $H$ and $(\Omega^1(B),\mathrm{d}_{B})$ a $\sigma$-twisted $H$-module FODC on $B$ we obtain a right $H$-covariant FODC $(\Omega^1(B\#_\sigma H),\mathrm{d}_{\#_\sigma})$ on $B\#_{\sigma}H$, where 
$$
\Omega^1(B\#_{\sigma}H):=(\Omega^1(B)\otimes H)\oplus(B\otimes\Omega^1(H))\ \text{and}\ \mathrm{d}_{\#_\sigma}:B\#_{\sigma}H\to\Omega^1(B\#_{\sigma}H),\ b\otimes h\mapsto\mathrm{d}_{B}b\otimes h+b\otimes \mathrm{d}_{H}h.
$$
Moreover, the right $H$-coaction of $B\#_{\sigma}H$ is differentiable with respect to this calculus.
\end{theorem}

\begin{proof}
By Corollary \ref{cor:sigmatwisted} we know that $\Omega^1(B\#_{\sigma}H):=(\Omega^1(B)\otimes H)\oplus(B\otimes\Omega^1(H))$ is in $_{B\#_{\sigma}H}\mm^{H}_{B\#_{\sigma}H}$.
In a next step 
we prove that $\mathrm{d}_{\#_\sigma}$, as defined above, satisfies the Leibniz rule. Indeed, we have 
\[
\begin{split}
(\mathrm{d}_{\#_\sigma}(b\otimes h))\cdot&(b'\otimes h')+(b\otimes h)\cdot(\mathrm{d}_{\#_\sigma}(b'\otimes h'))\\&=(\mathrm{d}_{B}b\otimes h+b\otimes\mathrm{d}_{H}h)\cdot(b'\otimes h')+(b\otimes h)\cdot(\mathrm{d}_{B}b'\otimes h'+b'\otimes\mathrm{d}_{H}h')\\&=(\mathrm{d}_{B}b)(h_{1}\cdot b')\sigma(h_{2}\otimes h'_{1})\otimes h_{3}h'_{2}+b((\mathrm{d}_{H}h)_{-2}\cdot b')\sigma((\mathrm{d}_{H}h)_{-1}\otimes h'_{1})\otimes(\mathrm{d}_{H}h)_{0}h'_{2}\\&\quad+b(h_{1}\cdot \mathrm{d}_{B}b')\sigma(h_{2}\otimes h'_{1})\otimes h_{3}h'_{2}+b(h_{1}\cdot b')\sigma(h_{2}\otimes(\mathrm{d}_{H}h')_{-1})\otimes h_{3}(\mathrm{d}_{H}h')_{0}\\&=(\mathrm{d}_{B}b)(h_{1}\cdot b')\sigma(h_{2}\otimes h'_{1})\otimes h_{3}h'_{2}+b(h_{1}\cdot b')\sigma(h_{2}\otimes h'_{1})\otimes(\mathrm{d}_{H}h_{3})h'_{2}\\&\quad+b\mathrm{d}_{B}(h_{1}\cdot b')\sigma(h_{2}\otimes h'_{1})\otimes h_{3}h'_{2}+b(h_{1}\cdot b')\sigma(h_{2}\otimes h'_{1})\otimes h_{3}\mathrm{d}_{H}h'_{2}\\&=b(h_{1}\cdot b')\sigma(h_{2}\otimes h'_{1})\otimes \mathrm{d}_{H}(h_{3}h'_{2})+\mathrm{d}_{B}(b(h_{1}\cdot b'))\sigma(h_{2}\otimes h'_{1})\otimes h_{3}h'_{2}\\&=b(h_{1}\cdot b')\sigma(h_{2}\otimes h'_{1})\otimes \mathrm{d}_{H}(h_{3}h'_{2})+\mathrm{d}_{B}(b(h_{1}\cdot b')\sigma(h_{2}\otimes h'_{1}))\otimes h_{3}h'_{2}\\&=\mathrm{d}_{\#_\sigma}((b\otimes h)(b'\otimes h')),
\end{split}
\]
where we used the fact that $\mathrm{d}_{H}$ is left $H$-colinear and \eqref{H-lin} in the third equality, Leibniz rules of $\mathrm{d}_{H}$ and $\mathrm{d}_{B}$ in the fourth and \eqref{dsigma} in the fifth. \medskip

Furthermore, $\Omega^1(B\#_{\sigma}H)$ is generated by $\mathrm{d}_{\#_\sigma}(B\#_{\sigma}H)$ as left $B\#_{\sigma}H$-module. In order to prove this it is sufficient to show that the left $B\#_{\sigma}H$-submodule of $\Omega^1(B\#_\sigma H)$ generated by $\mathrm{d}_{\#_\sigma}(B\#_{\sigma}H)$ contains $b\mathrm{d}_{B}b'\otimes h$ and $b\otimes h\mathrm{d}_{H}h'$ for every $b,b'\in B$ and $h,h'\in H$. The first element is obtained by 
\[
\begin{split}
(b\otimes1_{H})\cdot &\mathrm{d}_{\#_\sigma}(b'\otimes h)-(bb'\otimes1_{H})\cdot\mathrm{d}_{\#_\sigma}(1_{B}\otimes h)\\
&=(b\otimes1_{H})\cdot(\mathrm{d}_{B}b'\otimes h+b'\otimes \mathrm{d}_{H}h)-(bb'\otimes1_{H})\cdot(1_{B}\otimes \mathrm{d}_{H}h)\\&=b(1_{H}\cdot \mathrm{d}_{B}b')\sigma(1_{H}\otimes h_{1})\otimes1_{H}h_{2}+b(1_{H}\cdot b')\sigma(1_{H}\otimes(\mathrm{d}_{H}h)_{-1})\otimes1_{H}(\mathrm{d}_{H}h)_{0}\\&\quad-bb'(1_{H}\cdot1_{B})\sigma(1_{H}\otimes(\mathrm{d}_{H}h)_{-1})\otimes1_{H}(\mathrm{d}_{H}h)_{0}\\&=b\mathrm{d}_{B}b'\varepsilon(h_{1})\otimes h_{2}=b\mathrm{d}_{B}b'\otimes h,
\end{split}
\]
while for the second element we have that
\[
\begin{split}
(b\sigma^{-1}(h_{1}\otimes h'_{1})\otimes h_{2})\cdot\mathrm{d}_{\#_\sigma}(1_{B}\otimes h'_{2})&=(b\sigma^{-1}(h_{1}\otimes h'_{1})\otimes h_{2})\cdot(1_{B}\otimes\mathrm{d}_{H}h'_{2})\\&=b\sigma^{-1}(h_{1}\otimes h'_{1})(h_{2}\cdot1_{B})\sigma(h_{3}\otimes(\mathrm{d}_{H}h'_{2})_{-1})\otimes h_{4}(\mathrm{d}_{H}h'_{2})_{0}\\&=b\sigma^{-1}(h_{1}\otimes h'_{1})\varepsilon(h_{2})\sigma(h_{3}\otimes h'_{2})\otimes h_{4}\mathrm{d}_{H}h'_{3}\\&=b\sigma^{-1}(h_{1}\otimes h'_{1})\sigma(h_{2}\otimes h'_{2})\otimes h_{3}\mathrm{d}_{H}h'_{3}\\&=b\varepsilon(h_{1})\varepsilon(h'_{1})\otimes h_{2}\mathrm{d}_{H}h'_{2}=b\otimes h\mathrm{d}_{H}h',
\end{split}
\]
using again that $\mathrm{d}_{H}$ is left $H$-colinear. Hence, the surjectivity condition is shown and $(\Omega^1(B\#_{\sigma}H),\mathrm{d}_{\#_\sigma})$ is a FODC on the crossed product algebra $B\#_{\sigma}H$. 

Now, denoting by $\rho$ the right $H$-coaction of $B\#_{\sigma}H$, define $\widehat{\rho}:\Omega^1(B\#_{\sigma}H)\to\Omega^1((B\#_{\sigma}H)\otimes H)$ as $\beta\otimes h+b\otimes\gamma\mapsto\beta\otimes h_{1}\otimes h_{2}+b\otimes\gamma_{0}\otimes\gamma_{1}+b\otimes\gamma_{-1}\otimes\gamma_{0}$. This is a morphism in $_{B\#_{\sigma}H}\mm_{B\#_{\sigma}H}$, since
\[
\begin{split}
\widehat{\rho}((b'\otimes h')\cdot(\beta\otimes h+b\otimes\gamma))
&=\widehat{\rho}(b'(h'_{1}\cdot\beta)\sigma(h'_{2}\otimes h_{1})\otimes h'_{3}h_{2}+b'(h'_{1}\cdot b)\sigma(h'_{2}\otimes\gamma_{-1})\otimes h'_{3}\gamma_{0})
\\&=b'(h'_{1}\cdot\beta)\sigma(h'_{2}\otimes h_{1})\otimes h'_{3}h_{2}\otimes h'_{4}h_{3}+b'(h'_{1}\cdot b)\sigma(h'_{2}\otimes\gamma_{-1})\otimes h'_{3}\gamma_{0}\otimes h'_{4}\gamma_{1}\\
&\quad+b'(h'_{1}\cdot b)\sigma(h'_{2}\otimes\gamma_{-2})\otimes h'_{3}\gamma_{-1}\otimes h'_{4}\gamma_{0}\\&=(b'\otimes h'_{1})(\beta\otimes h_{1})\otimes h'_{2}h_{2}+(b'\otimes h'_{1})(b\otimes\gamma_{0})\otimes h'_{2}\gamma_{1}+(b'\otimes h'_{1})(b\otimes\gamma_{-1})\otimes h'_{2}\gamma_{0}\\
&=(b'\otimes h')\cdot\widehat{\rho}(\beta\otimes h+b\otimes\gamma)
\end{split}
\]
and, analogously, $\widehat{\rho}((\beta\otimes h+b\otimes\gamma)\cdot(b'\otimes h'))=\widehat{\rho}(\beta\otimes h+b\otimes\gamma)\cdot(b'\otimes h')$. Furthermore, we have
\[
\begin{split}
    \widehat{\rho}\mathrm{d}_{\#_\sigma}(b\otimes h)&=\mathrm{d}_{B}b\otimes h_{1}\otimes h_{2}+b\otimes(\mathrm{d}_{H}h)_{0}\otimes(\mathrm{d}_{H}h)_{1}+b\otimes(\mathrm{d}_{H}h)_{-1}\otimes(\mathrm{d}_{H}h)_{0}\\&=\mathrm{d}_{B}b\otimes h_{1}\otimes h_{2}+b\otimes\mathrm{d}_{H}h_{1}\otimes h_{2}+b\otimes h_{1}\otimes\mathrm{d}_{H}h_{2}=\mathrm{d}_{\#_\sigma}(b\otimes h_{1})\otimes h_{2}+b\otimes h_{1}\otimes\mathrm{d}_{H}h_{2}\\&=\mathrm{d}_{\otimes}\rho(b\otimes h),
\end{split}
\]
hence $\rho$ is differentiable. Then, $(\Omega^1(B\#_{\sigma}H),\mathrm{d}_{\#_\sigma})$ is right $H$-covariant such that, denoting the canonical projection by $\pi_{1}:\Omega^1((B\#_{\sigma}H)\otimes H)\to\Omega^1(B\#_{\sigma}H)\otimes H$, the right $H$-coaction on $\Omega^1(B\#_{\sigma}H)$ is given by $\rho':=\pi_{1}\circ\widehat{\rho}:\beta\otimes h+b\otimes\gamma\mapsto\beta\otimes h_{1}\otimes h_{2}+b\otimes\gamma_{0}\otimes\gamma_{1}$. On the other hand, the right $H$-covariance of the calculus also follows explicitly from the computation 
\[
\begin{split}
\rho'\mathrm{d}_{\#_\sigma}(b\otimes h)&=\rho'(\mathrm{d}_{B}b\otimes h+b\otimes\mathrm{d}_{H}h)=\mathrm{d}_{B}b\otimes h_{1}\otimes h_{2}+b\otimes(\mathrm{d}_{H}h)_{0}\otimes(\mathrm{d}_{H}h)_{1}=\mathrm{d}_{B}b\otimes h_{1}\otimes h_{2}+b\otimes\mathrm{d}_{H}h_{1}\otimes h_{2}\\&=\mathrm{d}_{\#_\sigma}(b\otimes h_{1})\otimes h_{2}=(\mathrm{d}_{\#_\sigma}\otimes\mathrm{Id}_{H})\rho(b\otimes h),
\end{split}
\]
using that $\mathrm{d}_{H}$ is right $H$-colinear. 
\end{proof}
As one easily verifies, the above crossed product calculus $(\Omega^1(B\#_\sigma H),\mathrm{d}_{\#_\sigma})$ reduces to the smash product calculus $(\Omega^1(B\# H),\mathrm{d}_\#)$ of \cite[Theorem 2.7]{PS} in case of a trivial extension, where $\sigma$ is the trivial $2$-cocycle.

We show that the condition $\mathrm{d}_{B}\circ\sigma=0$ is necessary in case we want $(\Omega^1(B\#_\sigma H),\mathrm{d}_{\#_\sigma})$ to be a FODC on $B\#_{\sigma}H$ starting from a bicovariant FODC $(\Omega^1(H),\mathrm{d}_{H})$ on $H$ and a FODC $(\Omega^1(B),\mathrm{d}_{B})$ on $B$ such that $\Omega^1(B)$ is a $\sigma$-twisted $B$-bimodule.

\begin{proposition}\label{prop:dsigma}
Let $(\Omega^1(B),\mathrm{d}_{B})$ be a FODC on $B$ such that $\Omega^1(B)$ is a $\sigma$-twisted $B$-bimodule 
and $(\Omega^1(H),\mathrm{d}_{H})$ a bicovariant FODC on $H$ such that $(\Omega^1(B\#_\sigma H),\mathrm{d}_{\#_\sigma})$, as defined in Theorem \ref{canFODC}, is a FODC on $B\#_{\sigma}H$. Then $\mathrm{d}_{B}\circ\sigma=0$.
\end{proposition}

\begin{proof}
Since $\Omega^1(B)$ is a $\sigma$-twisted $B$-bimodule and $\Omega^1(H)$ is in $^{H}_{H}\mm^{H}_{H}$, then $\Omega^1(B\#_\sigma H)$ is in $_{B\#_{\sigma}H}\mm^{H}_{B\#_{\sigma}H}$ by Lemma \ref{bimoduleconstruction}. Since $\mathrm{d}_{\#_\sigma}$ satisfies the Leibniz rule by assumption we obtain
\[
\begin{split}
\mathrm{d}_{B}(\sigma(h_{1}\otimes h'_{1}))\otimes h_{2}h'_{2}+\sigma(h_{1}\otimes h'_{1})\otimes\mathrm{d}_{H}(h_{2}h'_{2})&=\mathrm{d}_{\#_\sigma}(\sigma(h_{1}\otimes h'_{1})\otimes h_{2}h'_{2})=\mathrm{d}_{\#_\sigma}((1_{B}\otimes h)(1_{B}\otimes h'))\\&=(\mathrm{d}_{\#_\sigma}(1_{B}\otimes h))(1_{B}\otimes h')+(1_{B}\otimes h)\mathrm{d}_{\#_\sigma}(1_{B}\otimes h')\\&=\sigma(h_{1}\otimes h'_{1})\otimes(\mathrm{d}_{H}h_{2})h'_{2}+\sigma(h_{1}\otimes h'_{1})\otimes h_{2}\mathrm{d}_{H}h'_{2}\\&=\sigma(h_{1}\otimes h'_{1})\otimes\mathrm{d}_{H}(h_{2}h'_{2}),
\end{split}
\]
so that $\mathrm{d}_{B}(\sigma(h_{1}\otimes h'_{1}))\otimes h_{2}h'_{2}=0$ and then $\mathrm{d}_{B}(\sigma(h\otimes h'))=0$, for every $h,h'\in H$.
\end{proof}

We conclude this section by discussing an example of the crossed product calculus construction. It is based on a cleft extension which is not trivial, with a non-trivial $2$-cocycle, and thus goes beyond the smash product calculus approach. More examples, based on pointed Hopf algebras, will be outlined in Section \ref{Sec4}.

\begin{example}\label{ex:torus}
We now exemplify the construction of the crossed product calculus given in Theorem \ref{canFODC} on the noncommutative 2-torus. Here $A=
\mathbb{C}[u,u^{-1},v,v^{-1}]/\langle vu-e^{i\theta} uv\rangle$ is the algebra generated by two invertible generators $u,v$, modulo the relation $vu=e^{i\theta}uv$, where $\theta\in\mathbb{R}$. Then, $A$ is a right $H$-comodule algebra, where $H=\mathbb{C}[t,t^{-1}]$ is the Hopf algebra with coproduct, counit and antipode given by $\Delta(t^{n})=t^{n}\otimes t^{n}$, $\varepsilon(t^{n})=1$ and $S(t^{n})=t^{-n}$, for $n\in\mathbb{Z}$. The coaction $\rho:A\to A\otimes H$ is defined on generators by $\rho(u):=u\otimes t$ and $\rho(v):=v\otimes t^{-1}$ and extended as an algebra morphism. The subalgebra of coinvariants is given by $B=A^{\mathrm{co}H}=\mathrm{span}_{\mathbb{C}}\{(uv)^{k}\ |\ k\in\mathbb{Z}\}$ and $B\subseteq A$ forms a cleft extension with cleaving map defined as
\[
j:H\to A,\ t^{k}\mapsto u^{k},\ t^{-k}\mapsto v^{k},
\]
for all $k\geq0$, see \cite{KLS}. In particular, $j$ is not an algebra morphism and thus $B\subseteq A$ is not a trivial extension. The convolution inverse $j^{-1}:H\to A$ is given by $j^{-1}(t^{k})=u^{-k}$ and $j^{-1}(t^{-k})=v^{-k}$, for all $k\geq 0$. Hence we obtain $\cdot:H\otimes B\to B$, $h\otimes b\mapsto j(h_{1})bj^{-1}(h_{2})$ and $\sigma:H\otimes H\to B$, $h\otimes h'\mapsto j(h_{1})j(h'_{1})j^{-1}(h_{2}h'_{2})$ such that $A\cong B\#_{\sigma}H$ as right $H$-comodule algebras. More explicitly, for $k\geq0$ and $l\in\mathbb{Z}$, we have
\[
t^{k}\cdot(uv)^{l}=
u^{k}(uv)^{l}u^{-k}=e^{-i\theta kl}(uv)^{l},\quad t^{-k}\cdot(uv)^{l}=
v^{k}(uv)^{l}v^{-k}=e^{i\theta kl}(uv)^{l},
\]
i.e., $t^{k}\cdot(uv)^{l}=e^{-i\theta kl}(uv)^{l}$, for all $k,l\in\mathbb{Z}$. Moreover, for all $k,s\geq0$, we have
\[
\sigma(t^{k}\otimes t^{s})=u^{k}u^{s}u^{-(k+s)}=1,\qquad \sigma(t^{-k}\otimes t^{-s})=v^{k}v^{s}v^{-(k+s)}=1,
\]
while
\[
\sigma(t^{k}\otimes t^{-s})=u^{k}v^{s}j^{-1}(t^{k-s})=\begin{cases}
    u^{k}v^{s}u^{s-k}=
    e^{-i\theta s(k-s)}u^{s}v^{s}&\text{if}\ s\leq k\\
    %u^{k}v^{s}v^{k-s}=
    u^{k}v^{k}&\text{if}\ s>k
\end{cases}
\]
and
\[
\sigma(t^{-k}\otimes t^{s})=v^{k}u^{s}j^{-1}(t^{s-k})=\begin{cases}
    %v^{k}u^{s}u^{-s+k}=
e^{i\theta k^{2}}u^{k}v^{k}&\text{if}\ k\leq s\\
    v^{k}u^{s}v^{s-k}=
e^{i\theta sk}u^{s}v^{s}&\text{if}\ k>s
\end{cases}.
\]
We now construct a crossed product calculus on $B\#_{\sigma}H$. Let us observe that, for a FODC $(\Omega^1(B),\mathrm{d}_{B})$ on $B$, the linear map $\cdot:H\otimes\Omega^1(B)\to\Omega^1(B)$ given by
\[
t^{k}\cdot((uv)^{l}\mathrm{d}_{B}((uv)^{s}))=(t^{k}\cdot(uv)^{l})\mathrm{d}_{B}(t^{k}\cdot(uv)^{s})=e^{-i\theta k(l+s)}(uv)^{l}\mathrm{d}_{B}((uv)^{s})
\]
is well-defined for all $k,l,s\in\mathbb{Z}$. Hence, by Lemma \ref{twistedHmodcal}, $(\Omega^1(B),\mathrm{d}_{B})$ is a $\sigma$-twisted $H$-module calculus on $B$ if and only if $\mathrm{d}_{B}\circ\sigma=0$ is satisfied. Let us assume that $\mathrm{d}_{B}\circ\sigma=0$ holds true. Since
$\sigma(t\ot t^{-1})=uv$, we obtain $\mathrm{d}_{B}((uv))=0$ and then $\mathrm{d}_{B}((uv)^{l})=0$ for all $l\geq0$. But now, from the Leibniz rule we obtain
\[
0=\mathrm{d}_{B}((uv)(uv)^{-1})=(uv)\mathrm{d}_{B}((uv)^{-1}),\qquad 0=\mathrm{d}_{B}((uv)^{-1}(uv))=\mathrm{d}_{B}((uv)^{-1})(uv).
\]
Thus, $\mathrm{d}_{B}((uv)^{-1})=(uv)^{-1}\cdot\big((uv)\mathrm{d}_{B}((uv)^{-1})\big)=0$ and so $\mathrm{d}_{B}((uv)^{l})=0$ for all $l\leq0$. Hence $(\Omega^1(B),\mathrm{d}_{B})$ has to be the zero calculus on $B$. As a consequence, $\Omega^1(B\#_{\sigma}H)=B\otimes\Omega^1(H)$ and $\mathrm{d}_{\#_{\sigma}}=\mathrm{Id}_{B}\otimes\mathrm{d}_{H}$, where $(\Omega^1(H),\mathrm{d}_{H})$ is a bicovariant FODC on $H$. We choose the bicovariant calculus on $H$ given in \cite[Example 1.11]{BM}. Fix a complex number $q\not=0,1,-1$ and define $\Omega^1(H)$ as the left $H$-module generated by $\mathrm{d}_{H}t$ such that $\mathrm{d}_{H}t\cdot t^{n}:=q^{n}t^{n}\mathrm{d}_{H}t$, for all $n\in\mathbb{Z}$, and $\mathrm{d}_{H}:H\to\Omega^1(H)$ as $\mathrm{d}_{H}(f(t)):=\frac{f(qt)-f(t)}{t(q-1)}\mathrm{d}_{H}t$ for any rational polynomial $f$ in $t$. The right and left $H$-coactions are simply given by $\rho:\Omega^1(H)\to\Omega^1(H)\otimes H$, $t^{n}\mathrm{d}_{H}t\mapsto t^{n}\mathrm{d}_{H}t\otimes t^{n+1}$ and $\lambda:\Omega^1(H)\to H\otimes\Omega^1(H)$, $t^{n}\mathrm{d}_{H}t\mapsto t^{n+1}\otimes t^{n}\mathrm{d}_{H}t$, respectively. Hence the left and right $B\#_{\sigma}H$-actions on $\Omega^1(B\#_{\sigma}H)$ are given by
\begin{align*}
    ((uv)^{l}\otimes t^{k})\cdot((uv)^{m}\otimes t^{n}\mathrm{d}_{H}t)
    &=(uv)^{l}(t^{k}\cdot(uv)^{m})\sigma(t^{k}\otimes t^{n+1})\otimes t^{k}t^{n}\mathrm{d}_{H}t\\
    &=-e^{i\theta km}(uv)^{l+m}\sigma(t^{k}\otimes t^{n+1})\otimes t^{k+n}\mathrm{d}_{H}t
\end{align*}
and
\begin{align*}
    ((uv)^{m}\otimes t^{n}\mathrm{d}_{H}t)\cdot((uv)^{l}\otimes t^{k})
    &=(uv)^{m}(t^{n+1}\cdot(uv)^{l})\sigma(t^{n+1}\otimes t^{k})\otimes(t^{n}\mathrm{d}_{H}t)\cdot t^{k}\\
    &=q^{k}e^{-i\theta(n+1)l}(uv)^{l+m}\sigma(t^{n+1}\otimes t^{k})\otimes t^{k+n}\mathrm{d}_{H}t,
\end{align*}
for all $l,k,m,n\in\mathbb{Z}$, where $\sigma$ assumes the previous values. Moreover, the right $H$-coaction on $\Omega^1(B\#_{\sigma}H)$ is given by $(uv)^{l}\otimes t^{n}\mathrm{d}_{H}t\mapsto (uv)^{l}\otimes t^{n}\mathrm{d}_{H}t\otimes t^{n+1}$. Even if the $\sigma$-twisted $H$-module calculus on $B$ is the trivial one, the non-trivial $\sigma$ appears in the structure of the crossed product calculus on $B\#_{\sigma}H$. In the last section we will discuss a class of examples of crossed product calculi for which the $\sigma$-twisted $H$-module calculus on $B$ is not the trivial one.
\end{example}

\subsection{A classification of the smash product calculus}\label{Sec:ClasSmash}

In this subsection we consider a right $H$-comodule algebra $A$ such that $B:=A^{\mathrm{co}H}\subseteq A$ is a trivial extension. We denote the algebra inclusion by $\iota:B\to A$ and the cleaving map by $j:H\to A$. In the following, we characterize the calculi on $A$ which are isomorphic to the smash product calculus $(\Omega^1(B\# H),\mathrm{d}_\#)$ via the differential of the algebra isomorphism $A\cong B\#H$. Recall that the calculus $(\Omega^1(B\# H),\mathrm{d}_\#)$ is obtained from Theorem \ref{canFODC} if one considers the trivial $2$-cocycle and it corresponds to the calculus given in \cite[Theorem 2.7]{PS}. 
Calculi on smash product algebras have also been considered in \cite{Ryan} and \cite[Section 4]{Majid2020}.

\begin{lemma}\label{iota12}
The algebra morphisms $\iota_1:B\to B\#H$, $b\mapsto b\otimes1_{H}$ and $\iota_2:H\to B\#H$, $h\mapsto1_{B}\otimes h$ are differentiable.
\end{lemma}

\begin{proof}
Clearly $\iota_1$ and $\iota_2$ are morphisms of algebras. We consider the morphisms $\widehat{\iota_1}:\Omega^1(B)\to\Omega^1(B\#H)$, $\beta\mapsto\beta\otimes1_{H}$ and $\widehat{\iota_2}:\Omega^1(H)\to\Omega^1(B\#H)$, $\gamma\mapsto1_{B}\otimes\gamma$. Now $\widehat{\iota_{1}}(b\beta b')=b\beta b'\otimes1_{H}=(b\otimes1_{H})\cdot(\beta\otimes1_{H})\cdot(b'\otimes1_{H})=b\widehat{\iota_1}(\beta)b'$, so $\widehat{\iota_1}$ is a morphism of $B$-bimodules and $\widehat{\iota_{2}}(h\gamma h')=1_{B}\otimes h\gamma h'=(1_{B}\otimes h)\cdot(1_{B}\otimes\gamma)\cdot(1_{B}\otimes h')=h\widehat{\iota_{2}}(\gamma)h'$, so $\widehat{\iota_{2}}$ is a morphism of $H$-bimodules. Furthermore, $\widehat{\iota_1}\mathrm{d}_{B}b=\mathrm{d}_{B}b\otimes1_{H}=\mathrm{d}_{\#}\iota_{1}b$ and $\widehat{\iota_{2}}\mathrm{d}_{H}h=1_{B}\otimes\mathrm{d}_{H}h=\mathrm{d}_{\#}\iota_{2}h$.
\end{proof}

Note that $\widehat{\iota_{1}}$ and $\widehat{\iota_{2}}$ are injective. Observe also that, in case of a pure crossed product $B\#_{\sigma}H$ (i.e., with $\sigma$ not trivial), $\iota_2$ is not a morphism of algebras in general.

\begin{remark}\label{Hmodtrivext}
Consider a FODC $(\Omega^1(A),\mathrm{d}_{A})$ on $A$ and the corresponding pullback calculus $(\Omega^1(B),\mathrm{d}_{B})$ on $B$, where we recall that $\mathrm{d}_B:=\mathrm{d}_A|_B:B\to\Omega^1(B)\subseteq\Omega^1(A)$. In particular, the inclusion $\iota\colon B\to A$ is differentiable. In order to build the right $H$-covariant FODC on $B\#H$ given by Theorem \ref{canFODC} we need that $(\Omega^1(B),\mathrm{d}_{B})$ is an $H$-module FODC, i.e., there must exist an action $\cdot:H\otimes\Omega^1(B)\to\Omega^1(B)$ that satisfies $h\cdot(b\mathrm{d}_{B}b')=(h_{1}\cdot b)(h_{2}\cdot\mathrm{d}_{B}b')=(h_{1}\cdot b)\mathrm{d}_{B}(h_{2}\cdot b')$, i.e.,
\begin{equation}\label{condTheoremsmash}
\begin{split}
h\cdot(b\mathrm{d}_{B}b')&=j(h_{1})bj^{-1}(h_{2})\mathrm{d}_{B}(j(h_{3})b'j^{-1}(h_{4}))\\&%=j(h_{1})bj^{-1}(h_{2})j(h_{3})\mathrm{d}_{A}(b'j^{-1}(h_{4}))+j(h_{1})bj^{-1}(h_{2})\mathrm{d}_{A}(j(h_{3}))b'j^{-1}(h_{4})\\&
=j(h_{1})bb'\mathrm{d}_{A}(j^{-1}(h_{2}))+j(h_{1})b\mathrm{d}_{B}(b')j^{-1}(h_{2})+j(h_{1})bj^{-1}(h_{2})\mathrm{d}_{A}(j(h_{3}))b'j^{-1}(h_{4})
\end{split}
\end{equation}
for all $h\in H$, $b,b'\in B$.
\end{remark}

We know that
the smash product algebra $B\#H$ is isomorphic to $A$ as right $H$-comodule algebras through $\theta\colon A\to B\#H$, $a\mapsto a_{0}j^{-1}(a_{1})\otimes a_{2}$, with inverse $\theta^{-1}:b\ot h\mapsto bj(h)$.
We are now prepared to give a classification of the smash product calculus in terms of the embedding $\iota\colon B\to A$, the cleaving map $j\colon H\to A$ and a "commutativity condition" on the differential $\mathrm{d}_A$.

\begin{theorem}\label{classificationFODC}
Let $B\subseteq A$ be a trivial extension with cleaving map $j:H\to A$. Consider a FODC $(\Omega^1(A),\mathrm{d}_{A})$ on $A$ with $\Omega^1(A)$ torsion-free as left and right $A$-module and $(\Omega^1(B),\mathrm{d}_{B})$ the corresponding pullback calculus on $B$ and a bicovariant FODC $(\Omega^1(H),\mathrm{d}_{H})$ on $H$. Then, $(\Omega^1(A),\mathrm{d}_{A})$ is isomorphic to $(\Omega^1(B\#H),\mathrm{d}_{\#})$ through $\theta$ if and only if:
\begin{enumerate}
\item[(\namedlabel{specialHmod1}{1})] $j$ is differentiable with $\widehat{j}$ \text{injective}. \medskip
\item[(\namedlabel{specialHmod'}{2})] $\widehat{\iota}(\Omega^1(B))j(H)\cap\iota(B)\widehat{j}(\Omega^1(H))=\{0\}$. \medskip
\item[(\namedlabel{specialHmod}{3})]\ $\mathrm{d}_{A}(j(h_{1}))bj^{-1}(h_{2})=-j(h_{1})b\mathrm{d}_{A}(j^{-1}(h_{2}))$ for all $h\in H,\ b\in B$.
\end{enumerate}
If these conditions are satisfied then $(\Omega^1(A),\mathrm{d}_{A})$ is a right $H$-covariant FODC on $A$.
\end{theorem}

\begin{proof}
We are assuming that $(\Omega^1(B),\mathrm{d}_{B})$ is an $H$-module FODC in order to build $(\Omega^1(B\#H),\mathrm{d}_{\#})$, i.e., that \eqref{condTheoremsmash} is satisfied. Thus, $(\Omega^1(A),\mathrm{d}_{A})$ and $(\Omega^1(B\#H),\mathrm{d}_{\#})$ are isomorphic through $\theta$ if and only if $\theta$ is differentiable and $\hat{\theta}$ is bijective (equivalently, $\theta^{-1}$ is differentiable and $\widehat{\theta^{-1}}$ is bijective). We prove that this is equivalent to (1), (2) and (3).

Suppose that $(\Omega^1(A),\mathrm{d}_{A})\cong(\Omega^1(B\#H),\mathrm{d}_{\#})$ as FODCi through $\theta$. Then, $\theta$ and $\theta^{-1}$ are differentiable with $\widehat{\theta^{-1}}=\widehat{\theta}^{-1}\colon\Omega^1(B\#H)\to\Omega^1(A)$. 
Furthermore, we know that $\iota_2$ is differentiable by Lemma \ref{iota12}. Thus, since $j=\theta^{-1}\circ\iota_{2}$, we obtain that $j$ is differentiable with $\widehat{j}=\widehat{\theta^{-1}}\circ\widehat{\iota_2}$. In particular, the diagram 
\[
    \begin{tikzcd}
	\Omega^1(H) & \Omega^1(B\#H) & \Omega^1(A) & \Omega^1(B) \\
	H & B\#H & A & B
	\arrow[from=2-1, to=1-1, "\mathrm{d}_{H}"]
	\arrow[from=2-1, to=2-2, "\iota_{2}"']
	\arrow[from=2-2, to=1-2, "\mathrm{d}_{\#}"']
	\arrow[from=1-1, to=1-2, "\widehat{\iota_2}"]
	\arrow[from=2-2, to=2-3, "\theta^{-1}"']
	\arrow[from=2-3, to=1-3, "\mathrm{d}_{A}"]
	\arrow[from=1-2, to=1-3, "\widehat{\theta^{-1}}"]
	\arrow[from=2-4, to=2-3, "\iota"]
	\arrow[from=2-4, to=1-4, "\mathrm{d}_{B}"']
	\arrow[from=1-4, to=1-3, "\widehat{\iota}"']
\end{tikzcd}
\]
commutes. Moreover, $\widehat{\iota_2}$ and $\widehat{\theta^{-1}}$ are injective, so $\widehat{j}$ is injective and \eqref{specialHmod1} holds.

We give the explicit expression of $\widehat{\theta^{-1}}$, which is uniquely determined by $\theta^{-1}$. 
\[
\begin{split}
    \widehat{\theta^{-1}}(b\mathrm{d}_{B}b'\otimes h)&=\widehat{\theta^{-1}}((b\otimes1_{H})\mathrm{d}_{\#}(b'\otimes h)-(bb'\otimes1_{H})\mathrm{d}_{\#}(1_{B}\otimes h))\\&=\theta^{-1}(b\otimes1_{H})\mathrm{d}_{A}(\theta^{-1}(b'\otimes h))-\theta^{-1}(bb'\otimes1_{H})\mathrm{d}_{A}(\theta^{-1}(1_{B}\otimes h))\\&=b\mathrm{d}_{A}(b'j(h))-bb'\mathrm{d}_{A}(j(h))=bb'\mathrm{d}_{A}(j(h))+b\mathrm{d}_{A}(b')j(h)-bb'\mathrm{d}_{A}(j(h))\\&=b\mathrm{d}_{A}(b')j(h)
\end{split}
\]
and so $\widehat{\theta^{-1}}(\Omega^1(B)\otimes H)=\widehat{\iota}(\Omega^1(B))j(H)$, while 
\[
\widehat{\theta^{-1}}(b\otimes h\mathrm{d}_{H}h')=\widehat{\theta^{-1}}((b\otimes h)\mathrm{d}_{\#}(1_{B}\otimes h'))=\theta^{-1}(b\otimes h)\mathrm{d}_{A}(\theta^{-1}(1_{B}\otimes h'))=bj(h)\mathrm{d}_{A}(j(h'))=b\widehat{j}(h\mathrm{d}_{H}h')
\]
and then $\widehat{\theta^{-1}}(B\otimes\Omega^1(H))=\iota(B)\widehat{j}(\Omega^1(H))$. Thus, we obtain
\begin{align*}
    \widehat{\theta}(\widehat{\iota}(\Omega^1(B))j(H)\cap\iota(B)\widehat{j}(\Omega^1(H)))&\subseteq\widehat{\theta}(\widehat{\theta^{-1}}(\Omega^1(B)\otimes H))\cap\widehat{\theta}(\widehat{\theta^{-1}}(B\otimes\Omega^1(H)))\\
    &=(\Omega^1(B)\otimes H)\cap(B\otimes\Omega^1(H))\\
    &=\{0\},
\end{align*}
from which $\widehat{\iota}(\Omega^1(B))j(H)\cap\iota(B)\widehat{j}(\Omega^1(H))=\{0\}$ follows, since $\widehat{\theta}$ is injective. Thus, \eqref{specialHmod'} is satisfied.

Furthermore, recalling that $j^{-1}(h)_{0}\otimes j^{-1}(h)_{1}=j^{-1}(h_{2})\otimes S(h_{1})$ and $j^{-1}(hh')=j^{-1}(h')j^{-1}(h)$ for all $h,h'\in H$, we obtain
\[
\begin{split}
    \widehat{\theta}(\mathrm{d}_{A}(j(h_{1}))bj^{-1}(h_{2}))&=\mathrm{d}_{\#}(\theta(j(h_{1})))\theta(bj^{-1}(h_{2}))\\
    &=\mathrm{d}_{\#}(j(h_{1})_{0}j^{-1}(j(h_{1})_{1})\otimes j(h_{1})_{2})(b_{0}j^{-1}(h_{2})_{0}j^{-1}(b_{1}j^{-1}(h_{2})_{1})\otimes b_{2}j^{-1}(h_{2})_{2})\\&=\mathrm{d}_{\#}(1_{A}\otimes h_{1})(bj^{-1}(h_{4})j^{-1}(S(h_{3}))\otimes S(h_{2}))=(1_{A}\otimes\mathrm{d}_{H}h_{1})(bj^{-1}(S(h_{3})h_{4})\otimes S(h_{2}))\\&=(1_{A}\otimes\mathrm{d}_{H}h_{1})(b\otimes S(h_{2}))=((\mathrm{d}_{H}h_{1})_{-1}\cdot b)\otimes(\mathrm{d}_{H}h_{1})_{0}S(h_{2})=(h_{1}\cdot b)\otimes\mathrm{d}_{H}(h_{2})S(h_{3})\\&=-(h_{1}\cdot b)\otimes h_{2}\mathrm{d}_{H}(S(h_{3}))=-((h_{1}\cdot b)\otimes h_{2})(1_{A}\otimes\mathrm{d}_{H}(S(h_{3})))\\&=-(j(h_{1})bj^{-1}(h_{2})\otimes h_{3})\mathrm{d}_{\#}(1_{A}\otimes S(h_{4}))=-\theta(j(h_{1})b)\mathrm{d}_{\#}(\theta(j^{-1}(h_{2})))\\&=-\widehat{\theta}(j(h_{1})b\mathrm{d}_{A}(j^{-1}(h_{2})))
\end{split}
\]
and then, since $\widehat{\theta}$ is injective, we obtain that \eqref{specialHmod} holds true. \medskip

On the other hand, suppose that \eqref{specialHmod1}, \eqref{specialHmod'} and \eqref{specialHmod} are satisfied. 
First note that, since \eqref{condTheoremsmash} is satisfied, using \eqref{specialHmod} 
we have $h\cdot(b\mathrm{d}_{B}b')=j(h_{1})b\mathrm{d}_{B}(b')j^{-1}(h_{2})$ for all  $h\in H$ and $b,b'\in B$. Thus, we obtain
\begin{equation}\label{specialHmod2}
h\cdot\beta=j(h_{1})\beta j^{-1}(h_{2}),\qquad \text{for all}\ h\in H\ \text{and}\ \beta\in\Omega^1(B). 
\end{equation}
Next, we show that $\theta^{-1}$ is differentiable, i.e., that there exists a $B\#H$-bimodule morphism $\widehat{\theta^{-1}}:\Omega^1(B\#H)\to\Omega^1(A)$ such that $\widehat{\theta^{-1}}\circ\mathrm{d}_{\#}=\mathrm{d}_{A}\circ\theta^{-1}$. 
We define
\[
\widehat{\theta^{-1}}(\beta\otimes h+b\otimes\gamma):=\widehat{\iota}(\beta)j(h)+b\widehat{j}(\gamma),\qquad \mathrm{for}\ \beta\otimes h+b\otimes\gamma\in\Omega^1(B\#H).
\]
Then
\[
\widehat{\theta^{-1}}\mathrm{d}_{\#}(b\otimes h)=\widehat{\theta^{-1}}(\mathrm{d}_{B}b\otimes h+b\otimes\mathrm{d}_{H}h)=\widehat{\iota}(\mathrm{d}_{B}b)j(h)+b\widehat{j}(\mathrm{d}_{H}h)=\mathrm{d}_{A}(b)j(h)+b\mathrm{d}_{A}(j(h))=\mathrm{d}_{A}(bj(h))=\mathrm{d}_{A}\theta^{-1}(b\otimes h),
\]
since $j$ is differentiable by \eqref{specialHmod1}. Furthermore, $\widehat{\theta^{-1}}$ is a morphism of $B\#H$-bimodules. Indeed, we can compute
\[
\begin{split}
    \widehat{\theta^{-1}}((b'\otimes h')\cdot(\beta\otimes h+b\otimes\gamma))&=\widehat{\theta^{-1}}(b'(h'_{1}\cdot\beta)\otimes h'_{2}h+b'(h'_{1}\cdot b)\otimes h'_{2}\gamma)=
    %=\widehat{\iota}(b'(h'_{1}\cdot\beta))j(h'_{2}h)+b'(h'_{1}\cdot b)\widehat{j}(h'_{2}\gamma)
b'(h'_{1}\cdot\beta)j(h'_{2}h)+b'(h'_{1}\cdot b)\widehat{j}(h'_{2}\gamma)\\&\overset{\eqref{specialHmod2}}{=}b'j(h'_{1})\beta j^{-1}(h'_{2})j(h'_{3})j(h)+b'j(h'_{1})bj^{-1}(h'_{2})\widehat{j}(h'_{3}\gamma)\\&=b'j(h')\beta j(h)+b'j(h'_{1})bj^{-1}(h'_{2})j(h'_{3})\widehat{j}(\gamma)=b'j(h')\beta j(h)+b'j(h')b\widehat{j}(\gamma)\\&=b'j(h')(\beta j(h)+b\widehat{j}(\gamma))=\theta^{-1}(b'\otimes h')(\widehat{\iota}(\beta)j(h)+b\widehat{j}(\gamma))\\&=(b'\otimes h')\cdot\widehat{\theta^{-1}}(\beta\otimes h+b\otimes\gamma),
\end{split}
\]
as well as
\[
\begin{split}
\widehat{\theta^{-1}}((\beta\otimes h)\cdot(b'\otimes h'))&=\widehat{\theta^{-1}}(\beta(h_{1}\cdot b')\otimes h_{2}h')=\beta(h_{1}\cdot b')j(h_{2}h')=\beta j(h_{1})b'j^{-1}(h_{2})j(h_{3})j(h')=\beta j(h)b'j(h')\\&=\widehat{\theta^{-1}}(\beta\otimes h)\cdot(b'\otimes h')
\end{split}
\]
and
\[
\begin{split}
\widehat{\theta^{-1}}((b\otimes h\mathrm{d}_{H}h')\cdot(b'\otimes h''))&=\widehat{\theta^{-1}}(b((h\mathrm{d}_{H}h')_{-1}\cdot b')\otimes(h\mathrm{d}_{H}h')_{0}h'')=\widehat{\theta^{-1}}(b(h_{1}h'_{1}\cdot b')\otimes h_{2}\mathrm{d}_{H}(h'_{2})h'')\\&=b(h_{1}h'_{1}\cdot b')\widehat{j}(h_{2}\mathrm{d}_{H}(h'_{2})h'')=bj(h_{1}h'_{1})b'j^{-1}(h_{2}h'_{2})j(h_{3})\mathrm{d}_{A}(j(h'_{3}))j(h'')\\&=bj(h)j(h'_{1})b'j^{-1}(h'_{2})\mathrm{d}_{A}(j(h'_{3}))j(h'')=-bj(h)j(h'_{1})b'\mathrm{d}_{A}(j^{-1}(h'_{2}))j(h'_{3})j(h'')\\&\overset{\eqref{specialHmod}}{=}bj(h)\mathrm{d}_{A}(j(h'_{1}))b'j^{-1}(h'_{2})j(h'_{3})j(h'')=bj(h)\mathrm{d}_{A}(j(h'))b'j(h'')=b\widehat{j}(h\mathrm{d}_{H}h')\theta^{-1}(b'\otimes h'')\\&=\widehat{\theta^{-1}}(b\otimes h\mathrm{d}_{H}h')\cdot(b'\otimes h'').
\end{split}
\]
Hence, we have that $\theta^{-1}$ is differentiable and $\widehat{\theta^{-1}}$ is given by $\mu^{r}_{\Omega^1(A)}\circ(\widehat{\iota}\otimes j)+\mu^{l}_{\Omega^1(A)}\circ(\iota\otimes\widehat{j})$. Now we show that $\widehat{\theta^{-1}}$ is bijective in order to obtain that it is an isomorphism of $B\#H$-bimodules. 
The right and left $A$-actions $\mu^{r}_{\Omega^1(A)}$ and $\mu^{l}_{\Omega^1(A)}$ on $\Omega^1(A)$ are injective since $\Omega^1(A)$ is a right and left torsion-free $A$-module. Moreover, $\iota$, $\widehat{\iota}$, $j$ are injective and so is $\widehat{j}$ by \eqref{specialHmod1}, thus $\mu^{r}_{\Omega^1(A)}\circ(\widehat{\iota}\otimes j)$ and $\mu^{l}_{\Omega^1(A)}\circ(\iota\otimes\widehat{j})$ are injective. It follows that $\widehat{\theta^{-1}}$ is injective, since $\widehat{\iota}(\Omega^1(B))j(H)\cap\iota(B)\widehat{j}(\Omega^1(H))=\{0\}$ holds true by \eqref{specialHmod'}. Therefore, it only remains to prove the surjectivity of $\widehat{\theta^{-1}}$. Let $a\mathrm{d}_{A}a'\in\Omega^1(A)$ be arbitrary and consider $(a_{0}j^{-1}(a_{1})\otimes a_{2})\mathrm{d}_{\#}(a'_{0}j^{-1}(a'_{1})\otimes a'_{2})\in\Omega^1(B\#H)$. Thus, we have 
\[
\begin{split}
    (a_{0}j^{-1}(a_{1})\otimes a_{2})&(\mathrm{d}_{B}(a'_{0}j^{-1}(a'_{1}))\otimes a'_{2}+a'_{0}j^{-1}(a'_{1})\otimes\mathrm{d}_{H}a'_{2})\\&=a_{0}j^{-1}(a_{1})(a_{2}\cdot\mathrm{d}_{B}(a'_{0}j^{-1}(a'_{1})))\otimes a_{3}a'_{2}+a_{0}j^{-1}(a_{1})(a_{2}\cdot a'_{0}j^{-1}(a'_{1}))\otimes a_{3}\mathrm{d}_{H}a'_{2}\\&\overset{\eqref{specialHmod2}}{=}a_{0}j^{-1}(a_{1})j(a_{2})\mathrm{d}_{B}(a'_{0}j^{-1}(a'_{1}))j^{-1}(a_{3})\otimes a_{4}a'_{2}+a_{0}j^{-1}(a_{1})j(a_{2})a'_{0}j^{-1}(a'_{1})j^{-1}(a_{3})\otimes a_{4}\mathrm{d}_{H}a'_{2}\\&=a_{0}\mathrm{d}_{B}(a'_{0}j^{-1}(a'_{1}))j^{-1}(a_{1})\otimes a_{2}a'_{2}+a_{0}a'_{0}j^{-1}(a'_{1})j^{-1}(a_{1})\otimes a_{2}\mathrm{d}_{H}a'_{2},
\end{split}
\]
which implies
\[
\begin{split}
    \widehat{\theta^{-1}}((a_{0}j^{-1}(a_{1})\otimes a_{2})\mathrm{d}_{\#}(a'_{0}j^{-1}(a'_{1})\otimes a'_{2}))&=\widehat{\theta^{-1}}(a_{0}\mathrm{d}_{B}(a'_{0}j^{-1}(a'_{1}))j^{-1}(a_{1})\otimes a_{2}a'_{2}+a_{0}a'_{0}j^{-1}(a'_{1})j^{-1}(a_{1})\otimes a_{2}\mathrm{d}_{H}a'_{2})\\&=a_{0}\mathrm{d}_{A}(a'_{0}j^{-1}(a'_{1}))j^{-1}(a_{1})j(a_{2})j(a'_{2})+a_{0}a'_{0}j^{-1}(a'_{1})j^{-1}(a_{1})j(a_{2})\mathrm{d}_{A}(j(a'_{2}))\\&=a\mathrm{d}_{A}(a'_{0}j^{-1}(a'_{1}))j(a'_{2})+aa'_{0}j^{-1}(a'_{1})\mathrm{d}_{A}(j(a'_{2}))\\&=a(\mathrm{d}_{A}(a'_{0}j^{-1}(a'_{1}))j(a'_{2})+a'_{0}j^{-1}(a'_{1})\mathrm{d}_{A}(j(a'_{2})))\\&=a\mathrm{d}_{A}(a'_{0}j^{-1}(a'_{1})j(a'_{2}))=a\mathrm{d}_{A}(a'_{0}\varepsilon(a'_{1}))\\&=a\mathrm{d}_{A}a'.
\end{split}
\]
Thus, $\widehat{\theta^{-1}}$ is also surjective and consequently it is an isomorphism of $B\#H$-bimodules. Hence, $(\Omega^1(A),\mathrm{d}_{A})\cong(\Omega^1(B\#H),\mathrm{d}_{\#})$ and the converse implication follows.\medskip

If  $(\Omega^1(A),\mathrm{d}_{A})\cong(\Omega^1(B\#H),\mathrm{d}_{\#})$ are isomorphic as FODCi through $\theta$ then, clearly, $(\Omega^1(A),\mathrm{d}_{A})$ is a right $H$-covariant FODC since this is true for $(\Omega^1(B\#H),\mathrm{d}_{\#})$. More explicitly,
we define the right $H$-coaction $\rho:\Omega^1(A)\to\Omega^1(A)\otimes H$ by $\rho:=(\widehat{\theta^{-1}}\otimes\mathrm{Id}_{H})\circ\rho'\circ\widehat{\theta}$, where $\rho'$ is the right $H$-coaction on $\Omega^1(B\#H)$, so that $\widehat{\theta}$ is automatically right $H$-colinear. Using that $\widehat{\theta}$ is left $A$-linear and that the left $B\#H$-action of $\Omega^1(B\#H)$, $\theta$ and $\mathrm{d}_{\#}$ are right $H$-colinear we explicitly obtain
\[
\begin{split}
\rho(a\mathrm{d}_{A}a')&=(\widehat{\theta^{-1}}\otimes\mathrm{Id}_{H})\rho'(\theta(a)\mathrm{d}_{\#}(\theta(a')))=(\widehat{\theta^{-1}}\otimes\mathrm{Id}_{H})(\theta(a)_{0}\mathrm{d}_{\#}(\theta(a'))_{0}\otimes\theta(a)_{1}\mathrm{d}_{\#}(\theta(a'))_{1})\\&=(\widehat{\theta^{-1}}\otimes\mathrm{Id}_{H})(\theta(a_{0})\mathrm{d}_{\#}(\theta(a')_{0})\otimes a_{1}\theta(a')_{1})=(\widehat{\theta^{-1}}\otimes\mathrm{Id}_{H})(\theta(a_{0})\mathrm{d}_{\#}(\theta(a'_{0}))\otimes a_{1}a'_{1})\\&=\theta^{-1}(\theta(a_{0}))\mathrm{d}_{A}(\theta^{-1}(\theta(a'_{0})))\otimes a_{1}a'_{1}=a_{0}\mathrm{d}_{A}a'_{0}\otimes a_{1}a'_{1}.
\end{split}
\]
Then, $\Omega^1(A)$ is in $_{A}\mm^{H}_{A}$ since $A$ is a right $H$-comodule algebra and $\mathrm{d}_{A}$ is right $H$-colinear.
\end{proof}

\begin{remark}
    Note that, assuming a non-trivial calculus $(\Omega^1(A),\mathrm{d}_{A})$ on the algebra $A$ (i.e., $\Omega^1(A)\not=\{0\}$) such that $\Omega^1(A)$ is torsion-free as left (right) $A$-module, it easily follows that $A$ can not have zero divisors.
\end{remark}

\subsection{Higher order forms}\label{Sec3.3}

In this subsection $H$ denotes a Hopf algebra and $B$ a $\sigma$-twisted left $H$-module algebra with measure $\cdot\colon H\otimes B\to B$ and $2$-cocycle $\sigma\colon H\otimes H\to B$.
In the following we introduce higher order forms on $B\#_{\sigma}H$. By generalizing Definition \ref{HmodFODC} we give the following: 

\begin{definition}\label{def:HOsmash}
Let $(\Omega^{\bullet}(B),\mathrm{d}_{B})$ be a DC on $B$. We say that it is a $\sigma$-twisted $H$-module DC if there exist linear maps $\cdot:H\otimes\Omega^{n}(B)\to\Omega^{n}(B)$ for all $n\geq1$, such that 
\begin{align*}
    h\cdot(b^{0}\mathrm{d}_{B}b^{1}\wedge\cdot\cdot\cdot\wedge\mathrm{d}_{B}b^{n})
    &=(h_{1}\cdot b^{0})(h_{2}\cdot\mathrm{d}_{B}b^{1})\wedge\cdot\cdot\cdot\wedge(h_{n+1}\cdot\mathrm{d}_{B}b^{n}),\\
    h\cdot\mathrm{d}_{B}b
    &=\mathrm{d}_{B}(h\cdot b),\\
    \mathrm{d}_{B}\circ\sigma
    &=0
\end{align*}
for all $b,b^0,\ldots,b^n\in B$ and $h\in H$.
\end{definition}
Let us discuss some immediate consequences of the above definition.
\begin{remark}
From the first two equations in Definition \ref{def:HOsmash} we obtain that the linear maps $\cdot:H\otimes\Omega^{n}(B)\to\Omega^{n}(B)$, if they exist, are uniquely determined by
$$
h\cdot(b^{0}\mathrm{d}_{B}b^{1}\wedge\cdot\cdot\cdot\wedge\mathrm{d}_{B}b^{n})=(h_{1}\cdot b^{0})\mathrm{d}_{B}(h_{2}\cdot b^{1})\wedge\cdot\cdot\cdot\wedge\mathrm{d}_{B}(h_{n+1}\cdot b^{n}).
$$
Clearly $1_{H}\cdot\beta=\beta$ for all $\beta\in\Omega^\bullet(B)$ and, as in the proof of Lemma \ref{twistedHmodcal}, one shows that 
\begin{equation}\label{actionDC}
h\cdot(h'\cdot\beta)=\sigma(h_{1}\otimes h'_{1})(h_{2}h'_{2}\cdot\beta)\sigma^{-1}(h_{3}\otimes h'_{3})\ \text{for all}\ \beta\in\Omega^{\bullet}(B)\ \text{and}\ h,h'\in H.
\end{equation}
In addition,   
\begin{equation}\label{actionDC2}
    h\cdot(\beta\wedge\beta')=(h_{1}\cdot\beta)\wedge(h_{2}\cdot\beta')\ \text{for all}\ h\in H,\ \beta\in\Omega^{n}(B)\ \text{and}\ \beta'\in\Omega^{m}(B)
\end{equation}
holds true for arbitrary $n,m\geq0$. Hence, we have that the graded algebra $\Omega^{\bullet}(B)$ is a $\sigma$-twisted left $H$-module algebra with measure given by the collection of the linear maps $\cdot:H\otimes\Omega^n(B)\to\Omega^n(B)$, for $n\geq0$.
\end{remark}

Given a bicovariant DC $\Omega^\bullet(H)$
it is proven in \cite[Lemma 5.4]{SchDC} that $\Delta$ is differentiable if and only if $\Omega^{\bullet}(H)$ is a differential graded bialgebra with $\Omega^{0}(H)=H$ as bialgebras. Thus, in this case, we have $\Omega^{\bullet}(\Delta):\Omega^{\bullet}(H)\to\Omega^{\bullet}_{\otimes}(H\otimes H)$.
For an arbitrary $\gamma\in\Omega^{\bullet}(H)$ we write $\Omega^{\bullet}(\Delta)(\gamma)=\gamma_{[1]}\otimes\gamma_{[2]}$ and, denoting by $\pi^{kl}:\Omega^{k+l}(H\otimes H)\to\Omega^{k}(H)\otimes\Omega^{l}(H)$ the canonical projection, for a homogeneous element $\gamma\in\Omega^{j}(H)$ and $k+l=j$ we write $\pi^{kl}\Omega^{j}(\Delta)(\gamma)=:\gamma_{<1,k>}\otimes\gamma_{<2,l>}$, so that $\Omega^{j}(\Delta)(\gamma)=\sum_{k+l=j}{\gamma_{<1,k>}\otimes\gamma_{<2,l>}}$ in accordance with \cite{PS}. In terms of this notation, the fact that $\Omega^{\bullet}(\Delta)$ is an algebra map reads
\begin{equation}\label{p1}
(\gamma\wedge\gamma')_{[1]}\otimes(\gamma\wedge\gamma')_{[2]}
=\sum_{\substack{k+l=j\\ k'+l'=j'}}(-1)^{lk'}\gamma_{<1,k>}\wedge\gamma'_{<1,k'>}\otimes\gamma_{<2,l>}\wedge\gamma'_{<2,l'>}
\end{equation}
for every $\gamma\in\Omega^{j}(H)$ and $\gamma'\in\Omega^{j'}(H)$. 
Applying the projection $\pi^{0,k,l}$ to the equation $(\mathrm{Id}\otimes\Omega^\bullet(\Delta))\Omega^\bullet(\Delta)(\gamma)=(\Omega^\bullet(\Delta)\otimes\mathrm{Id})\Omega^\bullet(\Delta)(\gamma)$ gives
\begin{equation}\label{p2}
\gamma_{-1}\otimes\gamma_{0,<1,k>}\otimes\gamma_{0,<2,l>}
=\gamma_{<1,k>,-1}\otimes\gamma_{<1,k>,0}\otimes\gamma_{<2,l>}
\end{equation}
and we have a similar formula
for the projection $\pi^{k,l,0}$.

We now define higher order forms on $B\#_{\sigma}H$ by generalizing Theorem \ref{canFODC}.

\begin{theorem}\label{thm:higerorderforms}
Let 
$(\Omega^{\bullet}(B),\mathrm{d}_{B})$ be a $\sigma$-twisted $H$-module DC on $B$ and $(\Omega^{\bullet}(H),\mathrm{d}_{H})$ a DC on $H$ with respect to which $\Delta:H\to H\otimes H$ is differentiable. Let us define
\[
\Omega^{n}(B\#_{\sigma}H):=\bigoplus_{i=0}^{n}{\Omega^{n-i}(B)\otimes\Omega^{i}(H)}\qquad \text{for all}\ n\geq0 
\]
and set
\begin{align*}
    (\omega\otimes\gamma)\wedge(\omega'\otimes\gamma'):&=(-1)^{jk}(\omega\wedge(\gamma_{-2}\cdot\omega')\sigma(\gamma_{-1}\otimes\gamma'_{-1}))\otimes(\gamma_{0}\wedge\gamma'_{0}),\\
    \mathrm{d}_{\#_\sigma}(\omega\otimes\gamma):
    &=\mathrm{d}_{B}\omega\otimes\gamma+(-1)^{i}\omega\otimes\mathrm{d}_{H}\gamma
\end{align*}
for $\omega\in\Omega^{i}(B)$, $\gamma\in\Omega^{j}(H)$ and $\omega'\in\Omega^{k}(B)$. Then $(\Omega^{\bullet}(B\#_{\sigma}H),\mathrm{d}_{\#_\sigma})$ is a right $H$-covariant DC on $B\#_{\sigma}H$. Furthermore, the right $H$-coaction $\rho:B\#_{\sigma}H\to(B\#_{\sigma}H)\otimes H$ is differentiable with differential
\begin{align*}
    \Omega^{\bullet}(\rho):\Omega^{\bullet}(B\#_{\sigma}H)&\to\Omega^{\bullet}((B\#_{\sigma}H)\otimes H)\\
    \omega\otimes\gamma&\mapsto\omega\otimes\gamma_{[1]}\otimes\gamma_{[2]}.
\end{align*}
\end{theorem}

\begin{proof}
First, note that $\wedge$ is a map of degree 0, since this is the case for the products of $\Omega^{\bullet}(B)$ and $\Omega^{\bullet}(H)$, $\lambda(\Omega^{n}(H))\subseteq H\otimes\Omega^{n}(H)$ and $H\cdot\Omega^n(B)\subseteq\Omega^{n}(B)$ for all $n\geq0$.
Now, we show the associativity of the product. Given $\omega\in\Omega^{i}(B)$, $\omega'\in\Omega^{i'}(B)$, $\omega''\in\Omega^{i''}(B)$, $\gamma\in\Omega^{j}(H)$, $\gamma'\in\Omega^{j'}(H)$ and $\gamma''\in\Omega^{j''}(H)$, we obtain
\[
\begin{split}
    &((\omega\otimes\gamma)\wedge(\omega'\otimes\gamma'))\wedge(\omega''\otimes\gamma'')=(-1)^{ji'}(\omega\wedge(\gamma_{-2}\cdot\omega')\sigma(\gamma_{-1}\otimes\gamma'_{-1})\otimes\gamma_{0}\wedge\gamma'_{0})\wedge(\omega''\otimes\gamma'')\\
    &=(-1)^{ji'+(j+j')i''}\omega\wedge(\gamma_{-2}\cdot\omega')\sigma(\gamma_{-1}\otimes\gamma'_{-1})\wedge((\gamma_{0}\wedge\gamma'_{0})_{-2}\cdot\omega'')\sigma((\gamma_{0}\wedge\gamma'_{0})_{-1}\otimes\gamma''_{-1})\otimes(\gamma_{0}\wedge\gamma'_{0})_{0}\wedge\gamma''_{0}\\
    &=(-1)^{j(i'+i'')+j'i''}\omega\wedge(\gamma_{-4}\cdot\omega')\sigma(\gamma_{-3}\otimes\gamma'_{-3})\wedge(\gamma_{-2}\gamma'_{-2}\cdot\omega'')\sigma(\gamma_{-1}\gamma'_{-1}\otimes\gamma''_{-1})\otimes\gamma_{0}\wedge\gamma'_{0}\wedge\gamma''_{0}\\
    &=(-1)^{j(i'+i'')+j'i''}\omega\wedge(\gamma_{-6}\cdot\omega')\wedge\sigma(\gamma_{-5}\otimes\gamma'_{-5})(\gamma_{-4}\gamma'_{-4}\cdot\omega'')\sigma^{-1}(\gamma_{-3}\otimes\gamma'_{-3})\sigma(\gamma_{-2}\otimes\gamma'_{-2})\sigma(\gamma_{-1}\gamma'_{-1}\otimes\gamma''_{-1})\\&\quad\otimes\gamma_{0}\wedge\gamma'_{0}\wedge\gamma''_{0}\\
    &\overset{\eqref{actionDC}}{=}(-1)^{j(i'+i'')+j'i''}\omega\wedge(\gamma_{-4}\cdot\omega')\wedge(\gamma_{-3}\cdot(\gamma'_{-3}\cdot\omega''))(\gamma_{-2}\cdot\sigma(\gamma'_{-2}\otimes\gamma''_{-2}))\sigma(\gamma_{-1}\otimes\gamma'_{-1}\gamma''_{-1})\otimes\gamma_{0}\wedge\gamma'_{0}\wedge\gamma''_{0}\\
    &\overset{\eqref{actionDC2}}{=}(-1)^{j(i'+i'')+j'i''}\omega\wedge(\gamma_{-2}\cdot(\omega'\wedge(\gamma'_{-2}\cdot\omega'')\sigma(\gamma'_{-1}\otimes\gamma''_{-1})))\sigma(\gamma_{-1}\otimes(\gamma'_{0}\wedge\gamma''_{0})_{-1})\otimes\gamma_{0}\wedge(\gamma'_{0}\wedge\gamma''_{0})_{0}\\
    &=(\omega\otimes\gamma)\wedge((-1)^{j'i''}\omega'\wedge(\gamma'_{-2}\cdot\omega'')\sigma(\gamma'_{-1}\otimes\gamma''_{-1})\otimes\gamma'_{0}\wedge\gamma''_{0})\\
    &=(\omega\otimes\gamma)\wedge((\omega'\otimes\gamma')\wedge(\omega''\otimes\gamma'')).
\end{split}
\]
Furthermore, $\mathrm{d}_{\#_\sigma}$ satisfies the Leibniz rule, indeed 
\[
\begin{split}
    &\mathrm{d}_{\#_\sigma}(\omega\otimes\gamma)\wedge(\omega'\otimes\gamma')+(-1)^{i+j}(\omega\otimes\gamma)\wedge\mathrm{d}_{\#_\sigma}(\omega'\otimes\gamma')\\&=(\mathrm{d}_{B}\omega\otimes\gamma)\wedge(\omega'\otimes\gamma')+(-1)^{i}(\omega\otimes\mathrm{d}_{H}\gamma)\wedge(\omega'\otimes\gamma')+(-1)^{i+j}(\omega\otimes\gamma)\wedge(\mathrm{d}_{B}\omega'\otimes\gamma')+(-1)^{i+j+i'}(\omega\otimes\gamma)\wedge(\omega'\otimes\mathrm{d}_{H}\gamma')\\&=(-1)^{ji'}\mathrm{d}_{B}\omega\wedge(\gamma_{-2}\cdot\omega')\sigma(\gamma_{-1}\otimes\gamma'_{-1})\otimes\gamma_{0}\wedge\gamma'_{0}+(-1)^{i+(j+1)i'}\omega\wedge((\mathrm{d}_{H}\gamma)_{-2}\cdot\omega')\sigma((\mathrm{d}_{H}\gamma)_{-1}\otimes\gamma'_{-1})\otimes(\mathrm{d}_{H}\gamma)_{0}\wedge\gamma'_{0}\\&+(-1)^{i+j+j(i'+1)}\omega\wedge(\gamma_{-2}\cdot\mathrm{d}_{B}\omega')\sigma(\gamma_{-1}\otimes\gamma'_{-1})\otimes\gamma_{0}\wedge\gamma'_{0}+(-1)^{i+j+i'+ji'}\omega\wedge(\gamma_{-2}\cdot\omega')\sigma(\gamma_{-1}\otimes(\mathrm{d}_{H}\gamma')_{-1})\otimes\gamma_{0}\wedge(\mathrm{d}_{H}\gamma')_{0}\\&=(-1)^{ji'}\mathrm{d}_{B}\omega\wedge(\gamma_{-2}\cdot\omega')\sigma(\gamma_{-1}\otimes\gamma'_{-1})\otimes\gamma_{0}\wedge\gamma'_{0}+(-1)^{ji'+i+i'}(\omega\wedge(\gamma_{-2}\cdot\omega')\sigma(\gamma_{-1}\otimes\gamma'_{-1})\otimes\mathrm{d}_{H}\gamma_{0}\wedge\gamma'_{0}\\&+(-1)^{i+ji'}\omega\wedge\mathrm{d}_{B}((\gamma_{-2}\cdot\omega')\sigma(\gamma_{-1}\otimes\gamma'_{-1}))\otimes\gamma_{0}\wedge\gamma'_{0}+(-1)^{i+j+i'+ji'}\omega\wedge(\gamma_{-2}\cdot\omega')\sigma(\gamma_{-1}\otimes\gamma'_{-1})\otimes\gamma_{0}\wedge\mathrm{d}_{H}\gamma'_{0}\\&=(-1)^{ji'}(\mathrm{d}_{B}\omega\wedge(\gamma_{-2}\cdot\omega')\sigma(\gamma_{-1}\otimes\gamma'_{-1})\otimes\gamma_{0}\wedge\gamma'_{0}+(-1)^{i}\omega\wedge\mathrm{d}_{B}((\gamma_{-2}\cdot\omega')\sigma(\gamma_{-1}\otimes\gamma'_{-1}))\otimes\gamma_{0}\wedge\gamma'_{0})\\&+(-1)^{ji'+i+i'}(\omega\wedge(\gamma_{-2}\cdot\omega')\sigma(\gamma_{-1}\otimes\gamma'_{-1})\otimes\mathrm{d}_{H}\gamma_{0}\wedge\gamma'_{0}+(-1)^{j}\omega\wedge(\gamma_{-2}\cdot\omega')\sigma(\gamma_{-1}\otimes\gamma'_{-1})\otimes\gamma_{0}\wedge\mathrm{d}_{H}\gamma'_{0})\\&=(-1)^{ji'}(\mathrm{d}_{B}(\omega\wedge(\gamma_{-2}\cdot\omega')\sigma(\gamma_{-1}\otimes\gamma'_{-1}))\otimes\gamma_{0}\wedge\gamma'_{0})+(-1)^{ji'+i+i'}(\omega\wedge(\gamma_{-2}\cdot\omega')\sigma(\gamma_{-1}\otimes\gamma'_{-1})\otimes\mathrm{d}_{H}(\gamma_{0}\wedge\gamma'_{0}))\\&=(-1)^{ji'}(\mathrm{d}_{B}(\omega\wedge(\gamma_{-2}\cdot\omega')\sigma(\gamma_{-1}\otimes\gamma'_{-1}))\otimes\gamma_{0}\wedge\gamma'_{0}+(-1)^{i+i'}\omega\wedge(\gamma_{-2}\cdot\omega')\sigma(\gamma_{-1}\otimes\gamma'_{-1})\otimes\mathrm{d}_{H}(\gamma_{0}\wedge\gamma'_{0}))\\&=(-1)^{ji'}\mathrm{d}_{\#_\sigma}(\omega\wedge(\gamma_{-2}\cdot\omega')\sigma(\gamma_{-1}\otimes\gamma'_{-1})\otimes\gamma_{0}\wedge\gamma'_{0})=\mathrm{d}_{\#_\sigma}((\omega\otimes\gamma)\wedge(\omega'\otimes\gamma')),
\end{split}
\]
where in the third equality we use second and third equations of Definition \ref{def:HOsmash}, while the fact that  $\mathrm{d}_{\#_\sigma}^{2}=0$ follows in complete analogy to the tensor product DC. Hence $(\Omega^{\bullet}(B\#_{\sigma}H),\wedge,\mathrm{d}_{\#_\sigma})$ is a differential graded algebra such that $\Omega^0(B\#_{\sigma}H)=B\#_{\sigma}H$. Moreover, from Theorem \ref{canFODC} we know that $\Omega^1(B\#_{\sigma}H)$ is spanned by elements of the form $(b\otimes h)\mathrm{d}_{\#_\sigma}(b'\otimes h')$, for $b,b'\in B$ and $h,h'\in H$. Now, suppose that for a fixed $n>1$ $\Omega^{n-1}(B\#_{\sigma}H)$ is spanned by elements of the form $(b^{0}\otimes h^{0})\mathrm{d}_{\#_\sigma}(b^{1}\otimes h^{1})\wedge\cdot\cdot\cdot\wedge\mathrm{d}_{\#_\sigma}(b^{n-1}\otimes h^{n-1})$ for $b^{i}\in B$ and $h^{i}\in H$. A generic element in $\Omega^{n}(B\#_{\sigma}H)$ is a sum of elements in $\Omega^{k}(B)\otimes\Omega^{l}(H)$ with $k,l\in\{0,\ldots,n\}$ such that $k+l=n$ and then, by using the "surjectivity condition" of $\Omega^{k}(B)$ and $\Omega^{l}(H)$, sum of elements of the form $(b^{0}\mathrm{d}_{B}b^{1}\wedge\mathrm{d}_{B}b^{2}\wedge\cdot\cdot\cdot\wedge\mathrm{d}_{B}b^{k})\otimes(h^{0}\mathrm{d}_{H}h^{1}\wedge\cdot\cdot\cdot\wedge\mathrm{d}_{H}h^{l})$. Then, we have 
\[
\begin{split}
    (b^{0}\mathrm{d}_{B}b^{1}\otimes1_{H})&\wedge((\mathrm{d}_{B}b^{2}\wedge\cdot\cdot\cdot\wedge\mathrm{d}_{B}b^{k})\otimes(h^{0}\mathrm{d}_{H}h^{1}\wedge\cdot\cdot\cdot\wedge\mathrm{d}_{H}h^{l}))\\
    &=b^{0}\mathrm{d}_{B}b^{1}\wedge(1_{H}\cdot\mathrm{d}_{B}b^{2}\wedge\cdot\cdot\cdot\wedge\mathrm{d}_{B}b^{k})\sigma(1_{H}\otimes(h^{0}\mathrm{d}_{H}h^{1}\wedge\cdot\cdot\cdot\wedge\mathrm{d}_{H}h^{l})_{-1})\otimes1_{H}(h^{0}\mathrm{d}_{H}h^{1}\wedge\cdot\cdot\cdot\wedge\mathrm{d}_{H}h^{l})_{0}\\
    &=(b^{0}\mathrm{d}_{B}b^{1}\wedge\mathrm{d}_{B}b^{2}\wedge\cdot\cdot\cdot\wedge\mathrm{d}_{B}b^{k})\otimes(h^{0}\mathrm{d}_{H}h^{1}\wedge\cdot\cdot\cdot\wedge\mathrm{d}_{H}h^{l})
\end{split}
\]
and $(\mathrm{d}_{B}b^{2}\wedge\cdot\cdot\cdot\wedge\mathrm{d}_{B}b^{k})\otimes(h^{0}\mathrm{d}_{H}h^{1}\wedge\cdot\cdot\cdot\wedge\mathrm{d}_{H}h^{l})$ is in $\Omega^{n-1}(B\#_{\sigma}H)$, which is by hypothesis a finite sum of elements of the aforementioned form. Since $b^{0}\mathrm{d}_{B}b^{1}\otimes1_{H}=(b^{0}\otimes1_{H})\mathrm{d}_{\#_\sigma}(b^{1}\otimes1_{H})$, we can conclude by induction that $\Omega^{n}(B\#_{\sigma}H)$ is spanned by elements of the form $(b^{0}\otimes h^{0})\mathrm{d}_{\#_\sigma}(b^{1}\otimes h^{1})\wedge\cdot\cdot\cdot\wedge\mathrm{d}_{\#_\sigma}(b^{n}\otimes h^{n})$ with $b^{i}\in B$ and $h^{i}\in H$. Thus, we have obtained that $(\Omega^{\bullet}(B\#_{\sigma}H),\mathrm{d}_{\#_\sigma})$ is a DC on $B\#_{\sigma}H$ and it is right $H$-covariant since the DC $(\Omega^\bullet(H),\mathrm{d}_H)$ is right covariant. 

Finally, we show that the algebra map $\rho:B\#_{\sigma}H\to(B\#_{\sigma}H)\otimes H$, $b\otimes h\mapsto b\otimes h_{1}\otimes h_{2}$ is differentiable. Writing $\Omega^{\bullet}(\Delta)(\gamma)=\gamma_{[1]}\otimes\gamma_{[2]}$, we define $\Omega^{\bullet}(\rho):\Omega^{\bullet}(B\#_{\sigma}H)\to\Omega^{\bullet}((B\#_{\sigma}H)\otimes H)$ as $\Omega^{\bullet}(\rho)(\omega\otimes\gamma)=\omega\otimes\gamma_{[1]}\otimes\gamma_{[2]}$ such that $\Omega^{0}(\rho)=\rho$.

Then, $\Omega^{\bullet}(\rho)$ is an algebra map. Indeed, given $\omega$, $\omega'$, $\gamma$, $\gamma'$ as before, we obtain
\[
\begin{split}
&\Omega^{\bullet}(\rho)((\omega\otimes\gamma)\wedge(\omega'\otimes\gamma'))=(-1)^{ji'}\Omega^{\bullet}(\rho)(\omega\wedge(\gamma_{-2}\cdot\omega')\sigma(\gamma_{-1}\otimes\gamma'_{-1})\otimes\gamma_{0}\wedge\gamma'_{0})\\
&=(-1)^{ji'}\omega\wedge(\gamma_{-2}\cdot\omega')\sigma(\gamma_{-1}\otimes\gamma'_{-1})\otimes(\gamma_{0}\wedge\gamma'_{0})_{[1]}\otimes(\gamma_{0}\wedge\gamma'_{0})_{[2]}\\
&\overset{\eqref{p1}}{=}\sum_{\substack{k+l=j\\ k'+l'=j'}}{(-1)^{ji'+lk'}\omega\wedge(\gamma_{-2}\cdot\omega')\sigma(\gamma_{-1}\otimes\gamma'_{-1})\otimes(\gamma_{0,<1,k>}\wedge\gamma'_{0,<1,k'>})\otimes(\gamma_{0,<2,l>}\wedge\gamma'_{0,<2,l'>}})\\
&=\sum_{\substack{k+l=j\\ k'+l'=j'}}{(-1)^{ki'+l(i'+k')}\omega\wedge(\gamma_{<1,k>,-2}\cdot\omega')\sigma(\gamma_{<1,k>,-1}\otimes\gamma'_{<1,k'>,-1})\otimes(\gamma_{<1,k>,0}\wedge\gamma'_{<1,k'>,0})\otimes(\gamma_{<2,l>}\wedge\gamma'_{<2,l'>})}\\
&=\sum_{\substack{k+l=j\\ k'+l'=j'}}{(-1)^{l(i'+k')}(\omega\otimes\gamma_{<1,k>})\wedge(\omega'\otimes\gamma'_{<1,k'>})\otimes(\gamma_{<2,l>}\wedge\gamma'_{<2,l'>})}\\
&=\sum_{\substack{k+l=j\\ k'+l'=j'}}(\omega\otimes\gamma_{<1,k>}\otimes\gamma_{<2,l>})\tilde{\wedge}(\omega'\otimes\gamma'_{<1,k'>}\otimes\gamma'_{<2,l'>})\\
&=\Omega^\bullet(\rho)(\omega\otimes\gamma)\tilde{\wedge}\Omega^\bullet(\rho)(\omega'\otimes\gamma'),
\end{split}
\]
where we used 
\eqref{p2} for $\gamma$ and $\gamma'$
in the fourth equation and the product
\begin{align*}
    \tilde{\wedge}\colon\big(\Omega^\bullet(B\#_\sigma H)\otimes\Omega^\bullet(H)\big)^{\otimes 2}&\to\Omega^\bullet(B\#_\sigma H)\otimes\Omega^\bullet(H)\\
    (\vartheta\otimes\gamma)\otimes(\vartheta'\otimes\gamma')&\mapsto(-1)^{|\vartheta'||\gamma|}(\vartheta\wedge\vartheta')\otimes(\gamma\wedge\gamma')
\end{align*}
in the sixth equation. Furthermore, from $\mathrm{d}_{\otimes}\circ\Omega^{\bullet}(\Delta)=\Omega^{\bullet}(\Delta)\circ\mathrm{d}_{H}$, we obtain that
\[
\begin{split}
\Omega^{\bullet}(\rho)\mathrm{d}_{\#_\sigma}(\omega\otimes\gamma)&=\Omega^{\bullet}(\rho)(\mathrm{d}_{B}\omega\otimes\gamma+(-1)^{|\omega|}\omega\otimes\mathrm{d}_{H}\gamma)=\mathrm{d}_{B}\omega\otimes\gamma_{[1]}\otimes\gamma_{[2]}+(-1)^{|\omega|}\omega\otimes(\mathrm{d}_{H}\gamma)_{[1]}\otimes(\mathrm{d}_{H}\gamma)_{[2]}\\&=\mathrm{d}_{B}\omega\otimes\gamma_{[1]}\otimes\gamma_{[2]}+(-1)^{|\omega|}\omega\otimes\mathrm{d}_{\otimes}(\gamma_{[1]}\otimes\gamma_{[2]})=\mathrm{d}_{\otimes}(\omega\otimes\gamma_{[1]}\otimes\gamma_{[2]})\\&=\mathrm{d}_{\otimes}\Omega^{\bullet}(\rho)(\omega\otimes\gamma),
\end{split}
\]
hence $\Omega^{\bullet}(\rho)\circ\mathrm{d}_{\#_\sigma}=\mathrm{d}_{\otimes}\circ\Omega^{\bullet}(\rho)$. Thus, $\rho$ is differentiable. 
\end{proof}

As one easily verifies, the above crossed product DC $(\Omega^{\bullet}(B\#_{\sigma}H),\mathrm{d}_{\#_{\sigma}})$ reduces to the smash product DC $(\Omega^{\bullet}(B\#H),\mathrm{d}_{\#})$ of \cite[Theorem 3.5]{PS} in case $\sigma$ is the trivial 2-cocycle. Moreover, we observe that the DC $\Omega^{\bullet}(B\#_{\sigma}H)$ restricted to zero-forms on $H$ coincides with the crossed product algebra $\Omega^{\bullet}(B)\#_{\sigma}H$. The latter can be understood as a crossed product of differential graded algebras if we consider the trivial DC on $H$.

\subsection{Quantum principal bundle and connections}\label{Sec3.4}

Now we study some constructions related to the FODC on the crossed product algebra $B\#_{\sigma}H$. In this subsection we show that there is a (regular and strong) QPB on $B\#_{\sigma}H$ with respect to this calculus. Afterwards, we study bimodule covariant derivatives on associated bundles of the crossed product algebra. Furthermore, we compute the noncommutative de Rham cohomology of $B\#_{\sigma}H$ from the cohomologies of $B$ and $H$ via the Leray--Serre spectral sequence. The latter employs higher order vertical maps and Atiyah sequences.

We consider $B\#_\sigma H$ as right $H$-comodule algebra via $\rho:=\mathrm{Id}_{B}\otimes\Delta$ and the right $H$-covariant FODC $(\Omega^1(B\#_{\sigma}H),\mathrm{d}_{\#_\sigma})$ on it given in Theorem \ref{canFODC}, when $(\Omega^1(B),\mathrm{d}_{B})$ is a $\sigma$-twisted $H$-module FODC on $B$ and $(\Omega^1(H),\mathrm{d}_{H})$ is a bicovariant FODC on $H$. Recalling that $(B\#_{\sigma}H)^{\mathrm{co}H}=B\otimes\Bbbk1_{H}$, the induced pullback calculus on $B\otimes\Bbbk1_{H}$ is given by
\[
\Omega^1(B\otimes\Bbbk1_{H})=\mathrm{span}_{\Bbbk}\{(b\otimes1_{H})\mathrm{d}_{\#_\sigma}(b'\otimes1_{H})\ |\ b,b'\in B\}=\mathrm{span}_{\Bbbk}\{b\mathrm{d}_{B}b'\otimes1_{H}\ |\ b,b'\in B\}=\Omega^1(B)\otimes\Bbbk1_{H}
\]
with differential given by the restriction and corestriction of $\mathrm{d}_{\#_\sigma}$. The horizontal forms are 
$$
\Omega^1(B\#_{\sigma}H)_{\mathrm{hor}}:=(B\#_{\sigma}H)(\Omega^1(B)\otimes\Bbbk1_{H})(B\#_{\sigma}H)=\Omega^1(B)\otimes H.
$$
The last equality holds, since clearly $\Omega^1(B\#_{\sigma}H)_{\mathrm{hor}}\subseteq\Omega^1(B)\otimes H$ and, given $b\mathrm{d}_{B}b'\otimes h$ for $b,b'\in B$ and $h\in H$
\[
((b\otimes1_{H})\cdot(\mathrm{d}_{B}b'\otimes1_{H}))\cdot(1_{B}\otimes h)=(b\mathrm{d}_{B}b'\otimes1_{H})\cdot(1_{B}\otimes h)=b\mathrm{d}_{B}b'\varepsilon(h_{1})\otimes h_{2}=b\mathrm{d}_{B}b'\otimes h.
\]
We prove that the horizontal forms fit into a short exact sequence, forming a QPB. This QPB is automatically regular, in the sense that it descends from a universal QPB with a strong connection.

\begin{theorem}\label{Theorem:QPB}
    Given $(\Omega^1(H),\mathrm{d}_{H})$ a bicovariant FODC  on $H$ and $(\Omega^1(B),\mathrm{d}_{B})$ a $\sigma$-twisted $H$-module FODC on $B$, there is a regular and strong QPB on the crossed product algebra $B\#_{\sigma}H$. Explicitly,
\begin{equation}\label{Atiyahseq}
0\longrightarrow\Omega^1(B)\otimes H\overset{\iota_1}\longrightarrow\Omega^1(B\#_{\sigma}H)\overset{\mathrm{ver}}\longrightarrow(B\#_{\sigma}H)\otimes{}^{\mathrm{co}H}\Omega^1(H)\longrightarrow0
\end{equation}
    is an exact sequence in $_{B\#_{\sigma}H}\mm$, where $\iota_1$ is the inclusion and $(\Omega^1(B\#_{\sigma}H),\mathrm{d}_{\#_\sigma})$ is the right $H$-covariant FODC on $B\#_{\sigma}H$, constructed in Theorem \ref{canFODC}.
\end{theorem}

\begin{proof}
Since we have shown in Theorem \ref{canFODC} that the right $H$-coaction $\rho:B\#_{\sigma}H\to(B\#_{\sigma}H)\otimes H$ is differentiable we have that $\mathrm{ver}=\phi\circ\pi'_{2}\circ\widehat{\rho}$ is well-defined, where 
\begin{align*}
    \widehat{\rho}:\Omega^1(B\#_{\sigma}H)
    &\to\Omega^1((B\#_{\sigma}H)\otimes H),\\
    \beta\otimes h+b\otimes\gamma&\mapsto\beta\otimes h_{1}\otimes h_{2}+b\otimes\gamma_{0}\otimes\gamma_{1}+b\otimes\gamma_{-1}\otimes\gamma_{0},
\end{align*}
$\pi'_{2}:\Omega^1((B\#_{\sigma}H)\otimes H)\to(B\#_{\sigma}H)\otimes\Omega^1(H)$ is the canonical projection and 
$$
\phi:(B\#_{\sigma}H)\square_{H}\Omega^1(H)\to(B\#_{\sigma}H)\otimes{}^{\mathrm{co}H}\Omega^1(H),\qquad
b\otimes h\otimes\gamma\mapsto b\otimes h\otimes S(\gamma_{-1})\gamma_{0}
$$
is an isomorphism. Thus, we obtain $\mathrm{ver}(\beta\otimes h+b\otimes\gamma)=b\otimes\gamma_{-2}\otimes S(\gamma_{-1})\gamma_{0}$, which we can rewrite as $\mathrm{ver}=p\circ\pi_{2}$, where $\pi_{2}:\Omega^1(B\#_{\sigma}H)\to B\otimes\Omega^1(H)$ is the canonical projection and $p:B\otimes\Omega^1(H)\to(B\#_{\sigma}H)\otimes{}^{\mathrm{co}H}\Omega^1(H)$ is the linear map defined by $p(b\otimes\gamma):=b\otimes\gamma_{-2}\otimes S(\gamma_{-1})\gamma_{0}$.
Observe that $p$ is left $B\#_{\sigma}H$-linear, since we have 
\[
\begin{split}
p((b'\otimes h')\cdot(b\otimes\gamma))&=p(b'(h'_{1}\cdot b)\sigma(h'_{2}\otimes\gamma_{-1})\otimes h'_{3}\gamma_{0})\\&=b'(h'_{1}\cdot b)\sigma(h'_{2}\otimes\gamma_{-3})\otimes h'_{3}\gamma_{-2}\otimes S(h'_{4}\gamma_{-1})h'_{5}\gamma_{0}\\&=%b'(h'_{1}\cdot b)\sigma(h'_{2}\otimes\gamma_{-3})\otimes h'_{3}\gamma_{-2}\otimes S(\gamma_{-1})S(h'_{4})h'_{5}\gamma_{0}\\&=
b'(h'_{1}\cdot b)\sigma(h'_{2}\otimes\gamma_{-3})\otimes h'_{3}\gamma_{-2}\otimes S(\gamma_{-1})\gamma_{0}\\&=(b'\otimes h')(b\otimes\gamma_{-2})\otimes S(\gamma_{-1})\gamma_{0}
\end{split}
\]
for all $b,b'\in B$, $h'\in H$ and $\gamma\in\Omega^1(H)$.

We show that $p$ is bijective in order to obtain that $\mathrm{ver}$ is surjective and that $\mathrm{ker}(\mathrm{ver})=\mathrm{ker}(\pi_{2})=\Omega^1(B)\otimes H$. We define the linear map $g:(B\#_{\sigma}H)\otimes{}^{\mathrm{co}H}\Omega^1(H)\to B\otimes\Omega^1(H)$ by $g(b\otimes h\otimes\gamma):=b\otimes h\gamma$. 
The morphism $g$ is the inverse of $p$. Indeed, given $b\otimes h\otimes\gamma\in(B\#_{\sigma}H)\otimes{}^{\mathrm{co}H}\Omega^1(H)$, we have
\[
\begin{split}
pg(b\otimes h\otimes\gamma)&=p(b\otimes h\gamma)=b\otimes(h\gamma)_{-2}\otimes S((h\gamma)_{-1})(h\gamma)_{0}=b\otimes h_{1}\gamma_{-2}\otimes S(h_{2}\gamma_{-1})h_{3}\gamma_{0}\\&=b\otimes h_{1}\otimes S(h_{2})h_{3}\gamma=b\otimes h\otimes\gamma
\end{split}
\]
since $\gamma$ is in $^{\mathrm{co}H}\Omega^1(H)$ and, given $b\otimes\gamma\in B\otimes\Omega^1(H)$, we obtain
\[
gp(b\otimes\gamma)=g(b\otimes\gamma_{-2}\otimes S(\gamma_{-1})\gamma_{0})=b\otimes\gamma_{-2}S(\gamma_{-1})\gamma_{0}=b\otimes\varepsilon(\gamma_{-1})\gamma_{0}=b\otimes\gamma. 
\]
Thus, $p$ is bijective and the sequence \eqref{Atiyahseq} is exact in $_{B\#_{\sigma}H}\mm$, hence we have a QPB on the crossed product algebra $B\#_{\sigma}H$. Since crossed product algebras correspond to cleft extensions, the regularity of the QPB follows.
Finally, observe that 
$(\Omega^1(B)\otimes H)^{\mathrm{co}H}=\Omega^1(B)\otimes H^{\mathrm{co}H}=\Omega^1(B)\otimes\Bbbk1_{H}$ and then, using \cite[Corollary 5.53]{BM}, the QPB 
is strong 
(the bijectivity of the antipode of $H$ is needed here).
\end{proof}
\begin{remark}
    Note that clearly $(\Omega^1(B)\otimes\Bbbk1_{H})(B\#_{\sigma}H)=\Omega^1(B)\otimes H$ since every $\beta\otimes h\in\Omega^1(B)\otimes H$ can be written as $(\beta\otimes1_{H})\cdot(1_{B}\otimes h)$. Since the QPB of Theorem \ref{Theorem:QPB} is strong we also obtain $(B\#_{\sigma}H)(\Omega^1(B)\otimes\Bbbk1_{H})=\Omega^1(B)\otimes H$.
\end{remark}

We show that the above QPB exhibits a canonical strong connection. We recall that a connection is a section $c\colon(B\#_\sigma H)\otimes{}^{\mathrm{co}H}\Omega^1(H)\to\Omega^1(B\#_\sigma H)$ in ${}_{B\#_\sigma H}\mathfrak{M}^H$ of the vertical map and it is strong if $(\mathrm{Id}-c\circ\mathrm{ver})\mathrm{d}_{\#_\sigma}(B\#_{\sigma}H)\subseteq\Omega^1(B)(B\#_{\sigma}H)$.
\begin{proposition}\label{Prop:SC}
The QPB on the crossed product algebra $B\#_\sigma H$ of Theorem \ref{Theorem:QPB} admits a strong connection
\begin{align*}
    c\colon(B\#_\sigma H)\otimes{}^{\mathrm{co}H}\Omega^1(H)\to\Omega^1(B\#_\sigma H),\qquad
    (b\otimes h)\otimes\gamma\mapsto b\otimes h\gamma.
\end{align*}
\end{proposition}
\begin{proof}
Let us define $c:=\iota_2\circ g$, where $\iota_2\colon B\otimes\Omega^1(H)\to\Omega^{1}(B\#_\sigma H)$ is the canonical inclusion and $g\colon(B\#_{\sigma}H)\otimes{}^{\mathrm{co}H}\Omega^1(H)\to B\otimes\Omega^1(H)$, $b\otimes h\otimes\gamma\mapsto b\otimes h\gamma$.
We already know that $\iota_2$ is a morphism in $_{B\#_{\sigma}H}\mm^{H}$. Moreover, $g$ is left $B\#_{\sigma}H$-linear since
\[
\begin{split}
g((b'\otimes h')\cdot(b\otimes h\otimes\gamma))&=g(b'(h'_{1}\cdot b)\sigma(h'_{2}\otimes h_{1})\otimes h'_{3}h_{2}\otimes\gamma)=b'(h'_{1}\cdot b)\sigma(h'_{2}\otimes h_{1})\otimes h'_{3}h_{2}\gamma\\&=b'(h'_{1}\cdot b)\sigma(h'_{2}\otimes h_{1}\gamma_{-1})\otimes h'_{3}h_{2}\gamma_{0}=b'(h'_{1}\cdot b)\sigma(h'_{2}\otimes(h\gamma)_{-1})\otimes h'_{3}(h\gamma)_{0}\\&=(b'\otimes h')\cdot g(b\otimes h\otimes\gamma)
\end{split}
\]
for all $b,b'\in B$, $h,h'\in H$ and $\gamma\in{}^{\mathrm{co}H}\Omega^1(H)$ 
and right $H$-colinear, indeed 
\[
g(b\otimes h\otimes\gamma)_{0}\otimes g(b\otimes h\otimes\gamma)_{1}=b\otimes(h\gamma)_{0}\otimes(h\gamma)_{1}=b\otimes h_{1}\gamma_{0}\otimes h_{2}\gamma_{1}=(g\otimes\mathrm{Id}_{H})(b\otimes h_{1}\otimes\gamma_{0}\otimes h_{2}\gamma_{1}).
\]
Thus, $g$ 
is a morphism in $_{B\#_{\sigma}H}\mm^{H}$ and then also $\iota_2\circ g$ is in $_{B\#_{\sigma}H}\mm^{H}$. Since $\mathrm{ver}=p\circ\pi_2$, where $p\colon B\otimes\Omega^1(H)\to(B\#_{\sigma}H)\otimes{}^{\mathrm{co}H}\Omega^1(H)$, $b\otimes\gamma\mapsto b\otimes\gamma_{-2}\otimes S(\gamma_{-1})\gamma_{0}$ is the inverse of $g$, we obtain
\[
\text{ver}\circ\iota_{2}\circ g=p\circ\pi_{2}\circ\iota_{2}\circ g=p\circ g=\mathrm{Id},
\]
proving that $c=\iota_2\circ g$ is a connection. It remains to prove that $c$ is strong.

On elements $b\otimes h\in B\#_{\sigma}H$ the strong condition $(\mathrm{Id}-c\circ\mathrm{ver})\mathrm{d}_{\#_\sigma}(B\#_{\sigma}H)\subseteq\Omega^1(B)(B\#_{\sigma}H)$ reads $\mathrm{d}_{B}b\otimes h+b\otimes\mathrm{d}_{H}h-c(b\otimes h_{1}\otimes S(h_{2})\mathrm{d}_{H}h_{3})\in\Omega^1(B)(B\#_{\sigma}H)$. The latter is satisfied, since
\[
\mathrm{d}_{B}b\otimes h+b\otimes\mathrm{d}_{H}h-c(b\otimes h_{1}\otimes S(h_{2})\mathrm{d}_{H}h_{3})=\mathrm{d}_{B}b\otimes h+b\otimes\mathrm{d}_{H}h-b\otimes h_{1}S(h_{2})\mathrm{d}h_{3}=\mathrm{d}_{B}b\otimes h=(\mathrm{d}_{B}b\otimes1_{H})(1_{B}\otimes h)
\]
holds.
\end{proof}

\begin{remark}
As we recalled in the preliminaries, given a right $H$-comodule algebra $A$, $B\subseteq A$ is a faithfully flat Hopf-Galois extension if and only if there exists a strong connection on the universal FODC on $A$. This is also equivalent to the existence of a linear map $l:H\to A\otimes A$, $h\mapsto l(h):=h^{(1)}\otimes h^{(2)}$ such that the following equalities are satisfied for all $h\in H$:
\[
l(1_{H})=1_{A}\otimes1_{A},\qquad h^{(1)}h^{(2)}=\varepsilon(h)1_{A}, 
\]
\[
h^{(1)}\otimes(h^{(2)})_{0}\otimes(h^{(2)})_{1}=(h_{1})^{(1)}\otimes(h_{1})^{(2)}\otimes h_{2}, \qquad(h^{(1)})_{0}\otimes(h^{(1)})_{1}\otimes h^{(2)}=(h_{2})^{(1)}\otimes S(h_{1})\otimes(h_{2})^{(2)},
\]
see \cite[Theorem 6.19 and Theorem 6.20]{Hajac}. In case of the tensor product algebra $B\otimes H$, as well as the smash product algebra $B\#H$, such a morphism $l$ is given by $h\mapsto 1_{B}\otimes S(h_{1})\otimes 1_{B}\otimes h_{2}$, which descends to a strong connection $\bar{l}:{}^{\mathrm{co}H}\Omega^1(H)\to\Omega^1(B\#H)$, $\xi\mapsto1_{B}\otimes\xi$, see \cite[Example 5.55]{BM}. Note that $h\mapsto 1_{B}\otimes S(h)$ is the convolution inverse of the cleaving map $j:h\mapsto 1_{B}\otimes h$ in the trivial case. According to \cite[Corollary 5.14]{BM} strong connections on trivial extensions are equivalent to morphisms $\alpha:H^{+}\to\Omega^1(B)$. The previous choice corresponds to the trivial $\alpha$.

We show that the strong connection of Proposition \ref{Prop:SC} is the generalization of precisely this connection to cleft extensions and the crossed product calculus. In the crossed product case the convolution inverse of the cleaving map is given by $j^{-1}:h\mapsto\sigma^{-1}(S(h_{2})\otimes h_{3})\otimes S(h_{1})$. It follows that the map $l:=(j^{-1}\otimes j)\Delta:h\mapsto \sigma^{-1}(S(h_{2})\otimes h_{3})\otimes S(h_{1})\otimes1_{B}\otimes h_{4}$ satisfies the previous four properties. Hence, we obtain the strong connection $\bar{l}:H^{+}\to\Omega^1(B\#_{\sigma}H)$, $h\mapsto h^{(1)}\mathrm{d}_{\#_{\sigma}}h^{(2)}=j^{-1}(h_{1})\mathrm{d}_{\#_{\sigma}}(j(h_{2}))$. We compute for $h\in H^+$
\[
\begin{split}
    \bar{l}(h)&=(\sigma^{-1}(S(h_{2})\otimes h_{3})\otimes S(h_{1}))(1_{B}\otimes\mathrm{d}_{H}h_{4})=\sigma^{-1}(S(h_{2})\otimes h_{3})\sigma(S(h_{1})_{1}\otimes(\mathrm{d}_{H}h_{4})_{-1})\otimes S(h_{1})_{2}(\mathrm{d}_{H}h_{4})_{0}\\&=\sigma^{-1}(S(h_{3})\otimes h_{4})\sigma(S(h_{2})\otimes h_{5})\otimes S(h_{1})\mathrm{d}_{H}h_{6}=1_{B}\otimes S(h_{1})\mathrm{d}_{H}h_{2}=1_B\otimes\varpi(h),
\end{split}
\]
where $\varpi:H^+\to{}^{\mathrm{co}H}\Omega^1(H)$, $h\mapsto S(h_1)\mathrm{d}_Hh_2$ denotes the Cartan-Maurer form, a surjective map (see e.g. \cite[Theorem 2.26]{BM}).
Thus, $c((b\otimes h)\otimes\varpi(g))=b\otimes h\bar{l}(g)=b\otimes h\varpi(g)$, where $b\in B$, $h\in H$ and $g\in H^+$, in fact recovers the strong connection of Proposition \ref{Prop:SC}.
\end{remark}

Using the strong connection of Proposition \ref{Prop:SC} we can construct covariant derivatives of associated bundles of the crossed product QPB. Recall that, given a FODC $(\Omega^1(B),\mathrm{d}_{B})$ on an algebra $B$, a \textit{covariant derivative} on a left $B$-module $E$ is a linear map $\nabla_{E}\colon E\to\Omega^1(B)\otimes_{B}E$ obeying the left Leibniz rule 
$$
\nabla_{E}(b\cdot e)=\mathrm{d}_{B}b\otimes_B e+b\cdot\nabla_{E}e,
$$
for $e\in E$ and $b\in B$. In the notation of \cite[page 227]{BM} we denote by $_{B}\xi$ the category whose objects are given by pairs $(E,\nabla_{E})$ where $E$ is a left $B$-module and $\nabla_{E}$ is a covariant derivative on $E$ and morphisms are left $B$-linear maps intertwining $\nabla$. More explicitly, a morphism between two objects $(E,\nabla_{E})$ and $(K,\nabla_{K})$ in $_B\xi$ is a left $B$-linear map $T\colon E\to K$ such that $\nabla_{K}\circ T=(\mathrm{Id}\otimes_{B}T)\circ\nabla_{E}$.
If $E$ is a $B$-bimodule, a covariant derivative $\nabla_{E}:E\to\Omega^1(B)\otimes_{B}E$ on $E$ is called \textit{bimodule covariant derivative} if it is equipped with a $B$-bimodule map $\sigma_{E}:E\otimes_{B}\Omega^1(B)\to\Omega^1(B)\otimes_{B}E$ obeying $\nabla_{E}(e\cdot b)=\sigma_{E}(e\otimes_{B}\mathrm{d}_{B}b)+(\nabla_{E}e)\cdot b$, for $e\in E$ and $b\in B$. Note that $\sigma_{E}$ is not an additional data as, if it exists, it is uniquely determined by the stated condition. Hence, being a bimodule covariant derivative or not is a property of a covariant derivative $\nabla_{E}$ on a given bimodule $E$. A particular class of interest is given by the \textit{associated bundles} $E:=(A\otimes V)^{\mathrm{co}H}$, where $B:=A^{\mathrm{co}H}\subseteq A$ is a regular QPB and $V$ is a right $H$-comodule, see \cite[Definition A.3]{BrMa}. Here $E$ is a $B$-bimodule with left and right $B$-actions given by multiplication on $A$. 
Those admit associated covariant derivatives, see \cite[Proposition 5.48]{BM}. We give their explicit expression for the regular QPB on the crossed product algebra of Theorem \ref{Theorem:QPB}: they correspond to the differential on the base algebra.
\begin{proposition}\label{cor:derivative}
Given $V$ a right $H$-comodule, then the $B$-bimodule $E:=((B\#_{\sigma}H)\otimes V)^{\mathrm{co}H}$, with actions induced from the multiplication of $B\#_{\sigma}H$, acquires an associated covariant derivative $\nabla_{E}:E\to\Omega^1(B)\otimes_{B}E$ given by
\[
\nabla_{E}(b\otimes h\otimes v)=(\mathrm{Id}-c\circ\mathrm{ver})\mathrm{d}_{\#_\sigma}(b\otimes h)\otimes v=\mathrm{d}_{B}b\otimes h\otimes v,
\]
where $c$ is the canonical strong connection of Proposition \ref{Prop:SC}. Moreover, $\nabla_{E}$ is a bimodule covariant derivative with associated $\sigma_{E}:E\otimes_{B}\Omega^1(B)\to\Omega^1(B)\otimes_{B}E$ defined by $\sigma_{E}(b\otimes h\otimes v\otimes_{B}\beta):=b(h_{1}\cdot\beta)\otimes h_{2}\otimes v$.
\end{proposition}

\begin{proof}
    The first part follows by applying \cite[Proposition 5.48]{BM} as we have discussed before. It remains to prove that $\sigma_{E}$ is a $B$-bimodule map obeying $\nabla_{E}(e\cdot b)=\sigma_{E}(e\otimes_{B}\mathrm{d}_{B}b)+(\nabla_{E}e)\cdot b$, for $e\in E$ and $b\in B$. We compute
\[
\begin{split}
    \sigma_{E}(b'\cdot(b\otimes h\otimes v\otimes_{B}\beta))&=\sigma_{E}(b'b\otimes h\otimes v\otimes_{B}\beta)=b'b(h_{1}\cdot\beta)\otimes h_{2}\otimes v=b'\cdot \sigma_{E}(b\otimes h\otimes v\otimes_{B}\beta),
\end{split}
\]
as well as
\[
\begin{split}
    \sigma_{E}((b\otimes h\otimes v\otimes_{B}\beta)\cdot b')&=\sigma_{E}(b\otimes h\otimes v\otimes_{B}\beta b')=b(h_{1}\cdot\beta b')\otimes h_{2}\otimes v\\&=b(h_{1}\cdot\beta)(h_{2}\cdot b')\otimes h_{3}\otimes v=\sigma_{E}(b\otimes h\otimes v\otimes_{B}\otimes\beta)\cdot b'
\end{split}
\]
for all $b\otimes h\otimes v\in E$, $b'\in B$ and $\beta\in\Omega^1(B)$.
Thus, $\sigma_{E}$ is a $B$-bimodule map. Moreover, we have 
\[
\begin{split}
    \nabla_{E}((b\otimes h\otimes v)\cdot b')&=\nabla_{E}(b(h_{1}\cdot b')\otimes h_{2}\otimes v)=\mathrm{d}_{B}(b(h_{1}\cdot b'))\otimes h_{2}\otimes v\\&=b\mathrm{d}_{B}(h_{1}\cdot b')\otimes h_{2}\otimes v+(\mathrm{d}_{B}b)(h_{1}\cdot b')\otimes h_{2}\otimes v\\&=b(h_{1}\cdot\mathrm{d}_{B}b')\otimes h_{2}\otimes v+(\mathrm{d}_{B}b\otimes h\otimes v)\cdot b'\\&=\sigma_{E}(b\otimes h\otimes v\otimes_{B}\mathrm{d}_{B}b')+\nabla_{E}(b\otimes h\otimes v)\cdot b'
\end{split}
\]
and then the thesis follows.
\end{proof}

We close this subsection by studying another application of the crossed product QPB. \medskip

\noindent\textbf{Leray--Serre spectral sequence}. 
The \textit{noncommutative de Rham cohomology} of a DC $(\Omega^\bullet,\mathrm{d})$ on an algebra $A$ is the graded algebra
\[
\mathrm{H}^{n}_{\mathrm{dR}}(A):=\frac{\mathrm{ker}(\mathrm{d}|_{\Omega^{n}})}{\mathrm{Im}(\mathrm{d}|_{\Omega^{n-1}})},
\]
where we understand $\mathrm{d}|_{\Omega^{n}}=0$ if $n<0$. In the following we show that the noncommutative de Rham cohomology of a crossed product algebra $B\#_\sigma H$ is determined by the noncommutative de Rham cohomology of the base algebra $B$ and of the structure Hopf algebra $H$. Namely, we prove that there is a spectral sequence with second page $E^{p,q}_{2}=\mathrm{H}_{\mathrm{dR}}^{p}(B)\otimes\mathrm{H}_{\mathrm{dR}}^{q}(H)$, which converges to $\mathrm{H}_{\mathrm{dR}}(B\#_{\sigma}H)$.

Recall that a \textit{spectral sequence} of degree $r_0\geq 0$ is a collection $(E_r^{\bullet,\bullet},D_r)_{r\geq r_0}$ of $\mathbb{Z}$-bigraded vector spaces $E_r^{\bullet,\bullet}=\bigoplus_{k,\ell\in\mathbb{Z}}E_r^{k,\ell}$ (where $E_r^{k,\ell}$ is assumed to be zero if either $k<0$ or $\ell<0$) and linear maps $D_r\colon E_r^{k,\ell}\to E_r^{k+r,\ell-r+1}$ such that $D_r\circ D_r=0$ and $E^{k,\ell}_{r+1}\cong\mathrm{H}^{k,\ell}(E_r,D_r)$, where
$$
\mathrm{H}^{k,\ell}(E_r,D_r):=\frac{\ker D_r}{\mathrm{Im}\,D_r|_{E_r^{k-r,\ell+r-1}}}.
$$
One calls $(E_r^{\bullet,\bullet},D_r)$ the $r^\mathrm{th}$ page of the spectral sequence. So the above means that the $(r+1)^\mathrm{th}$ page is isomorphic to the cohomology of the $r^\mathrm{th}$ page.

Now we proceed by constructing a filtration of $\Omega^\bullet(B\#_\sigma H)$, which will then provide the announced spectral sequence. Recall from Theorem \ref{thm:higerorderforms} that the right $H$-coaction $\rho\colon B\#_\sigma H\to(B\#_\sigma H)\otimes H$ is differentiable with differential $\Omega^\bullet(\rho)\colon\Omega^\bullet(B\#_\sigma H)\to\Omega^\bullet((B\#_\sigma H)\otimes H)$ and let us denote the graded components by $\rho^{n-k,k}\colon\Omega^n(B\#_\sigma H)\to\Omega^{n-k}(B\#_\sigma H)\otimes\Omega^k(H)$. We define for $m,n\geq 0$
$$
F^m(\Omega^n(B\#_\sigma H)):=\begin{cases}
\bigcap_{k>n-m}\ker\rho^{n-k,k} & \text{ if }m\leq n\\
0 & \text{ otherwise}
\end{cases}
$$
a filtration of $\Omega^\bullet(B\#_\sigma H)$.
The corresponding spectral sequence $(E_r^{\bullet,\bullet},D_r)_{r\geq 1}$ of degree 1 is constructed in the following way. Its first page is the $\mathbb{Z}$-bigraded vector space $E_1^{\bullet,\bullet}$, where
$$
E_1^{p,q}:=\frac{\ker\bigg(\mathrm{d}_{\#_{\sigma}}\colon\frac{F^p(\Omega^{p+q}(B\#_\sigma H))}{F^{p+1}(\Omega^{p+q}(B\#_\sigma H))}\to\frac{F^p(\Omega^{p+q+1}(B\#_\sigma H))}{F^{p+1}(\Omega^{p+q+1}(B\#_\sigma H))}\bigg)}{\mathrm{Im}\,\bigg(\mathrm{d}_{\#_{\sigma}}\colon\frac{F^p(\Omega^{p+q-1}(B\#_\sigma H))}{F^{p+1}(\Omega^{p+q-1}(B\#_\sigma H))}\to\frac{F^p(\Omega^{p+q}(B\#_\sigma H))}{F^{p+1}(\Omega^{p+q}(B\#_\sigma H))}\bigg)}.
$$
The differential $D_1\colon E_1^{p,q}\to E_1^{p+1,q}$ is the one induced from $\mathrm{d}_{\#_\sigma}$ on the quotient and higher pages are defined as the cohomology of the previous pages. According to \cite[Lemma 5.58]{BM} the spectral sequence $(E_r^{\bullet,\bullet},D_r)_{r\geq 1}$ converges to $\mathrm{H}_\mathrm{dR}(B\#_\sigma H)$.

In the rest of this subsection we show that the above spectral sequence coincides with famous Leray--Serre spectral sequence (see \cite[Theorem 4.66]{BM}).
In order to do this we have to introduce higher order analogues of the vertical map. We recall once more from Theorem \ref{thm:higerorderforms} that the right $H$-coaction $\rho\colon B\#_\sigma H\to(B\#_\sigma H)\otimes H$ is differentiable with 
\begin{align*}
    \Omega^\bullet(\rho)\colon\Omega^\bullet(B\#_\sigma H)&\to\Omega^\bullet((B\#_\sigma H)\otimes H),\\
    \omega\otimes\gamma&\mapsto\omega\otimes\gamma_{[1]}\otimes\gamma_{[2]},
\end{align*} 
where $\Omega^\bullet(\Delta)(\gamma)=\gamma_{[1]}\otimes\gamma_{[2]}$ is the differential of the coproduct $\Delta\colon H\to H\otimes H$ with graded components $\Omega^n(\Delta)(\gamma)=\sum\nolimits_{k+l=n}\gamma_{<1,k>}\otimes\gamma_{<2,l>}=\sum\nolimits_{m}\gamma_{<1,m>}\otimes\gamma_{<2,n-m>}$ for $\gamma\in\Omega^n(H)$. We denote the graded components of crossed product forms by $\omega^i\otimes\gamma^j\in\Omega^i(B)\otimes\Omega^j(H)$.
Following \cite[page 449]{BM} we define for all $n>0$ and $n\geq m\geq 0$
\begin{align*}
    \mathrm{ver}^{m,n-m}\colon\Omega^{n}(B\#_{\sigma}H)&\to\Omega^{m}(B\#_{\sigma}H)\otimes{}^{\mathrm{co}H}\Omega^{n-m}(H)\\
    \sum_{i+j=n}\omega^i\otimes\gamma^j&\mapsto\sum_{i+j=n}\omega^i\otimes\gamma^j_{<1,m>}\otimes S((\gamma^j_{<2,n-m>})_{-1})(\gamma^j_{<2,n-m>})_0,
\end{align*}
where we recall that $\gamma^j_{<1,m>}\otimes\gamma^j_{<2,n-m>}=0$ if $m>j$ or $n-m>j$. In other words, $\mathrm{ver}^{m,n-m}$ is defined as the composition $(\mathrm{Id}\otimes\phi^{n-m})\circ\pi^{m,n-m}\circ\Omega^n(\rho)$, where $\pi^{m,n-m}\colon\Omega^n((B\#_\sigma H)\otimes H)\to\Omega^m(B\#_\sigma H)\otimes\Omega^{n-m}(H)$ is the canonical projection and $\phi^{n-m}\colon\Omega^{n-m}(H)\to{}^{\mathrm{co}H}\Omega^{n-m}(H)$, $\gamma\mapsto S(\gamma_{-1})\gamma_0$ is the projection on left $H$-coinvariant elements.

Note that $\mathrm{ver}^{0,1}=\mathrm{ver}\colon\Omega^1(B\#_{\sigma}H)\to(B\#_{\sigma}H)\otimes{}^{\mathrm{co}H}\Omega^1(H)$ is the classical vertical map.
More in general, for $n>0$, the vertical map $\mathrm{ver}^{0,n}:\Omega^{n}(B\#_{\sigma}H)\to(B\#_{\sigma}H)\otimes{}^{\mathrm{co}H}\Omega^{n}(H)$ is given for $\omega^{i}\in\Omega^i(B)$, $\gamma^i\in\Omega^i(H)$,
$b\in B$ and $h\in H$ by
$$
\mathrm{ver}^{0,n}\big(\omega^n\otimes h+\omega^{n-1}\otimes\gamma^{1}+\cdot\cdot\cdot+\omega^{1}\otimes\gamma^{n-1}+b\otimes\gamma^{n}\big)=b\otimes\gamma^{n}_{-2}\otimes S(\gamma^{n}_{-1})\gamma^{n}_{0}.
$$
We have shown in Theorem \ref{Theorem:QPB} that $\mathrm{ver}=\mathrm{ver}^{0,1}$ fits into a short exact sequence, the Atiyah sequence (\ref{Atiyahseq}). It turns out that also $\mathrm{ver}^{0,n}$ for $n>0$ give rise to short exact sequences and, under a flatness assumption, the same is true for all higher vertical maps $\mathrm{ver}^{m,n}$.
\begin{proposition}
The sequence
\begin{equation}\label{exactseq}
0\longrightarrow\Omega^1(B)\wedge\Omega^{n-1}(B\#_{\sigma}H)\overset{i}\longrightarrow\Omega^{n}(B\#_{\sigma}H)\xrightarrow{\mathrm{ver}^{0,n}}(B\#_{\sigma}H)\otimes{}^{\mathrm{co}H}\Omega^{n}(H)\longrightarrow0
\end{equation}
is exact for all $n>0$.
\end{proposition}

\begin{proof}
Since $i$ is injective, for the exactness of \eqref{exactseq} it remains to prove that the map $\mathrm{ver}^{0,n}$ is surjective with $\mathrm{ker(ver}^{0,n})=\Omega^1(B)\wedge\Omega^{n-1}(B\#_{\sigma}H)$, for every $n>0$.
Observe that $\mathrm{ver}^{0,n}=p^{n}\circ\pi_{n}$, where $\pi_{n}:\Omega^{n}(B\#_{\sigma}H)\to B\otimes\Omega^{n}(H)$ is the canonical projection, and $p^{n}:B\otimes\Omega^{n}(H)\to(B\#_{\sigma}H)\otimes{}^{\mathrm{co}H}\Omega^{n}(H)$, $b\otimes\gamma^{n}\mapsto b\otimes\gamma^{n}_{-2}\otimes S(\gamma^{n}_{-1})\gamma^{n}_{0}$ for all $\gamma^{n}\in\Omega^{n}(H)$. The inverse of $p^n$ is given by $g^{n}\colon(B\#_{\sigma}H)\otimes{}^{\mathrm{co}H}\Omega^{n}(H)\to B\otimes\Omega^{n}(H)$, $b\otimes h\otimes\gamma^{n}\mapsto b\otimes h\gamma^{n}$ as one easily verifies. Thus, $\mathrm{ver}^{0,n}$ is surjective and $\mathrm{ker(ver}^{0,n})=\mathrm{ker}(\pi_{n})=\bigoplus_{k=0}^{n-1}{\Omega^{n-k}(B)\otimes\Omega^{k}(H)}=\Omega^1(B)\wedge\Omega^{n-1}(B\#_{\sigma}H)$. 
\end{proof}

Then, under a flatness assumption, all the hypotheses of \cite[Lemma 5.60]{BM} are satisfied and so we obtain:

\begin{corollary}
    If $\Omega^m(B)$ is flat as a right $B$-module for all $m>0$ then
\begin{equation}\label{exactseq2}
0\longrightarrow\Omega^{m+1}(B)\wedge\Omega^{n-1}(B\#_{\sigma}H)\overset{i}\longrightarrow\Omega^{m}(B)\wedge\Omega^{n}(B\#_{\sigma}H)\xrightarrow{\mathrm{ver}^{m,n}}\Omega^{m}(B)\otimes_B(B\#_{\sigma}H)\otimes{}^{\mathrm{co}H}\Omega^{n}(H)\longrightarrow0
\end{equation}
is exact for all $n>0$ and $m\geq 0$.
\end{corollary}

Hence, by specializing \cite[Corollary 5.62]{BM} to our setting, we obtain that the spectral sequence $(E_r^{\bullet,\bullet},D_r)_{r\geq 1}$ is the Leray--Serre spectral sequence and, in case $\mathrm{H}_{\mathrm{dR}}(H)$ is left-invariant, we obtain an explicit description of the second page in terms of the cohomology of the base algebra and structure Hopf algebra. More explicitly, we obtain:

\begin{proposition}
Let $(\Omega^{\bullet}(B),\mathrm{d}_{B})$ be a $\sigma$-twisted $H$-module DC on $B$ and $(\Omega^{\bullet}(H),\mathrm{d}_{H})$ a DC on $H$ with respect to which $\Delta:H\to H\otimes H$ is differentiable. If $\Omega^{m}(B)$ is flat as a right $B$-module for all $m>0$ and $\mathrm{H}_{\mathrm{dR}}(H)$ is left-invariant, then we get a spectral
sequence with second page $E^{p,q}_{2}=\mathrm{H}^{p}_{\mathrm{dR}}(B)\otimes\mathrm{H}^{q}_{\mathrm{dR}}(H)$ which converges to $\mathrm{H}_{\mathrm{dR}}(B\#_{\sigma}H)$.
\end{proposition}

\subsection{Connection 1-forms and fundamental vector fields}\label{Sec3.5}

In this subsection we fix a Hopf algebra $H$ with invertible antipode and a bicovariant FODC $(\Omega^1(H),\mathrm{d}_{H})$ on $H$, as well as a $\sigma$-twisted left $H$-module algebra $B$ and a $\sigma$-twisted $H$-module FODC $(\Omega^1(B),\mathrm{d}_{B})$ on $B$. Moreover, we assume that $^{\mathrm{co}H}\Omega^1(H)$ is finite-dimensional and we consider the \textit{quantum tangent space} $T_\varepsilon H:=\mathrm{Hom}_{\Bbbk}(^{\mathrm{co}H}\Omega^1(H),\Bbbk)$, as introduced in \cite{Wo}. We structure $T_\varepsilon H$ as a right $H$-comodule via $T_\varepsilon H\to(T_\varepsilon H)\otimes H,\ \alpha\mapsto\alpha_{0}\otimes\alpha_{1}$, where $\alpha_{0}(\gamma)\alpha_{1}=\alpha(\gamma_{0})S^{-1}(\gamma_{1})$ for any $\gamma\in{}^{\mathrm{co}H}\Omega^1(H)$, see e.g. \cite[Definition 4.4]{PS}. 
To any $\alpha\in T_\varepsilon H$ corresponds a \textit{noncommutative vector field} (in the sense of \cite{Bo1,Bo2}), namely the left $B\#_{\sigma}H$-linear map
$$
\overline{\alpha}:=(\mathrm{Id}\otimes\alpha)\circ\mathrm{ver}\colon\Omega^1(B\#_\sigma H)\to B\#_\sigma H,
$$
see \cite[Corollary 6.2 and page 490]{BM}. Moreover, $\overline{\alpha}$ is \textit{vertical}, meaning that $\overline{\alpha}(\Omega^1(B)\otimes H)=0$ and $\overline{\alpha}(1_B\otimes\gamma)=\alpha(\gamma)1_B\otimes 1_H$ for all $\gamma\in{}^{\mathrm{co}H}\Omega^1(H)$. It is the unique vector field with these two properties (since $\overline{\alpha}$ is left $B\#_{\sigma}H$-linear and every $b\otimes\gamma\in B\otimes\Omega^1(H)$ can be written as $(b\otimes\gamma_{-2})\cdot(1_{B}\otimes S(\gamma_{-1})\gamma_{0})$)
and thus,
following \cite[Definition 4.6]{PS}, we call $\overline{\alpha}$ the \textit{fundamental vector field} corresponding to $\alpha\in T_\varepsilon H$. \medskip

Suppose that $\mathrm{dim}(^{\mathrm{co}H}\Omega^1(H))=n$ and let $\{x_{j}\}_{j=1}^{n}$ be a basis of $T_\varepsilon H$ and $\{x^j\}_{j=1}^{n}$ be the dual basis of $^{\mathrm{co}H}\Omega^1(H)$, which can be identified with $(T_\varepsilon H)^{*}$. Recall that, in the setup of this subsection, a \textit{connection 1-form} is an element $\phi=\sum_{j=1}^{n}{x_{j}\otimes\phi_{j}}\in(T_\varepsilon H\otimes\Omega^1(B\#_\sigma H))^{\mathrm{co}H}$, satisfying $\sum_{j=1}^{n}{x_{j}\otimes\overline{\alpha}(\phi_{j})}=\alpha\otimes1_{B}\otimes1_{H}$, for all $\alpha\in T_\varepsilon H$.
This definition is in complete analogy with \cite[Definition 4.10]{PS}, however in case of the crossed product, rather than the smash product calculus. The following lemma and proposition are the corresponding analogues of \cite[Lemma 4.12 and Theorem 4.13]{PS}. We report here their proofs for the sake of  completeness.

\begin{lemma}
    An element $\phi\in(T_\varepsilon H\otimes\Omega^1(B\#_\sigma H))^{\mathrm{co}H}$ is a connection 1-form if and only if $(\mathrm{Id}\otimes\pi_{2})(\phi)=\sum_{j=1}^{n}{x_{j}\otimes1_{B}\otimes x^{j}}$, where $\pi_{2}:\Omega^1(B\#_{\sigma}H)\to B\otimes\Omega^1(H)$ is the canonical projection.
\end{lemma}

\begin{proof}
    First note that, given an element $\phi=\sum_{j=1}^{n}{x_{j}\otimes\phi_{j}}\in(T_\varepsilon H\otimes\Omega^1(B\#_\sigma H))^{\mathrm{co}H}$, we can write $\phi_{j}=\phi'_{j}+\phi''_{j}$ where $\phi'_{j}\in\Omega^1(B)\otimes H$ and $\phi''_{j}\in B\otimes\Omega^1(H)$, for every $j=1,...,n$. But now $\sum_{j=1}^{n}{x_{j}\otimes\overline{\alpha}(\phi_{j})}=\sum_{j=1}^{n}{x_{j}\otimes\overline{\alpha}(\phi''_{j})}$ for all $\alpha\in T_\varepsilon H$ and furthermore $(\mathrm{Id}\otimes\pi_{2})(\phi)=\sum_{j=1}^{n}{x_{j}\otimes\phi''_{j}}$. Thus, the claim of the lemma follows if we prove that $\sum_{j=1}^{n}{x_{j}\otimes\overline{\alpha}(\phi''_{j})}=\alpha\otimes1_{B}\otimes1_{H}$ holds for all $\alpha\in T_\varepsilon H$ if and only if $\sum_{j=1}^{n}{x_{j}\otimes\phi''_{j}}=\sum_{j=1}^{n}{x_{j}\otimes1_{B}\otimes x^{j}}$. Clearly, if we have that $\sum_{j=1}^{n}{x_{j}\otimes\phi''_{j}}=\sum_{j=1}^{n}{x_{j}\otimes1_{B}\otimes x^{j}}$, then 
\[
\begin{split}
    \sum_{j=1}^{n}{x_{j}\otimes\overline{\alpha}(\phi''_{j})}=\sum_{j=1}^{n}{x_{j}\otimes(\mathrm{Id}\otimes\alpha)\mathrm{ver}(1_{B}\otimes x^{j})}=\sum_{j=1}^{n}{\alpha(x^{j})x_{j}\otimes1_{B}\otimes1_{H}}=\alpha\otimes1_{B}\otimes1_{H}
\end{split}
\]
for all $\alpha\in T_\varepsilon H$.
On the other hand, suppose that $\sum_{j=1}^{n}{x_{j}\otimes\overline{\alpha}(\phi''_{j})}=\alpha\otimes1_{B}\otimes1_{H}$ for all $\alpha\in T_\varepsilon H$. As a consequence, given $k,i\in\{1,...,n\}$, we have $\sum_{j=1}^{n}{x_{j}(x^{i})\overline{x_{k}}(\phi''_{j})}=x_{k}(x^{i})1_{B}\otimes1_{H}$. We know that $\phi''_{j}\in B\otimes\Omega^1(H)$ and that $g:(B\#_{\sigma}H)\otimes{}^{\mathrm{co}H}\Omega^{1}(H)\to B\otimes\Omega^1(H)$, $b\otimes h\otimes\gamma\mapsto b\otimes h\gamma$ and $p:B\otimes\Omega^1(H)\to(B\#_{\sigma}H)\otimes{}^{\mathrm{co}H}\Omega^1(H)$, $b\otimes\gamma\mapsto b\otimes\gamma_{-2}\otimes S(\gamma_{-1})\gamma_{0}$ are inverse isomorphisms in $_{B\#_{\sigma}H}\mm^{H}$. Thus, for all $j\in\{1,...,n\}$, we can write $\phi''_{j}=\sum_{i=1}^{n}{g(b_{i,j}\otimes h_{i,j}\otimes x^{i})}=\sum_{i=1}^{n}{b_{i,j}\otimes h_{i,j}x^{i}}$. Hence, we can compute 
\[
\begin{split}
    b_{i,j}\otimes h_{i,j}&=\sum_{l=1}^{n}{b_{i,l}\otimes h_{i,l}x_{j}(x^{l})}=\sum_{l=1}^{n}{(\mathrm{Id}\otimes x_{j})pg(b_{i,l}\otimes h_{i,l}\otimes x^{l})}=\sum_{l=1}^{n}{(\mathrm{Id}\otimes x_{j})p(b_{i,l}\otimes h_{i,l}x^{l})}\\&=\sum_{l=1}^{n}{(\mathrm{Id}\otimes x_{j})\mathrm{ver}(b_{i,l}\otimes h_{i,l}x^{l})}=\sum_{l=1}^{n}{\overline{x_{j}}(b_{i,l}\otimes h_{i,l}x^{l})}=\sum_{k=1}^{n}{\sum_{l=1}^{n}{x_{k}(x^{i})\overline{x_{j}}(b_{k,l}\otimes h_{k,l}x^{l})}}\\&=\sum_{k=1}^{n}{x_{k}(x^{i})\overline{x_{j}}(\sum_{l=1}^{n}{b_{k,l}\otimes h_{k,l}x^{l}}})=\sum_{k=1}^{n}{x_{k}(x^{i})\overline{x_{j}}(\phi''_{k})}=x_{j}(x^{i})1_{B}\otimes 1_{H},
\end{split}
\]
for all $i,j\in\{1,...,n\}$. Thus, if $i\not=j$ we obtain that $b_{i,j}\otimes h_{i,j}=0$ while if $i=j$ we obtain that $b_{j,j}\otimes h_{j,j}=1_{B}\otimes1_{H}$ and then $\phi''_{j}=1_{B}\otimes x^{j}$ for every $j=1,...,n$. Hence $\sum_{j=1}^{n}{x_{j}\otimes\phi''_{j}}=\sum_{j=1}^{n}{x_{j}\otimes1_{B}\otimes x^{j}}$.
\end{proof}

\begin{remark}
Given a connection 1-form $\phi=\sum_{j=1}^{n}{x_{j}\otimes\phi_{j}}$, we then have 
\[
\phi=\sum_{j=1}^{n}{x_{j}\otimes\phi'_{j}}+\sum_{j=1}^{n}{x_{j}\otimes1_{B}\otimes x^{j}},
\]
where $\phi'_{j}\in\Omega^1(B)\otimes H$ for $j=1,...,n$. Thus, for all $\gamma\in{}^{\mathrm{co}H}\Omega^1(H)$, we obtain 
\[
\sum_{j=1}^{n}{x_{j}(\gamma)\phi_{j}}=\sum_{j=1}^{n}{x_{j}(\gamma)\phi'_{j}}+\sum_{j=1}^{n}{x_{j}(\gamma)1_{B}\otimes x^{j}}=\sum_{j=1}^{n}{x_{j}(\gamma)\phi'_{j}}+1_{B}\otimes\gamma
\]
and then 
\begin{equation}\label{eq:connection}
\sum_{j=1}^{n}{x_{j}(\gamma)\mathrm{ver}(\phi_{j})}=\sum_{j=1}^{n}{x_{j}(\gamma)\mathrm{ver}(\phi'_{j})}+\mathrm{ver}(1_{B}\otimes\gamma)=1_{B}\otimes1_{H}\otimes\gamma.
\end{equation}
Observe also that the coevaluation map $\Bbbk\to(T_\varepsilon H)\otimes(T_\varepsilon H)^{*}$, $1_{\Bbbk}\mapsto\sum_{j=1}^{n}{x_{j}\otimes x^{j}}$ is right $H$-colinear. Thus, since $\sum_{j=1}^{n}{x_{j0}\otimes1_{B}\otimes x^{j}_{0}\otimes x_{j1}x^{j}_{1}}=\sum_{j=1}^{n}{x_{j}\otimes1_{B}\otimes x^{j}\otimes1_{H}}$, the fact that $\phi$ is coinvariant is expressed by $\sum_{j=1}^{n}{x_{j0}\otimes\phi'_{j0}\otimes x_{j1}\phi'_{j1}}=\sum_{j=1}^{n}{x_{j}\otimes\phi'_{j}\otimes1_{H}}$. 
\end{remark}

\begin{proposition}\label{prop:bijection}
For the crossed product calculus there is a bijection between connections and connection 1-forms.
Explicitly, to a connection $c:(B\#_{\sigma}H)\otimes{}^{\mathrm{co}H}\Omega^1(H)\to\Omega^1(B\#_{\sigma}H)$ we assign the connection $1$-form $\phi_{c}:=\sum_{j=1}^{n}{x_{j}\otimes c(1_{B}\otimes 1_{H}\otimes x^{j})}\in(T_\varepsilon H\otimes\Omega^1(B\#_{\sigma}H))^{\mathrm{co}H}$, while a connection $1$-form $\phi=\sum_{j=1}^{n}{x_{j}\otimes\phi_{j}}\in(T_\varepsilon H\otimes\Omega^1(B\#_\sigma H))^{\mathrm{co}H}$ corresponds to the connection
\begin{align*}
    c_{\phi}:(B\#_{\sigma}H)\otimes{}^{\mathrm{co}H}\Omega^1(H)&\to\Omega^1(B\#_\sigma H)\\
    b\otimes h\otimes\gamma&\mapsto\sum_{j=1}^{n}{x_{j}(\gamma)((b\otimes h)\cdot\phi_{j})}.
\end{align*}
\end{proposition}

\begin{proof}
    First suppose to have a connection $c:(B\#_{\sigma}H)\otimes{}^{\mathrm{co}H}\Omega^1(H)\to\Omega^1(B\#_{\sigma}H)$, i.e., a morphism in $_{B\#_{\sigma}H}\mm^{H}$ such that $\mathrm{ver}\circ c=\mathrm{Id}$ and define $\phi_{c}:=\sum_{j=1}^{n}{x_{j}\otimes c(1_{B}\otimes 1_{H}\otimes x^{j})}\in T_\varepsilon H\otimes\Omega^1(B\#_{\sigma}H)$. Observe that $\phi_{c}\in(T_\varepsilon H\otimes\Omega^1(B\#_{\sigma}H))^{\mathrm{co}H}$, indeed
\[ 
\begin{split}
    \sum_{j=1}^{n}{x_{j0}\otimes c(1_{B}\otimes1_{H}\otimes x^{j})_{0}\otimes x_{j1}c(1_{B}\otimes1_{H}\otimes x^{j})_{1}}
    &=\sum_{j=1}^{n}{x_{j0}\otimes c(1_{B}\otimes1_{H}\otimes x^{j}_{0})\otimes x_{j1}x^{j}_{1}}\\&=\sum_{j=1}^{n}{x_{j}\otimes c(1_{B}\otimes1_{H}\otimes x^{j})\otimes1_{H}},
\end{split}
\]    
since $c$ is right $H$-colinear as well as the map $\Bbbk\in1_{\Bbbk}\mapsto\sum_{j=1}^{n}{x_{j}\otimes x^{j}}\in T_\varepsilon H\otimes(T_\varepsilon H)^{*}$. Furthermore, we have
\[ 
\begin{split}
    \sum_{j=1}^{n}{x_{j}\otimes\overline{\alpha}(c(1_{B}\otimes1_{H}\otimes x^{j}))}&=\sum_{j=1}^{n}{x_{j}\otimes(\mathrm{Id}\otimes\alpha)\mathrm{ver}(c(1_{B}\otimes1_{H}\otimes x^{j}))}=\sum_{j=1}^{n}{\alpha(x^{j})x_{j}\otimes1_{B}\otimes1_{H}}=\alpha\otimes1_{B}\otimes1_{H},
\end{split}    
\]  
 thus $\phi_{c}$ is a connection 1-form. 
 
 On the other hand, if we have a connection 1-form $\phi=\sum_{j=1}^{n}{x_{j}\otimes\phi_{j}}\in(T_\varepsilon H\otimes\Omega^1(B\#_\sigma H))^{\mathrm{co}H}$, we can define $c_{\phi}:(B\#_{\sigma}H)\otimes{}^{\mathrm{co}H}\Omega^1(H)\to\Omega^1(B\#_\sigma H)$ as the morphism in $_{B\#_{\sigma}H}\mm$ given by $c_{\phi}(1_{B}\otimes1_{H}\otimes\gamma):=\sum_{j=1}^{n}{x_{j}(\gamma)\phi_{j}}$, so that $c_{\phi}(b\otimes h\otimes\gamma)=(b\otimes h)\cdot c_{\phi}(1_{B}\otimes1_{H}\otimes\gamma)=\sum_{j=1}^{n}{x_{j}(\gamma)((b\otimes h)\cdot\phi_{j})}$. We have that $c_{\phi}$ is also right $H$-colinear, indeed
\[
\begin{split}
    c_{\phi}(b\otimes h\otimes\gamma)_{0}\otimes c_{\phi}(b\otimes h\otimes\gamma)_{1}&=\sum_{j=1}^{n}{x_{j}(\gamma)((b\otimes h)\cdot\phi_{j})_{0}\otimes((b\otimes h)\cdot\phi_{j})_{1}}=\sum_{j=1}^{n}{x_{j}(\gamma)((b\otimes h_{1})\cdot\phi_{j0})\otimes h_{2}\phi_{j1}}\\&=\sum_{j=1}^{n}{x_{j}(\gamma_{0})((b\otimes h_{1})\cdot\phi_{j0})\otimes h_{2}\varepsilon(\gamma_{1})\phi_{j1}}=\sum_{j=1}^{n}{((b\otimes h_{1})\cdot\phi_{j0})\otimes h_{2}\gamma_{2}x_{j}(\gamma_{0})S^{-1}(\gamma_{1})\phi_{j1}}\\&=\sum_{j=1}^{n}{x_{j0}(\gamma_{0})((b\otimes h_{1})\cdot\phi_{j0})\otimes h_{2}\gamma_{1}x_{j1}\phi_{j1}}=\sum_{j=1}^{n}{x_{j}(\gamma_{0})((b\otimes h_{1})\cdot\phi_{j})\otimes h_{2}\gamma_{1}}\\&=c(b\otimes h_{1}\otimes\gamma_{0})\otimes h_{2}\gamma_{1},
\end{split}
\]
using that the left $B\#_{\sigma}H$-action of $\Omega^1(B\#_\sigma H)$ is right $H$-colinear and that $\phi$ is a coinvariant element. Thus, $c_{\phi}$ is a morphism in $_{B\#_{\sigma}H}\mm^{H}$ and, since $\mathrm{ver}$ is left $B\#_{\sigma}H$-linear, we also have
\[
\mathrm{ver}c_{\phi}(b\otimes h\otimes\gamma)=(b\otimes h)\cdot\sum_{j=1}^{n}{x_{j}(\gamma)\mathrm{ver}(\phi_{j})}\overset{\eqref{eq:connection}}{=}(b\otimes h)\cdot(1_{B}\otimes1_{H}\otimes\gamma)=b\otimes h\otimes\gamma,
\]
so that $c_{\phi}$ is a connection.

It remains to prove that we have a bijection. If we start with a connection $c$ with corresponding connection 1-form $\phi_{c}=\sum_{j=1}^{n}{x_{j}\otimes c(1_{B}\otimes1_{H}\otimes x^{j})}$ and then we consider the connection $c_{\phi_{c}}$, we obtain
\[
c_{\phi_{c}}(b\otimes h\otimes\gamma)=\sum_{j=1}^{n}{x_{j}(\gamma)((b\otimes h)\cdot c(1_{B}\otimes1_{H}\otimes x^{j}))}=\sum_{j=1}^{n}{x_{j}(\gamma)c(b\otimes h\otimes x^{j})}=c(b\otimes h\otimes\gamma),
\]
i.e., $c_{\phi_{c}}=c$. On the other hand, if we start with a connection 1-form $\phi=\sum_{j=1}^{n}{x_{j}\otimes\phi_{j}}$ with corresponding connection $c_{\phi}$ and then we consider the connection 1-form $\phi_{c_{\phi}}=\sum_{j=1}^{n}{x_{j}\otimes c_{\phi}(1_{B}\otimes1_{H}\otimes x^{j})}$, we obtain $\phi_{c_{\phi}}=\phi$, since $c_{\phi}(1_{B}\otimes1_{H}\otimes x^{j})=\sum_{i=1}^{n}{x_{i}(x^{j})\phi_{i}}=\phi_{j}$ and thus the thesis follows.
\end{proof}

%\begin{remark}
Note that connection 1-forms $\phi$ are in bijection with right $H$-comodule maps $\omega:{}^{\mathrm{co}H}\Omega^1(H)\to\Omega^1(B\#_{\sigma}H)$ which satisfy $\mathrm{ver}\circ\omega=1\otimes\mathrm{Id}$. The bijection assigns a connection 1-form $\phi=\sum_{j=1}^{n}{x_{j}\otimes\phi_{j}}$ to $\omega:\gamma\mapsto\sum_{j=1}^{n}{x_{j}(\gamma)\phi_{j}}$ and a map $\omega$ of the previous type to the connection 1-form $\sum_{j=1}^{n}{x_{j}\otimes\omega(x^{j})}$. Thus, Proposition \ref{prop:bijection} can also be recovered from \cite[Proposition 5.41]{BM}.
%\end{remark}

\section{Construction for pointed Hopf algebras}\label{Sec4}

We apply the construction of the FODC on $B\#_{\sigma}H$ developed in Theorem \ref{canFODC} to build examples of FODCi on pointed Hopf algebras. \medskip

\noindent\textbf{Pointed Hopf algebras as crossed products}. Recall that a coalgebra $C$ is called \textit{simple} if it is not 0 and it has no non-zero proper subcoalgebras and that $C$ is called \textit{pointed} if all its simple subcoalgebras are 1-dimensional. A bialgebra (Hopf algebra) is pointed if it is pointed as a coalgebra. Since the set of group-like elements is in 1-1 correspondence with the set of 1-dimensional subcoalgebras via $g\mapsto\Bbbk g$ (see e.g. \cite[Lemma 8.0.1 e)]{Sw}), $C$ is pointed if and only if the coradical $C_{0}$ of $C$, i.e., the subcoalgebra of $C$ given by the sum of all its simple subcoalgebras, is the group algebra $\Bbbk[G(C)]$. Among the classical examples of pointed Hopf algebras are group algebras, universal enveloping algebras and the $q$-deformations of the universal enveloping algebras of semisimple Lie algebras. It is known that pointed Hopf algebras can be seen as crossed products as we recall in the following.

\begin{proposition}[{\cite[Theorem 15.2.4]{Ra}}]\label{isopointed}
Every pointed Hopf algebra $H$ is isomorphic to $H_{1}\#_{\sigma}\Bbbk[G/N]$, where $H_{1}$ is the indecomposable component of $H$ containing $\Bbbk1_{H}$ and $\Bbbk[G/N]$ is the group Hopf algebra over $G/N$, with $G$ and $N$ the groups of group-like elements of $H$ and $H_{1}$, respectively.
\end{proposition}
Let us recall some details of the previous result. Given $\{g_{i}\}_{i\in\mathcal{I}}$ a complete set of representatives for the left cosets of $N$ in $G$, we have that $H=\bigoplus_{i\in\mathcal{I}}{H_{g_{i}}}=\bigoplus_{i\in\mathcal{I}}{H_{1}g_{i}}$, see \cite[Proposition 15.2.3]{Ra}. Let $\bar{g_{i}}=g_{i}N$ for all $i\in\mathcal{I}$. The measure $\cdot:\Bbbk[G/N]\otimes H_{1}\to H_{1}$ is given by $\bar{g_{i}}\cdot a:=g_{i}ag_{i}^{-1}$, while the 2-cocycle $\sigma:\Bbbk[G/N]\otimes\Bbbk[G/N]\to H_{1}$ is defined in the following way. Given $i,j\in\mathcal{I}$, then $\bar{g_{i}}\bar{g_{j}}=\bar{g_{l}}$ for a unique $l\in\mathcal{I}$. Hence, there is a unique $n_{i,j}\in N\subseteq H_{1}$ such that $g_{i}g_{j}=n_{i,j}g_{l}$. One then defines $\sigma$ by $\sigma(\bar{g_{i}}\otimes\bar{g_{j}}):=n_{i,j}$. Moreover, the isomorphism of Hopf algebras $f:H\to H_{1}\#_{\sigma}\Bbbk[G/N]$ is given by $f(ag_{i})=a\otimes \bar{g_{i}}$ for all $a\in H_{1}$ and $i\in\mathcal{I}$, where $H_{1}\#_{\sigma}\Bbbk[G/N]$ has the tensor product coalgebra structure.\\

\noindent\textbf{Differential calculus on pointed Hopf algebras.}
Let $H\cong H_1\#_\sigma\Bbbk[G/N]$ be a pointed Hopf algebra with notation as before.
In the following we build
%, for some explicit examples of pointed Hopf algebras $H$, the corresponding 
a right $\Bbbk[G/N]$-covariant calculus on $H$. 
Starting from a $\sigma$-twisted $\Bbbk[G/N]$-module FODC on $H_{1}$ and a bicovariant FODC on $\Bbbk[G/N]$, we then use the construction of Theorem \ref{canFODC}. The bicovariant FODCi on $\Bbbk[G/N]$ are described by the classification theorem of Woronowicz (\cite[Theorem 1.8]{Wo}). In particular, we know that every bicovariant FODC $(\Omega^1(\Bbbk[G/N]),\mathrm{d}_{\Bbbk[G/N]})$ on $\Bbbk[G/N]$ is such that $\Omega^1(\Bbbk[G/N]):=\Bbbk[G/N]\otimes(\Bbbk[G/N]^{+}/I)$, where $I\subseteq\Bbbk[G/N]^{+}$ is a right ideal of $\Bbbk[G/N]$ and $\mathrm{d}_{\Bbbk[G/N]}\bar{g_{i}}=(\mathrm{Id}\otimes\pi)(\bar{g_{i}}\otimes\bar{g_{i}}-\bar{g_{i}}\otimes\bar{1})=\bar{g_{i}}\otimes[\bar{g_{i}}-\bar{1}]$, where $\pi:\Bbbk[G/N]^{+}\to\Bbbk[G/N]^{+}/I$, $\alpha\mapsto[\alpha]$ is the quotient map. Note that the condition $\mathrm{Ad}_{R}(I)\subseteq I\otimes\Bbbk[G/N]$ is automatically satisfied. Observe also that $\Bbbk[G/N]^{+}$ consists exactly of those elements $\sum_{i\in\mathcal{I}}{k_{i}\bar{g_{i}}}$ in $\Bbbk[G/N]$ such that $\sum_{i\in\mathcal{I}}{k_{i}}=0$, that we denote by $k_{i}\bar{g_{i}}$, summation understood. Left and right $\Bbbk[G/N]$-actions are defined as
\[
\bar{g_{l}}\cdot(\bar{g_{j}}\otimes[k_{i}\bar{g_{i}}]):=\bar{g_{l}}\bar{g_{j}}\otimes[k_{i}\bar{g_{i}}]\qquad \mathrm{and}\qquad (\bar{g_{j}}\otimes[k_{i}\bar{g_{i}}])\cdot\bar{g_{l}}:=\bar{g_{j}}\bar{g_{l}}\otimes[k_{i}\bar{g_{i}}\bar{g_{l}}],
\]
while left and right $\Bbbk[G/N]$-coactions are given by
\[
\lambda(\bar{g_{j}}\otimes[k_{i}\bar{g_{i}}]):=\bar{g_{j}}\otimes\bar{g_{j}}\otimes[k_{i}\bar{g_{i}}]\qquad \mathrm{and}\qquad \rho(\bar{g_{j}}\otimes[k_{i}\bar{g_{i}}]):=\bar{g_{j}}\otimes[k_{i}\bar{g_{i}}]\otimes\bar{g_{j}}.
\]
We also need a $\sigma$-twisted $\Bbbk[G/N]$-module FODC on $H_{1}$, i.e., a FODC $(\Omega^1(H_{1}),\mathrm{d}_{H_{1}})$ such that there exists a linear map $\cdot:\Bbbk[G/N]\otimes\Omega^1(H_{1})\to\Omega^1(H_{1})$ which satisfies \eqref{comp.}, \eqref{H-lin} and \eqref{dsigma} of Lemma \ref{twistedHmodcal}. In particular, from \eqref{comp.} and \eqref{H-lin} we know that this linear map, if it exists, is uniquely determined by
\begin{equation}\label{act}
\bar{g_{i}}\cdot(a\mathrm{d}_{H_{1}}a')=(\bar{g_{i}}\cdot a)\mathrm{d}_{H_{1}}(\bar{g_{i}}\cdot a')=g_{i}ag_{i}^{-1}\mathrm{d}_{H_{1}}(g_{i}a'g_{i}^{-1})
\end{equation}
for every $a,a'\in H_{1}$ and $\bar{g_{i}}\in\Bbbk[G/N]$. Thus, one has to prove that \eqref{act} is well-defined. Furthermore, we need $\mathrm{d}_{H_{1}}(n_{i,j})=0$ for every $n_{i,j}\in N$ and $i,j\in\mathcal{I}$ in order to satisfy \eqref{dsigma}. 
In particular, \eqref{dsigma} holds true if $\mathrm{d}_{H_{1}}|_{N}=0$. 
We summarize the previous discussion in the following statement.
\begin{proposition}\label{prop:pointcalc}
Given a pointed Hopf algebra $H\cong H_1\#_\sigma\Bbbk[G/N]$, a right ideal $I\subseteq\Bbbk[G/N]^+$ and a $\sigma$-twisted $\Bbbk[G/N]$-module FODC $(\Omega^1(H_1),\mathrm{d}_{H_1})$ on $H_{1}$ we obtain a right $\Bbbk[G/N]$-covariant FODC $(\Omega^1(H),\mathrm{d}_H)$ on $H$.
\end{proposition}
Observe that the above calculus on $H$ is not right $H$-covariant in general and thus it does not fit the classification theorem of Woronowicz.

Let us briefly analyze some special cases of the construction of Proposition~\ref{prop:pointcalc}, considering cocommutative pointed Hopf algebras, for which the 2-cocycle is trivial.

\begin{remark}

\begin{enumerate}
    
\item[i.)] If $H$ is a cocommutative pointed Hopf algebra then $N=\{1\}$ (see \cite[Proposition 15.3.1]{Ra}) and so $G/N=G$. Hence, the measure becomes $g\cdot a=gag^{-1}$ for all $g\in G$, $a\in H_{1}$ and the 2-cocyle is such that $\sigma(g\otimes g')=1_{H}$ for all $g,g'\in G$. As a consequence, $H$ is isomorphic to the smash product $H_{1}\#\Bbbk G$ as it is said in \cite[Theorem 15.3.2]{Ra}. In this case, given a right ideal $I\subseteq\Bbbk[G]^{+}$ and a $\Bbbk[G]$-module FODC on $H_{1}$, by Proposition \ref{prop:pointcalc} we obtain a right $\Bbbk[G]$-covariant FODC on $H$.
    
\item[ii.)] Two special cases of i.) are the group Hopf algebra $\Bbbk[\tilde{G}]$ for a group $\tilde{G}$ and the universal enveloping algebra $U(\mathfrak{g})$ for a Lie algebra $\mathfrak{g}$; for these two examples, the isomorphism with the smash product is trivial. Indeed, $(\Bbbk[\tilde{G}])_{1}=\Bbbk1_{\tilde{G}}$ since $\Bbbk[\tilde{G}]=\bigoplus_{g\in\tilde{G}}\Bbbk g$ as coalgebras and $G(\Bbbk[\tilde{G}])=\tilde{G}$, while $G(U(\mathfrak{g}))=1$ and $U(\mathfrak{g})_{1}=U(\mathfrak{g})$, since $U(\mathfrak{g})$ is irreducible (see \cite[Corollary 7.6.5]{Ra}) and the irreducible components coincide with the indecomposable ones in the cocommutative case (see \cite[Proposition 15.3.1]{Ra}). In the first case the right $\Bbbk[\tilde{G}]$-covariant calculus on $\Bbbk[\tilde{G}]$ is simply the starting bicovariant FODC on $\Bbbk[\tilde{G}]$, while in the second one the right covariant FODC on $U(\mathfrak{g})$ is simply the starting calculus on $U(\mathfrak{g})$ where $\Omega^1(U(\mathfrak{g}))$ is equipped with trivial coaction. 
\end{enumerate}
\end{remark}

Now we consider a class of examples of pointed Hopf algebras not cocommutative in order to preserve the pure crossed product in the isomorphism and we build the corresponding right covariant FODC explicitly. \\ 

\noindent\textbf{The pointed Hopf algebras $H_{(r,n,q)}$}. Consider the pointed Hopf algebras $H_{(r,n,q)}$ given in \cite[page 477]{Ra} (which were introduced in \cite{Ra2}). Thus, let $r,n\geq 1$ and $M=rn$ and suppose that $\Bbbk$ has a primitive $M^{\mathrm{th}}$ root of unity $q$. Let $H=H_{(r,n,q)}$ be the Hopf algebra over $\Bbbk$, sometimes referred to as \textit{Radford Hopf algebra}, described as follows. As an algebra $H$ is generated by $a$ and $x$ subject to the relations
\[
a^{M}=1,\ \ x^{n}=0,\ \ xa=qax 
\]
and the coalgebra structure of $H$ is given by
\[
\Delta(a)=a\otimes a,\ \ \Delta(x)=1\otimes x+x\otimes a^{r}.
\]
Using the previous notation, $G=(a)$ and $N=(a^{r})$, so that $G/N=\{\overline{a^{i}}\ |\ i=0,\ldots,r-1\}\cong C_{r}$, where $C_{r}$ is the cyclic group with $r$ elements. Hence, $H_{(r,n,q)}$ has $r$ distinct indecomposable components: $H_{1},H_{1}a,\ldots,H_{1}a^{r-1}$. Observe that, considering $r=1$, one obtains the Taft algebra, which is indecomposable and has $G=N=(a)$. We know that $H_{(r,n,q)}$ has a basis given by $\{a^{l}x^{m}\ |\ 0\leq l<M,0\leq m<n\}$, so that $H_{(r,n,q)}=\bigoplus_{0\leq l<M, 0\leq m<n}{\Bbbk a^{l}x^{m}}$ as vector spaces. Thus, it is not difficult to deduce that 
\[
H_{1}=\bigoplus_{0\leq m<n}{\Bbbk x^{m}}\oplus\bigoplus_{0\leq m<n}{\Bbbk a^{r}x^{m}}\oplus\cdot\cdot\cdot\oplus\bigoplus_{0\leq m<n}{\Bbbk a^{r(n-1)}x^{m}}
\]
and we know that $H\cong H_{1}\#_{\sigma}\Bbbk[C_{r}]$ by Proposition \ref{isopointed}. So, an element $\chi\in H_{1}$ is of the form $\chi=\sum_{0\leq m,l<n}{\alpha_{l,m}a^{lr}x^{m}}$, with $\alpha_{l,m}\in\Bbbk$ and then the measure on $H_{1}$ is given, for all $i\in\{0,\ldots,r-1\}$, by 
\[
\overline{a^{i}}\cdot\chi=a^{i}\chi a^{M-i}=\sum_{0\leq m,l<n}{\alpha_{l,m}a^{i}a^{lr}x^{m}a^{M-i}}=\sum_{0\leq m,l<n}{\alpha_{l,m}q^{m(M-i)}a^{i}a^{lr}a^{M-i}x^{m}}=\sum_{0\leq m,l<n}{\alpha_{l,m}q^{m(M-i)}a^{lr}x^{m}},
\]
thus it simply changes the coefficients of $\chi$. Consider a FODC $(\Omega^1(H_{1}),\mathrm{d}_{H_{1}})$ on $H_{1}$ and two elements $a^{lr}x^{m}$ and $a^{kr}x^{s}$ of the basis of $H_{1}$, for indexes $0\leq l,k,m,s<n$. Then, $a^{lr}x^{m}\mathrm{d}_{H_{1}}(a^{kr}x^{s})\in\Omega^1(H_{1})$ and we obtain
\[
\begin{split}
\overline{a^{i}}\cdot(a^{lr}x^{m}\mathrm{d}_{H_{1}}(a^{kr}x^{s}))&=a^{i}a^{lr}x^{m}a^{M-i}\mathrm{d}_{H_{1}}(a^{i}a^{kr}x^{s}a^{M-i})=q^{m(M-i)}a^{lr}x^{m}\mathrm{d}_{H_{1}}(q^{s(M-i)}a^{kr}x^{s})\\&=q^{(M-i)(m+s)}a^{lr}x^{m}\mathrm{d}_{H_{1}}(a^{kr}x^{s}).
\end{split}
\]
This shows that \eqref{act} is well-defined for every FODC on $H_{1}$. Furthermore, note that for $l,k\in\{0,\ldots,r-1\}$ we have $0\leq l+k\leq 2r-2$. Thus, if $0\leq l+k\leq r-1$ we have $a^{l}a^{k}=1a^{l+k}$ for unique $1\in N$, while if $r\leq l+k\leq 2r-2$ we have $a^{l}a^{k}=a^{r}a^{l+k-r}$ with $0\leq l+k-r\leq r-2$, for unique $a^{r}\in N$. 
Thus, $\sigma:\Bbbk[G/N]\otimes\Bbbk[G/N]\to (a^{r})\subseteq H_{1}$ is defined by
\[
\sigma(\overline{a^{l}}\otimes\overline{a^{k}})=\begin{cases}
    1&\text{if}\ \ 0\leq l+k\leq r-1\\
    a^{r}&\text{if}\ \ r\leq l+k\leq 2r-2
\end{cases}
\]
and
this implies that, in order to satisfy \eqref{dsigma}, we have to require $\mathrm{d}_{H_{1}}(a^{r})=0$. So, if the FODC on $H_{1}$ satisfies $\mathrm{d}_{H_{1}}(a^{r})=0$, assuming that $\mathrm{d}_{H_{1}}(a^{kr})=0$ for $k$ an arbitrary positive integer, we obtain
\[
\mathrm{d}_{H_{1}}(a^{(k+1)r})=\mathrm{d}_{H_{1}}(a^{kr}a^{r})=a^{kr}\mathrm{d}_{H_{1}}(a^{r})+\mathrm{d}_{H_{1}}(a^{kr})a^{r}=0
\]
so that, by induction, $\mathrm{d}_{H_{1}}(N)=0$. For a generic element $\chi=\sum_{0\leq m,l<n}{\alpha_{l,m}a^{lr}x^{m}}\in H_{1}$ this gives
\begin{equation}\label{H(r,n,q)FODC}
\mathrm{d}_{H_{1}}(\chi)
=\sum_{0\leq m,l<n}{\alpha_{l,m}a^{lr}\mathrm{d}_{H_{1}}(x^{m})}+\sum_{0\leq m,l<n}{\alpha_{l,m}\mathrm{d}_{H_{1}}(a^{lr})x^{m}}=\sum_{0\leq m,l<n}{\alpha_{l,m}a^{lr}\mathrm{d}_{H_{1}}(x^{m})}.
\end{equation}

Thus, we have obtained the following result:
\begin{proposition}
    Given the pointed Hopf algebra $H_{(r,n,q)}\cong(H_{(r,n,q)})_{1}\#_{\sigma}\Bbbk[C_{r}]$, a right ideal $I\subseteq\Bbbk[C_{r}]^{+}$ and a FODC on $(H_{(r,n,q)})_{1}$ which satisfies \eqref{H(r,n,q)FODC} we obtain a right $\Bbbk[C_{r}]$-covariant FODC on $H_{(r,n,q)}$. 
\end{proposition}

Let us consider a more specific example in order to show all the details of the right covariant FODC.

\begin{example}
Let $r=2$ and $n=2$. The generators $a$ and $x$ are subject to the relations
\[
a^{4}=1,\ x^{2}=0,\ xa=qax,
\]
where $q$ is a primitive fourth root of unity. We have two indecomposable components, i.e., $H_{1}$ and $H_{1}a$, where 
\[
H_{1}=\bigoplus_{0\leq m<2}{\Bbbk x^{m}}\oplus\bigoplus_{0\leq m<2}{\Bbbk a^{2}x^{m}}=\Bbbk1\oplus\Bbbk x\oplus\Bbbk a^{2}\oplus\Bbbk a^{2}x.
\]
Therefore, for $H=H_{(2,2,q)}$ an element $\chi\in H_{1}$ is of the form $\chi=\alpha1+\beta x+\gamma a^{2}+\delta a^{2}x$ for $\alpha,\beta,\gamma,\delta\in\Bbbk$ and $G/N=\{\bar{1},\bar{a}\}\cong C_{2}$. Thus, $\sigma:\Bbbk[C_{2}]\otimes\Bbbk[C_{2}]\to(a^{2})\subseteq H_{1}$ is defined by
\[
\sigma(\overline{a^{l}}\otimes\overline{a^{k}})=\begin{cases}
    1&\text{if}\ \ l=k=0,\ \ l=0,k=1,\ \ l=1,k=0\\
    a^{2}&\text{if}\ \ l=k=1
\end{cases}
\]
and
the $\sigma$-twisted $\Bbbk[C_{2}]$-module FODC on $H_{1}$ is determined by $\mathrm{d}_{H_{1}}x$, so that 
\[
\mathrm{d}_{H_{1}}(\chi)=\alpha\mathrm{d}_{H_{1}}1+\beta\mathrm{d_{H_{1}}}x+\gamma\mathrm{d}_{H_{1}}(a^{2})+\delta\mathrm{d}_{H_{1}}(a^{2}x)=\beta\mathrm{d}_{H_{1}}x+\delta a^{2}\mathrm{d}_{H_{1}}x.
\]
The only right ideals $I$ of $\Bbbk[C_{2}]$ which are contained in $\Bbbk[C_{2}]^{+}=\Bbbk(\bar{a}-\bar{1})$ are $I=\{0\}$ and $I=\Bbbk(\bar{a}-\bar{1})$.
Since the latter choice gives a trivial quotient we elaborate on the first one, which corresponds to the universal FODC.
For $I=\{0\}$ we have that $\Omega^1(\Bbbk[C_{2}]):=\Bbbk[ C_{2}]\otimes\Bbbk(\bar{a}-\bar{1})$ and $\mathrm{d}_{\Bbbk[ C_{2}]}\bar{a}=%\bar{a}\otimes\bar{a}-\bar{a}\otimes\bar{1}=
\bar{a}\otimes(\bar{a}-\bar{1})$. Thus, $\mathrm{d}_{\Bbbk[C_{2}]}(e\bar{1}+f\bar{a})=f\bar{a}\otimes(\bar{a}-\bar{1})$, for $e,f\in\Bbbk$ and the right $\Bbbk[C_{2}]$-covariant FODC $(\Omega^1(H_{1}\#_{\sigma}\Bbbk[C_{2}]),\mathrm{d}_{\#_{\sigma}})$ on $H_{1}\#_{\sigma}\Bbbk[C_{2}]$, where $\Omega^1(H_{1}\#_{\sigma}\Bbbk[C_{2}])=(\Omega^1(H_{1})\otimes\Bbbk[C_{2}])\oplus(H_{1}\otimes\Bbbk[C_{2}]\otimes\Bbbk(\bar{a}-\bar{1}))$, is given by
\[
\begin{split}
\mathrm{d}_{\#_{\sigma}}(\chi\otimes(e\bar{1}+f\bar{a}))&=\mathrm{d}_{H_{1}}(\chi)\otimes(e\bar{1}+f\bar{a})+\chi\otimes\mathrm{d}_{\Bbbk[C_{2}]}(e\bar{1}+f\bar{a})\\&=(\beta\mathrm{d}_{H_{1}}x+\delta a^{2}\mathrm{d}_{H_{1}}x)\otimes(e\bar{1}+f\bar{a})+(\alpha1+\beta x+\gamma a^{2}+\delta a^{2}x)\otimes f\bar{a}\otimes(\bar{a}-\bar{1})\\&=\beta e\mathrm{d}_{H_{1}}x\otimes\bar{1}+\beta f\mathrm{d}_{H_{1}}x\otimes\bar{a}+\delta ea^{2}\mathrm{d}_{H_{1}}x\otimes\bar{1}+\delta fa^{2}\mathrm{d}_{H_{1}}x\otimes\bar{a}\\&+\alpha f1\otimes\bar{a}\otimes(\bar{a}-\bar{1})+\beta fx\otimes\bar{a}\otimes(\bar{a}-\bar{1})+\gamma fa^{2}\otimes\bar{a}\otimes(\bar{a}-\bar{1})+\delta fa^{2}x\otimes\bar{a}\otimes(\bar{a}-\bar{1})
\end{split}
\]
for all $\chi\otimes(e\bar{1}+f\bar{a})\in H_{1}\#_{\sigma}\Bbbk[C_{2}]$.
\end{example}

\subsection*{Acknowledgements}

We would like to thank P. Aschieri for valuable suggestions throughout all phases of this work. In particular, he conjectured the results of Section~\ref{Sec:ClasSmash}. We would also like to thank A. Ardizzoni and P.M. Hajac for helpful comments and the referee for meaningful suggestions. This paper was written while the first author was member of the “National Group for Algebraic and Geometric Structures and their Applications” (GNSAGA-INdAM).  He was partially supported by MIUR within the National Research Project PRIN 2017. The second author was partially supported by HORIZON-MSCA-2021-SE-01-101086123 CaLIGOLA and by INFN Sezione di Bologna (Gast).

\end{document}